\documentclass[12pt,leqno]{article}

\usepackage{latexsym,amsmath,amscd,amssymb,framed,graphics,graphicx,mathrsfs,enumerate,color}

\usepackage[square,authoryear]{natbib}
\usepackage[colorlinks]{hyperref}
\usepackage{url}

\usepackage[all]{xy}

\makeatletter

\newcommand{\rem}[1]{}

\newcommand \al{\alpha}
\newcommand\be{\beta}
\newcommand\ga{\gamma}
\newcommand\de{\delta}

\newcommand\ze{\zeta}

\renewcommand\th{\theta}

\newcommand\ka{\kappa}
\newcommand\la{\lambda}

\newcommand\si{\sigma}

\newcommand\ph{\varphi}

\newcommand\ps{\psi}
\newcommand\om{\omega}
\newcommand\Ga{\Gamma}

\newcommand\Th{\Theta}

\newcommand\Ph{\Phi}
\newcommand\Ps{\Psi}
\newcommand\Om{\Omega}

\newcommand\ie{i.e.\ }

\renewcommand\o{\circ}
\renewcommand\div{\on{div}}
\newcommand\x{\times}
\newcommand\on{\operatorname}
\newcommand\Ad{\on{Ad}}
\newcommand\ad{\on{ad}}\newcommand\id{\on{id}}

\newcommand\Emb{\on{Emb}}
\newcommand\Hom{\on{Hom}}
\newcommand\Den{\on{Den}}
\newcommand\Aut{\mathcal{A}ut}
\newcommand\F{\mathcal{F}}
\newcommand\Diff{\on{Diff}}
\newcommand\hor{\on{hor}}
\newcommand\vol{\on{vol}}

\newcommand\ham{\on{ham}}

\newcommand\g{\mathfrak g}

\newcommand\ou{\mathfrak o}
\newcommand\dd{{\bf d}}

\newcommand\RR{\mathbb{R}}

\newcommand\J{\mathbf{J}}
\newcommand\X{\mathfrak X}

\renewcommand\O{\mathcal O}
\newcommand\A{\mathcal A}
\newenvironment{proof}[1][Proof]{\noindent\textbf{#1.} }{\ \rule{0.5em}{0.5em}}

\begin{document}

\newtheorem{theorem}{Theorem}[section]
\newtheorem{definition}[theorem]{Definition}
\newtheorem{lemma}[theorem]{Lemma}
\newtheorem{remark}[theorem]{Remark}
\newtheorem{proposition}[theorem]{Proposition}
\newtheorem{corollary}[theorem]{Corollary}
\newtheorem{example}[theorem]{Example}

%\def\below#1#2{\mathrel{\mathop{#1}\limits_{#2}}}

%%%%%%%%%%%%%%%%%%%%%%%%%%%%%%%%%%%%%%%%%%%%%%%%%%%%%%%%%%%%%%%%%%%%%%%%%%%%%%%
%%%%%%%%

\title{Geometric dynamics on\\the automorphism group of principal bundles:\\geodesic flows, dual pairs and chromomorphism groups}
\author{ Fran\c{c}ois Gay-Balmaz$^{1}$, Cesare Tronci$^{2}$, and Cornelia Vizman$^{3}$ }

\addtocounter{footnote}{1}
\footnotetext{Laboratoire de
M\'et\'eorologie Dynamique, \'Ecole Normale Sup\'erieure/CNRS, F--75231 Paris, France.
\texttt{francois.gay-balmaz@lmd.ens.fr}
\addtocounter{footnote}{1}}
\footnotetext{Section de
Math\'ematiques, \'Ecole Polytechnique F\'ed\'erale de Lausanne,
CH--1015 Lausanne, Switzerland, and Department of Mathematics, University of Surrey, GU2 7XH Guildford, United Kingdom.  \texttt{cesare.tronci@epfl.ch}
\addtocounter{footnote}{1} }
\footnotetext{Department of Mathematics,
West University of Timi\c soara. RO--300223 Timi\c soara. Romania.
\texttt{vizman@math.uvt.ro}
\addtocounter{footnote}{1} }

\date{ }
\maketitle

\makeatother
%\begin{center} DRAFT \end{center}
\maketitle

%|||-------------------text width----------------------|||

%\noindent \textbf{AMS Classification:}

%\noindent \textbf{Keywords:}

\begin{abstract}

We formulate Euler-Poincar\'e equations on the Lie group
$\Aut(P)$ of automorphisms of a principal bundle $P$. The corresponding flows are referred to as EP$\!\Aut$ flows. We mainly focus on geodesic flows associated to Lagrangians of Kaluza-Klein type.  In the special case of a trivial bundle $P$, we identify geodesics on certain infinite-dimensional semidirect-product Lie groups that emerge naturally from the construction.
This approach leads naturally to a dual pair structure containing $\delta\text{-like}$ momentum map solutions  that extend previous results on geodesic flows on the diffeomorphism group (EPDiff).
In the second part, we consider incompressible flows 
on the Lie group $\Aut_{\rm vol}(P)$ of volume-preserving bundle
automorphisms. In this context, the dual pair construction requires the definition of chromomorphism groups, i.e. suitable Lie group extensions generalizing the quantomorphism group.
\end{abstract}

\tableofcontents

%-------------------------------------------------------
%--------------------------------------------------------

\section{Introduction}

Many physical systems are described by geodesic dynamics on their
Lie group configuration space. Besides finite dimensional cases,
such as the rigid body, one of the most celebrated continuum systems
is Euler's equation for ideal incompressible fluids \cite{Ar66}, followed by
 magnetohydrodynamics in plasma physics \cite{Ha94}.
These geodesic fluid flows may allow for $\delta$-like measure
valued solutions, whose most famous example is the vortex filament
dynamics for Euler's vorticity equation. Another example is provided
by the magnetic vortex line solutions of MHD \cite{Andrade2006}, which are limiting
cases of the magnetic flux tubes observed, for example, on the solar
surface.

Then, the existence of vortex solutions in different physical
contexts leads to the natural question whether these singular
solutions may emerge spontaneously in general continuum
descriptions. For example, Euler's equation is not supposed to
generate vortex solutions, which are usually considered as invariant
solution manifolds. On the other hand, the singularity formation
phenomenon (steepening) is one of the special features of certain shallow-water
models such as the integrable Camassa-Holm equation (CH)
\cite{CaHo1993} and its higher dimensional version ($n$-CH)
\cite{HoMaRa98}, describing geodesics on the diffeomorphism group
with respect to $H^1$ metrics. Its extension to more general metrics
is known as EPDiff \cite{HoMa2004} and found applications in
different contexts, from turbulence \cite{FoHoTi01} to imaging
science \cite{HoRaTrYo04}. 
\rem{ %%%%%%%%%%%%%%%%%%%%%%%%%%%
In one dimension, the spontaneous
emergence of singular peakon solutions of the CH equation is
rigorously proved by the ``steepening lemma'' \cite{CaHo1993}, while
in higher dimensions this phenomenon has been observed in numerical
simulations \cite{HoPuSt04}. 
}  %%%%%%%%%%%%%%%%%%%%%%%%%%%
Remarkably, the singular solutions of
these systems have been shown to be momentum maps determining
collective dynamics on the space of embeddings. This momentum map
turns out to be one leg of a \textit{dual pair} associated to
commuting lifted actions on the cotangent bundle of embeddings,
\cite{HoMa2004}.

Another remarkable property of the Camassa-Holm and its
$n$-dimensional generalization on manifolds with boundary is the
smoothness of its geodesic spray in Lagrangian representation
(\cite{HoMa2004}, \cite{GB2009}) which proves the local
well-posedness of the equation in Sobolev spaces. Such a property
was first discovered in \cite{EbMa1970} for ideal incompressible
fluids and later extended to the averaged Euler equations with
Dirichlet or Navier slip boundary conditions (\cite{MaRaSh2000},
\cite{Sh2000}, \cite{GBRa2005}).

The integrability properties of the CH equation were extended to a
two component system (CH2) \cite{ChLiZh2005,Ku2007}. This system was
interpreted in terms of shallow water dynamics by Constantin and
Ivanov \cite{CoIv08}, who also proved the steepening mechanism. The
investigation of the geometric features of the CH2 system allowed to
extend this system to any number of dimension and to anisotropic
interactions \cite{HoTr2008}. Although possessing a steepening
mechanism (singularity formation), the CH2 equations do not allow for $\delta$-like solution
in both variables thereby destroying one of the peculiar features
possessed by the CH equation. However, slight modifications in the
CH2 Hamiltonian introduce the $\delta$-function solutions for a
modified system (MCH2). These solutions were shown to  emerge
spontaneously in \cite{HoOnTr2009}, where the corresponding steepening
mechanism was presented. 
\rem{ %%%%%%%%%%%%%%%%%%%%%%%%%%
Moreover, besides the shallow water
interpretation of CH2 dynamics \cite{CoIv08}, the MCH2 equation
extends EPDiff to comprise space charge effects for charged fluids.
} %%%%%%%%%%%%%%%%%%%%%%%%%%
Recently, the MCH2 system has been shown to be in strict relation with the two-dimensional EPDiff equation \cite{HoIv2011}.

All the systems mentioned  above  determine specific geodesic flows
on certain infinite dimensional Lie groups: CH and $n$-CH are
\textit{geodesics on the whole diffeomorphism group}, while CH2 and
MCH2, as well as their $n$ dimensional generalization or their
extension to anisotropic interactions, are \textit{geodesics on a
semidirect product involving diffeomorphisms and Lie group-valued
functions\/} (i.e. the gauge group). Therefore one is led to investigate further the
properties of geodesic flows on diffeomorphism groups (cf. e.g.
\cite{Vi08}), with particular attention towards the singular
solution dynamics and its underlying geometry. Some of the momentum map properties of MCH2-type systems were found in \cite{HoTr2008}. These suggest that more fundamental features are attached to systems on infinite-dimensional semidirect-products. This stands as the main motivation for this paper, which focuses on the investigation of the deep geometric features underlying MCH2-type systems, in terms of geodesic flows on the Lie group of bundle automorphisms. In physics, the Lie group $\Aut(P)$ of automorphisms of a principal bundle $P$ is the diffeomorphism group underling Yang-Mills charged fluids \cite{HoKu88}, \cite{GBRa2008a} and this provides a physical interpretation of the geodesic flows considered here. For example, MCH2 corresponds to an Abelian Yang-Mills charged fluid, that is a fluid possessing an electric charge. When the principal connection on $P$ is used to write the explicit form of the equations, this leads to the interpretation of a charged fluid in an external magnetic field, whose vector potential is given precisely by the connection one-form.

\medskip

\subsection{Goal of the paper} In this paper, we investigate
geodesics on the automorphism group of a principal bundle, thereby
formulating what we shall call the EP$\mathcal{A}ut$ equations
(\emph{Euler-Poincar\'e equations on the automorphisms}), which generalize the systems mentioned in the previous section. In particular, the principal bundle setting
allows  to comprise all the momentum map properties of 
MCH2-type systems \cite{HoTr2008} in a single dual pair diagram. One leg of this
dual pair characterizes the singular solution dynamics of
EP$\mathcal{A}ut$, whereas the other leg naturally yields Noether's
conserved quantities. In addition, the bundle setting simplifies the
group actions which are involved in the reduction process, so that
these actions are now simply the cotangent lifts of the left- and right-
composition by the automorphism group.

In a more physical context, the present work is also motivated by
the introduction of a vector potential $\mathbf{A}$ (associated to the principal connection, denoted by
$\mathcal{A}$) in MCH2-type dynamics, which arises naturally from our principal
bundle approach. This quantity is intrinsic to any principal bundle: for example,  in the case of a non-trivial bundle, a
connection becomes necessary in order to write the equations explicitly. A similar approach was used in \cite{GBRa2008a} to
present the reduced Lagrangian and Hamiltonian formulations of ideal
fluids interacting with Yang-Mills fields, \cite{GiHoKu1983}. The
introduction of non-trivial principal bundles in physics is motivated by their
frequent appearance in common gauge theories \cite{Bleeker}. 
\rem{ %%%%%%%%%%%%%%%%%%%%%%%%%%%
Note
that the introduction of the magnetic vector potential $\mathbf{A}$
generalizes MCH2 dynamics to consider space charge effects in a high
stationary magnetic field $\mathbf{B}=\nabla\times\mathbf{A}$ (arising from the associated curvature two-form, denoted by $\mathcal{B}=\mathbf{d}^{\mathcal{A}} \mathcal{A}$). These space-charge effects are
often observed in laboratory experiments, for
example involving incompressible flows of non-neutral plasmas. In this case, the realization of vortex $\delta$-like structures in
the plasma density is a well known experimental result
\cite{DuON1999}. 
}   %%%%%%%%%%%%%%%%%%%%%%%%%%%

Motivated by the dynamics of incompressible Yang-Mills fluids, we also
consider volume-preserving fluid flows. As a result, we obtain a
Yang-Mills version of Euler's equation on the volume-preserving
automorphism group $\Aut_{\rm vol}$. In this situation, an
incompressible fluid carries a Yang-Mills charge under the influence
of an external Yang-Mills magnetic potential (the connection
one-form). The incompressibility property significantly affects the
properties of the system. Then, the resulting equations may
involve the fluid vorticity ${\boldsymbol\omega}={\rm curl}\,{\bf u}$ and the
natural question arises whether the geometric properties of Euler's
equation (cf. \cite{MaWe83}) can be extended to the present case.
In this context, the momentum map properties deserve particular
care, since they require the use of appropriate extensions of
infinite dimensional Lie groups. The geometry involved is extremely
rich and highly non trivial, yielding relations with the group of
quantomorphisms, whose extension to the Yang-Mills
setting (the chromomorphism group) becomes necessary. In the last part of the paper we shall
present explicit formulas for the left and right momentum maps
associated to the EP$\Aut_{\rm vol}$ system.

%%%%%%%%%%%%%%%%%%%%%%%%%%%%%%%%%%%%%

\subsection{Summary of the main results} 

This paper formulates the general Euler-Poincar\'e equations on
the $\Aut(P)$ group of automorphisms on a principal $\O-$bundle $P$.
In physical terms this construction generalizes ordinary fluid
flows, such as Burgers or Camassa-Holm equations, to account for the
transport of some Yang-Mills charge under the influence of a
magnetic potential, given by the connection one-form on $P$. 

Besides
the explicit formulation of the equations, a first result is obtained in Section \ref{section2} for the case of a trivial bundle $P=M\times\O$, and
it regards the construction of a dual pair that extends the results
presented in \cite{HoTr2008}. Section \ref{Sec:DualPairTrivial} investigates the geometric properties
of the dual pair structure in terms of Clebsch variables, singular solutions and
their associated conservation laws. Remarkably, these properties are
all embodied in the dual pair structure, which is the unique feature
of this construction.

In Section \ref{Sec:NonTrivialEPAut}, the more general case of a non-trivial principal bundle is considered and the corresponding 
equations are presented in different fashions. The main result
regards again the dual pair structure. In particular, the Clebsch representation of the fluid momentum is found to involve its own magnetic field, which in turn affects the singular solution dynamics; see Section \ref{sec:momaps_and_singsol}. 

Motivated by the geometric dynamics of incompressible Yang-Mills fluids, Section \ref{Sec:EPAut_vol} investigates
the incompressible version of EP$\Aut$, thereby studying geodesic
flows on the bundle automorphisms that project down to
volume-preserving diffeomorphisms on the base. This immediately
extends Euler's equation for ideal fluids to account for the
transport of a Yang-Mills charge. 

After presenting the
explicit form of the equations in different cases, Section \ref{MomapsForEPAut_vol} focuses
on the geometry underlying this incompressible flow, thereby obtaining a new
Clebsch representation of the fluid vorticity, that also encodes the
presence of the advected Yang-Mills charge. Finally, Section \ref{Sec:EPAut_volForYM} compares both left and
right momentum maps with the corresponding
expressions applying to the case of Euler's vorticity equation
\cite{MaWe83}.
This construction requires the use of Lie group extensions. In
particular, the left leg momentum map requires (see Section \ref{Sec:Aut_ham-bar}) the definition of
new Lie groups, the chromomorphism groups, that extend the quantomorphism group 
to comprise  Yang-Mills particle systems.

%%%%%%%%%%%%%%%%%%%%%%%%%%%%%%%%%%%%%%%%%%%

\section{EP$\!\Aut$ flows on a trivial principal bundle\label{section2}}

This section presents the EP$\Aut$ system of equations in the special case of a trivial principal bundle. The discussion below recalls background definitions of the automorphism group of principal bundles and specializes them to the case of a trivial bundle.

\subsection{Generalities on bundle automorphisms}
Before giving the EP$\!\mathcal{A}ut$ equation we recall some facts
concerning the Euler-Poincar\'e and Lie-Poisson reductions. Consider
a right principal bundle $\pi: P\rightarrow M$, with structure group
$\mathcal{O}$. We denote by
\[
\Phi:\mathcal{O}\times P\rightarrow P\,,\; (g,p)\mapsto \Phi_g(p)
\]
the associated right action. The structure group $\mathcal{O}$ is sometimes called \emph{order parameter group} in condensed matter physics. The
automorphism group $\mathcal{A}ut(P)$ of $P$
consists of all $\mathcal{O}$-equivariant
diffeomorphisms of $P$, that is, we have
\[
\mathcal{A}ut(P)=\{\varphi\in\operatorname{Diff}(P)\mid \varphi\circ\Phi_g=\Phi_g\circ\varphi,\;\;\text{for all $g\in \mathcal{O}$}\},
\]
where $\operatorname{Diff}(P)$ denotes the group of all diffeomorphisms of $P$. An automorphism $\varphi$ induces a diffeomorphism $\bar\varphi$ of the base $M$, by the condition $\pi\circ\varphi=\bar\varphi\circ\pi$. The Lie algebra $\mathfrak{aut}(P)$ of the automorphism group consists of $\mathcal{O}$-equivariant vector fields on $P$.

When the bundle $P$ is trivial we can write $P=M\times \mathcal{O}$. In this case the group of all
automorphisms of $P$ is isomorphic to the semidirect product group
\[
\mathcal{A}ut(P)\simeq\operatorname{Diff}(M)\,\circledS\,\mathcal{F}(M,\mathcal{O}),
\]
where $\mathcal{F}(M,\mathcal{O})$ denotes the gauge group of $\mathcal{O}$-valued functions defined on $M$, and $\operatorname{Diff}(M)$ acts on $\mathcal{F}(M,\mathcal{O})$
by composition on the right.  To a couple $(\eta,
\chi)$ in the semidirect product, is associated the automorphism
\[
(x,g)\in M\times\mathcal{O}\mapsto (\eta(x),\chi(x)g)\in M\times\mathcal{O}.
\]
In the trivial case, the Lie algebra
of the automorphism group is the semidirect product Lie algebra
$\mathfrak{X}(M)\,\circledS\,\mathcal{F}(M,\mathfrak{o})$,
where $\mathfrak{X}(M)$ denotes the Lie algebra of vector fields on $M$, while $\mathfrak{o}$ denotes the Lie algebra of the structure group
$\mathcal{O}$, so that $\mathcal{F}(M,\mathfrak{o})$ is the Lie algebra of $\mathfrak{o}$-valued functions defined on $M$.

\subsection{The EP$\mathcal{A}ut$ system on a trivial principal bundle}

When $P$ is a trivial principal bundle, the Euler-Poincar\'e equations
on the automorphism group $\mathcal{A}ut(P)$ are given in the following proposition. (See e.g. \cite{GBRa2009} and \cite{Ho2002} for its proof in the context of complex fluid dynamics).

\begin{proposition}[The $\operatorname{EP}\mathcal{A}ut$ equations on a trivial principal bundle] \label{EPaut_trivial}
Let $L$ be an invariant Lagrangian $L:T\Aut(M\times\mathcal{O})\to\Bbb{R}$ , so that the isomorphism $\Aut(M\times\mathcal{O})\simeq\operatorname{Diff}(M)\,\circledS\,\mathcal{F}(M,\mathcal{O})$ yields $L(\varphi,\dot\varphi)=L(\eta,\dot\eta,\chi,\dot\chi)$. 
Then, upon introducing the reduced variables 
\[
(\mathbf{u},{\nu})=(\dot\eta\circ\eta^{-1},(\dot\chi\chi^{-1})\circ\eta^{-1})\in\mathfrak{X}(M)\,\circledS\,\mathcal{F}(M,\mathfrak{o}),
\]
the Eu\-ler-Poincar\'e equations on $\Aut(M\times\mathcal{O})\simeq\operatorname{Diff}(M)\,\circledS\,\mathcal{F}(M,\mathcal{O})$ that arise from the symmetry-reduced Hamilton's principle
\[
\delta\int_{t_1}^{t_2} \!l(\mathbf{u},\nu)\,{\rm d} t=0,
\]
are written as
\begin{equation}\label{EPAut_trivial}
\left\{\begin{array}{l}
\vspace{0.2cm}\displaystyle\frac{\partial}{\partial t} \frac{\delta l}{\delta\mathbf{u}} +\pounds_{\mathbf{u}}\frac{\delta l}{\delta\mathbf{u}}+\frac{\delta l}{\delta{\nu}}\!\cdot\!\mathbf{d}{\nu}=0\\
\displaystyle\frac{\partial}{\partial t}\frac{\delta l}{\delta{\nu}}+\pounds_{\mathbf{u}}\frac{\delta l}{\delta{\nu}}+\operatorname{ad}^*_{{\nu}}\frac{\delta l}{\delta{\nu}}=0,
\end{array}\right.
\end{equation}
where the operator $\pounds_\mathbf{u}$ denotes the Lie derivative acting on tensor densities.
\end{proposition}
%%%%%%%%%%%%%%%%%%%%%%%%%%%%%%%%%%%%%%%%%%%%%%%%%%%%%%%%%%%%%%%%%%%%%%
One can similarly write the Lie-Poisson equations associated to a
right-invariant Hamiltonian $H$ defined on $T^*\mathcal{A}ut(P)$; see \cite{HoTr2008}.
\rem{ %%%%%%%%%%%%%%%%%%%%%%%%%%%%%%%%%%%%%%%%%%%%%%%%%%%%%%%%%%%%%%%%%%%%%%
They are given by
\begin{equation}
\left\{\begin{array}{l}
\vspace{0.2cm}\displaystyle\partial_t\mathbf{m}+\pounds_{\frac{\delta h}{\delta\mathbf{m}}}\mathbf{m}+C\cdot\mathbf{d}\frac{\delta h}{\delta C}=0\\
\displaystyle\partial_tC+\pounds_{\frac{\delta h}{\delta\mathbf{m}}}C+\operatorname{ad}^*_{\frac{\delta h}{\delta C}}C=0,
\end{array}\right.
\end{equation}
where $h:\mathfrak{X}(M)^*\times\mathcal{F}(M,\mathfrak{o})^*\rightarrow\mathbb{R}$ is the reduced Hamiltonian and
$(\mathbf{m},C)\in\X(M)^*\x\F(M,\ou)^*$.
}        %%%%%%%%%%%%%%%%%%%%%%%%%%%%%%%%%%%%%%%%%%%%%%%%%%%%%%%%%%%%%%%%%%%%%%

\begin{remark}[Duality parings] {\rm In this paper, by the dual of a function space
$\mathcal{F}$ we mean a space $\mathcal{F}^*$ in nondegenerate
duality with $\mathcal{F}$ with respect to an $L^2$ pairing
$\langle\cdot,\cdot\rangle$. We  do not specify the precise
regularity involved, it will be clear that different situations
require different regularity properties, in order to obtain
well-defined expressions. For example, in the formulation of the
Euler-Poincar\'e equations \eqref{EPAut_trivial}, we have chosen the
regular dual, that is
\[
\mathfrak{X}(M)^*\times\mathcal{F}(M,\mathfrak{o})^*=
\left(\Omega^1(M)\otimes\operatorname{Den}(M)\right)\times\big(\mathcal{F}(M,\mathfrak{o}^*)\otimes\operatorname{Den}(M)\big),
\]
whereas, the momentum maps and singular solutions will take values in the topological dual space. }
\end{remark}

\begin{remark}[Pairing notations] \rm
Several notational issues emerge  from the various pairings appearing in this paper. However, angle brackets (e.g. $\langle\mathbf{u},\mathbf{v}\rangle$) and the dot pairing (e.g. $\sigma\cdot\nu$) will be used throughout the text depending only on convenience purposes. We hope that this will not generate confusion.
\end{remark}

In order to write the $\operatorname{EP}\!\mathcal{A}ut$ equation more explicitly, it is convenient to fix a Riemannian metric $\mathsf{g}$ on $M$. Let $\mathbf{v}\in\mathfrak{X}(M)$ and ${c}\in\mathcal{F}(M,\mathfrak{o}^*)$ such that $\delta l/\delta\mathbf{u}=\mathbf{v}^\flat\otimes\mu_M$ and $\delta l/\delta\nu= \sigma\otimes\mu_M$,
where $\mu_M$ is the Riemannian volume. Then equation \eqref{EPAut_trivial} can be equivalently written as
\begin{equation}
\left\{\begin{array}{l}
\vspace{0.2cm}\displaystyle\partial_t\mathbf{v}+\nabla_\mathbf{u}\mathbf{v}+\nabla\mathbf{u}^\mathsf{T}\cdot\mathbf{v}+\mathbf{v}\operatorname{div}(\mathbf{u})+({\sigma}\cdot\mathbf{d}{\nu})^\sharp=0\\
\displaystyle\partial_t{\sigma}+\mathbf{d}{\sigma}\cdot\mathbf{u}+{\sigma}\operatorname{div}(\mathbf{u})+\operatorname{ad}^*_{{\nu}}{\sigma}=0,
\end{array}\right.
\end{equation}
where $\sharp$ denotes the sharp operator associated to the
Riemannian metric $\sf g$. For $\delta l/\delta\nu=0$, this system reduces to an \emph{Euler-Poincar\'e equation on the diffeomorphisms} (EPDiff); see \cite{HoMa2004,HoMaRa98}.

\begin{remark}[Related physical systems]
{\rm Similar structures as in the equations \eqref{EPAut_trivial}
also appear in the geometric characterization of the dynamics of Euler-Yang-Mills fluids (\cite{GBRa2008a} and \cite{Vi08}) and
complex fluids (\cite{Ho2002}).
In the latter case the equations are
obtained by introducing advected quantities that are acted on by the
automorphism group through affine representations
(\cite{GBRa2009}).}
\end{remark}

\subsection{Kaluza-Klein Lagrangians and Hamiltonians}
When $l$ is a quadratic function of $(\mathbf{u},{\nu})$,
then the EP$\mathcal{A}ut$ equation is the Eulerian description of
geodesic motion on the automorphism group. A natural Lagrangian
arising in this context is the Kaluza-Klein Lagrangian given by
\begin{equation}\label{KK_Lagrangian_trivial}
l(\mathbf{u},{\nu})=\frac{1}{2}\|\mathbf{u}\|^2_1+\frac{1}{2}\|{\nu}+A\!\cdot\!\mathbf{u}\|^2_2,
\end{equation}
relative to a given one-form $A\in\Omega^1(M,\mathfrak{o})$ and to inner product norms $\|\,\|_1$ and $\|\,\|_2$ on
$\mathfrak{X}(M)$ and $\mathcal{F}(M,\mathfrak{o})$, respectively.
The inner products are usually given by positive symmetric
differential operators $Q_1:\mathfrak{X}(M)\rightarrow\mathfrak{X}(M)^*$ and $Q_2:\mathcal{F}(M,\mathfrak{o})\rightarrow\mathcal{F}(M,\mathfrak{o})^*$, that is, one defines the
inner products
\[
\langle \mathbf{u},\mathbf{v}\rangle_1:=\int_M\mathbf{u}\cdot Q_1\mathbf{v}\quad\text{and}\quad \langle {\nu},{\om}\rangle_2:=\int_M{\nu}\cdot Q_2 {\om},
\]
When these differential operators are
invertible, the associated Hamiltonian $h$ on $\X(M)^*\x\F(M,\mathfrak{o})^*$
%(\mathbf{m},C)$
may be found by Legendre transforming
\begin{align*}
{\bf m}&=\frac{\delta l}{\delta {\bf u}}=Q_1{\bf u}+Q_2(A\cdot\mathbf{u}+{\nu})\cdot A\in\Omega^1(M)\otimes\operatorname{Den}(M)
\\
{\sigma}&=\frac{\delta l}{\delta {\nu}}=Q_2(A\cdot \mathbf{u}+{\nu})\in\mathcal{F}(M,\mathfrak{o}^*)\otimes\operatorname{Den}(M).
\end{align*}
so that
\begin{equation}\label{EPAut-Ham}
h(\mathbf{m},{\sigma})=\frac{1}{2}\int_M
\left(\mathbf{m}-{\sigma}\!\cdot\!A\right)\cdot Q_1^{-1\!}\left(\mathbf{m}-{\sigma}\!\cdot\!A\right)+\frac{1}{2}\int_M
{\sigma}\cdot Q_2^{-1}{\sigma}.
\end{equation}
For example, if $M=\mathbb{R}^n$ and $G_i$ denotes the Green's
function of $Q_i$, then the EP$\!{\mathcal{A}ut}$ reduced
Hamiltonian
$h:\mathfrak{aut}(P)^*\simeq\mathfrak{X}(\mathbb{R}^n)^*\times\mathcal{F}(\mathbb{R}^n,\mathfrak{o})^*\to\mathbb{R}$
is written as
\begin{align}\label{h_example}
h({\bf m},{\sigma})\ &= \frac12\iint G_1(x-x')\left({\bf m}(x)-{\sigma}(
x)\!\cdot\!A(x)\right)\cdot\left({\bf m}(x')-{\sigma}(x')\!\cdot\!A(x')\right) {\rm d}^nx\ {\rm d}^n x'\nonumber
\\
&\hspace{2cm}+ \frac12\iint G_2(x-x')\, {\sigma}(x)\cdot {\sigma}(x')\ {\rm d}^n x\ {\rm d}^nx',
\end{align}
thereby producing the Lie-Poisson equations
\begin{equation}\label{LPAut_trivial}
\left\{\begin{array}{l}
\vspace{0.2cm}\displaystyle
\partial_t\mathbf{m}+\pounds_{G_1*\left({\bf m}-{\sigma}\cdot A\right)}\ {\bf m}+{\sigma}\!\cdot\! \mathbf{d}\big(G_2*{\sigma}-A\cdot G_1*({\bf m}-{\sigma}\!\cdot\!A)\big)=0\\
\partial_t{\sigma}+\pounds_{G_1*\left({\bf m}-{\sigma}\cdot A\right)\,}{\sigma}+{\rm ad}^*_{\,G_2\,*\,{\sigma} -A\,\cdot\, G_1\,*\,\left({\bf
m}-{\sigma}\cdot A\right)}\ {\sigma}=0.
\end{array}\right.
\end{equation}
In \cite{HoTr2008}, these equations were derived in the special case when $A=0$, i.e. in the absence of external magnetic { Yang-Mills} fields.

\begin{remark}[Relevant specializations]\normalfont When the manifold $M$ is one
dimensional, the Lie group is $\mathcal{O}=S^1$, and $A=0$, we
recover the two-component Camassa-Holm equation (CH2)
\cite{ChLiZh2005,Ku2007} by choosing the differential operators
$Q_1=(1-\alpha^2\partial^2_x)$ and $Q_2=1$. The modified
two-component system (MCH2) \cite{HoOnTr2009} requires
$Q_1=(1-\alpha_1^2\partial^2_x)$ and
$Q_2=(1-\alpha_2^2\partial^2_x)$. The higher dimensional and
anisotropic versions studied in \cite{HoTr2008} are obtained by
choosing an arbitrary manifold $M$, an arbitrary group
$\mathcal{O}$, and $A=0$. The corresponding choices for the
differential operators are $Q_1=(1-\alpha^2\Delta), Q_2=1$ (for
$n$-CH2) and $Q_1=(1-\alpha^2_1\Delta), Q_2=(1-\alpha_2^2\Delta)$
(for $n$-MCH2).
\end{remark}

\begin{remark}[The Kelvin-Noether circulation theorem]\normalfont The Kelvin-Noether
theorem is a version of the Noether theorem that holds for solutions
of the Euler-Poincar\'e equations \cite{HoMaRa98}. Let $(\mathbf{u},{\nu})$ be a
solution of the EP$\mathcal{A}ut$ equation \eqref{EPAut_trivial},
and let $\rho$ be a density variable satisfying the equation
$\partial_t\rho+\pounds_{\mathbf{u}}\rho=0$. Then we have
\begin{equation}\label{KN_trivial}
\frac{d}{dt}\oint_{c_t}\frac{1}{\rho}\frac{\delta l}{\delta \mathbf{u}}=-\oint_{c_t}\frac{1}{\rho}\frac{\delta l}{\delta {\nu}}\!\cdot\!\mathbf{d}{\nu},
\end{equation}
where $c_t$ is a loop in $M$ moving with the fluid velocity $\mathbf{u}$.
In the special case $\mathcal{O}=S^1$, one has $\rho:={\delta l}/{\delta{\nu}}$ so that the right-hand side vanishes and Kelvin-Noether theorem yields circulation conservation.
\end{remark}

%%%%%%%%%%%%%%%%%%%%%%%%%%%%%%%%%%%%%%%%%%%%%%%%%%%%%%%%%%%%

\subsection{Momentum maps}\label{Sec:DualPairTrivial}
{
\subsubsection{General setting: the Kaluza-Klein configuration space}
}

In order to characterize the geometry of singular solutions, we
consider their dynamics to take place in the phase space associated
to a Kaluza-Klein configuration manifold
\begin{equation}\label{QKK_trivial}
Q_{KK}=\operatorname{Emb}(S,M)\times\mathcal{F}(S,\mathcal{O})\ni
(\mathbf{Q},\theta)
\end{equation}
 and we consider the left action of
$(\eta,\chi)\in
\mathcal{A}ut(P)\simeq\operatorname{Diff}(M)\,\circledS\,\mathcal{F}(M,\mathcal{O})$
defined by
\begin{equation}\label{left_action_trivial}
(\eta,\chi)(\mathbf{Q},\theta):=(\eta\circ\mathbf{Q},(\chi\circ\mathbf{Q})\theta).
\end{equation}
Since $\operatorname{Emb}(S,M)$ is an open subset of
$\mathcal{F}(S,M)$, its tangent space at $\mathbf{Q}$ consists of
vector fields $V_\mathbf{Q}:S\rightarrow TM$ along $\mathbf{Q}$,
that is, we have $V_{\mathbf{Q}}(s)\in T_{\mathbf{Q}(s)}M$. In the
same way, by fixing a volume form ${\rm d}^ks$ on $S$, the cotangent
space at the embedding $\mathbf{Q}$ consists of one-forms
$\mathbf{P}_{\mathbf{Q}}:S\rightarrow T^*M$ along $\mathbf{Q}$. Note that when $M= \mathbb{R}  ^n $ then the cotangent bundle is the product $T^*\operatorname{Emb}(S,\mathbb{R}  ^n )=\operatorname{Emb}(S,\mathbb{R}  ^n )\times \mathcal{F} (S, \mathbb{R}  ^n )$ and therefore, we can denote an element in the cotangent bundle by a
pair $(\mathbf{Q},\mathbf{P})$. However, for an arbitrary manifold
$M$, $T^*\operatorname{Emb}(S,M)$ is not necessarily a product, and
one should use the notation $ \mathbf{P} _{ \mathbf{Q} } $. The tangent space to
$\mathcal{F}(S,\mathcal{O})$ at $\theta$ consists of functions
$v_\theta:S\rightarrow T\mathcal{O}$ along $\theta$. Similarly, when
a volume form ${\rm d}^ks$ is fixed on $S$, the cotangent bundle
$T^*\F(S,\O)$ at
$\theta$ consists of functions $\kappa_\theta:S\rightarrow
T^*\mathcal{O}$ covering $\theta$.

\subsubsection{Left-action momentum map and singular solutions} We
now compute the momentum map associated to the action
\eqref{left_action_trivial} cotangent lifted to $T^*Q_{KK}$. Given
$(\mathbf{u},{\nu})\in
\mathfrak{X}(M)\,\circledS\,\mathcal{F}(M,\mathfrak{o})$, the
infinitesimal generator associated to the left action
\eqref{left_action_trivial} reads
\[
(\mathbf{u},{\nu})_{Q_{KK}}(\mathbf{Q},\theta)=(\mathbf{u}\circ\mathbf{Q},({\nu}\circ\mathbf{Q})\theta).
\]
Given a Lie group $G$ acting on a configuration manifold $Q$, the momentum map $\mathbb{J}:T^*Q\rightarrow\mathfrak{g}^*$ for the cotangent lifted action to $T^*Q$ reads
\[
\langle \mathbb{J}(\alpha_q),\xi\rangle=\langle\alpha_q,\xi_Q(q)\rangle,
\]
where $\alpha_q\in T^*Q$, $\xi\in\mathfrak{g}$ and $\xi_Q$ is the infinitesimal generator of the action of $\mathfrak{g}$ on $Q$ associated to the Lie algebra element $\xi$. By applying this formula to our case, $\g=\mathfrak{aut}(P)$, we get the expression $\mathbf{J}_L:T^*Q_{KK}\rightarrow\mathfrak{X}(M)^*\times\mathcal{F}(M,\mathfrak{o})^*$,
\begin{equation}\label{left_momap_trivial}
\mathbf{J}_L\left(\mathbf{P}_\mathbf{Q},\kappa_\theta\right)=
\left(\int_S\mathbf{P}_\mathbf{Q}(s)\delta(x-\mathbf{Q}(s)){\rm d}^ks,\int_S\kappa_{\theta}(s)\theta(s)^{-1}\delta(x-\mathbf{Q}(s)){\rm d}^ks\right).
\end{equation}
We shall show that this momentum map recovers the singular solution { of \eqref{LPAut_trivial}
in \cite{HoTr2008}, in the case $A$=0}. First, one
observes that $\mathbf{J}_L$ is invariant under the \textit{right\,}
action of $\mathcal{F}(S,\mathcal{O})$ on its cotangent bundle.
Thus, $\mathbf{J}_L$ induces a map $\widetilde{\mathbf{J}}_L$ on the
reduced space
\[
T^*Q_{KK}/\mathcal{F}(S,\mathcal{O})\simeq
T^*\operatorname{Emb}(S,M)\times\mathcal{F}(S,\mathfrak{o})^*
\]
that is obtained by replacing $\mu(s)=\kappa_\theta(s)\theta^{-1}(s)\in\mathcal{F}(S,\mathfrak{o})^*$ in \eqref{left_momap_trivial}. 
The reduced Poisson structure on
$T^*\operatorname{Emb}(S,M)\times\mathcal{F}(S,\mathfrak{o})^*$ is
the sum of the canonical Poisson bracket on
$T^*\operatorname{Emb}(S,M)$ and the right Lie-Poisson bracket on
$\mathcal{F}(S,\mathfrak{o})^*$, that is, if $M= \mathbb{R}  ^n $, we have
\begin{equation}\label{reduced_Poisson}
\{F,G\}(\mathbf{Q},\mathbf{P},\mu)=
\int_S\left(\frac{\delta F}{\delta\mathbf{Q}}\!\cdot\!\frac{\delta G}{\delta\mathbf{P}}-\frac{\delta F}{\delta\mathbf{P}}\!\cdot\!\frac{\delta G}{\delta\mathbf{Q}}\right){\rm d}^ks+\int_S\mu\cdot\left[\frac{\delta F}{\delta\mu},\frac{\delta G}{\delta\mu}\right]{\rm d}^ks.
\end{equation}
Since the right action commutes with the left action of
$\operatorname{Diff}(M)\,\circledS\,\mathcal{F}(M,\mathcal{O})$ on
$T^*Q_{KK}$, the map $\widetilde{\mathbf{J}}_L$ is a momentum map
relative to the Poisson structure \eqref{reduced_Poisson} and to the
induced left action on the reduced space
$T^*\operatorname{Emb}(S,M)\times\mathcal{F}(S,\mathfrak{o})^*$.
This recovers the result of \cite{HoTr2008}, by reduction of the
canonical structure. This also shows that considering
$M\times\mathcal{O}$ as a trivial principal bundle is very natural
in this context. Indeed, the momentum map $\mathbf{J}_L$ is simply
associated to a cotangent lifted action. For a certain class of
EP$_{\!}\Aut$ Hamiltonians $h$, such as (\ref{EPAut-Ham}), it is
possible to define the collective Hamiltonian
\[
H:=h\circ\mathbf{J}_L: T^*Q_{KK}\to\Bbb{R}
\] 
which produces collective Hamiltonian dynamics; see \cite{HoMa2004} for the special case $\mu(s)\equiv 0$. This is for example the case for Hamiltonians
associated to $H^1$ metrics.

Because of its equivariance, the map $\mathbf{J}_L$ is also a
Poisson map with respect to the canonical Poisson structure on the
canonical cotangent bundle $T^*Q_{KK}$ and the right Lie-Poisson
structure on
$\mathfrak{X}(M)^*\times\mathcal{F}(M,\mathfrak{o})^*$.
Thus, any solution $(\mathbf{P}_\mathbf{Q},\kappa_\theta)$ of the
canonical Hamilton's equations on $T^*Q_{KK}$ associated to $H=h\circ\mathbf{J}_L$ projects formally,
via $\mathbf{J}_L$, to a (measure-valued) solution of the
Lie-Poisson equation associated to $h$. One can also use the
momentum map $\widetilde{\mathbf{J}}_L$. In this case, one has to
solve Hamilton's equations on
$T^*\operatorname{Emb}(S,M)\times\mathcal{F}(S,\mathfrak{o})^*$,
given by
\[
\dot{\mathbf{Q}}=\frac{\partial H}{\partial\mathbf{P}},\quad
\dot{\mathbf{P}}=-\frac{\partial H}{\partial\mathbf{Q}},\quad
\dot\mu=-\operatorname{ad}^*_{\delta H/\delta \mu}\mu,
\]
when $M= \mathbb{R}  ^n $.

\begin{remark}[Physical interpretation]{\rm
This
collective Hamiltonian extends the one in \cite{HoTr2008} to  consider an external Yang-Mills magnetic field, given by the
potential $A$ through the relation $B=\mathbf{d}^A A$. This may be of interest
in 
the gauge theory of fluids with non-Abelian interactions, which plays a fundamental role in certain areas of condensed matter physics
(cf. e.g.
\cite{HoKu88,Ho2002,GBRa2009}). }
\end{remark}

\begin{remark}[Existence of singular solutions]{\rm The
existence of the momentum map ${\bf J}_L$ does not guarantee that
the latter is also a solution of the system. For example, the use of
$L^2$ norms in the EP$\Aut$ Lagrangian \eqref{KK_Lagrangian_trivial}
prevents singularities. As a consequence, the collective Hamiltonian
$H=h\circ {\bf J}_L$ may not be defined, since it requires the
Hamiltonian $h$ to be well defined on the image of the momentum map.
However, the use of appropriate norms, such as $H^1$, allows for
singular solutions and this is one of the reasons why we consider
this particular case. Also, the absence of a magnetic field in the
Abelian case returns the MCH2 equations (cf. \cite{HoOnTr2009}),
which are known to produce singularities in finite time. Whether
this steepening phenomenon persists upon introducing a static
magnetic field, this is an open question that deserves future
investigation. }
\end{remark}

\subsubsection{Right-action momentum map and Noether's Theorem}
We now consider the right action of
$(\gamma,\beta)\in\mathcal{A}ut(P_S)=\operatorname{Diff}(S)\,\circledS\,\mathcal{F}(S,\mathcal{O})$
on $Q_{KK}$ defined by
\begin{equation}\label{right_action_trivial}
(\mathbf{Q},\theta)(\gamma,\beta):=(\mathbf{Q}\circ\gamma,(\theta\circ\gamma)\beta).
\end{equation}
Given $(\mathbf{v},\ze)\in\X(S)\,\circledS\,\F(S,\O)$, the infinitesimal generator is
\[
(\mathbf{v},\zeta)_{Q_{KK}}(\mathbf{Q},\theta)=(\mathbf{d}\mathbf{Q}\!\cdot\!\mathbf{v},\mathbf{d}\theta\!\cdot\!\mathbf{v}+\theta\zeta),
\]
thus, the momentum map associated to the cotangent lifted action is
\begin{equation}\label{momap_right_trivial}
\mathbf{J}_R:T^*Q_{KK}\rightarrow\mathfrak{X}(S)^*\times\mathcal{F}(S,\mathfrak{o})^*,\quad\mathbf{J}_R\left(\mathbf{P}_\mathbf{Q},\kappa_\theta\right)
=\left(\mathbf{P}_\mathbf{Q}\!\cdot\!\mathbf{d}\mathbf{Q}+\kappa_\theta\!\cdot\!\mathbf{d}\theta,\theta^{-1}\kappa_\theta\right).
\end{equation}
which in turn identifies a \textit{Clebsch representation} in the sense of \cite{MaWe83}.
Since the collective Hamiltonian $H$ on $T^*Q_{KK}$ is invariant under the cotangent lift of the right action \eqref{right_action_trivial}, Noether's Theorem asserts that
\[
\frac{d}{dt}\left(\mathbf{P}_\mathbf{Q}\!\cdot\!\mathbf{d}\mathbf{Q}+\kappa_\theta\!\cdot\!\mathbf{d}\theta\right)=0\quad\text{and}\quad \frac{d}{dt}(\theta^{-1}\kappa_\theta)=0,
\]
for any solutions $(\mathbf{P}_\mathbf{Q},\kappa_\theta)\in
T^*Q_{KK}$ of Hamilton's equations.
To summarize the situation, we have found the following dual pair
structure

\begin{picture}(150,100)(-70,0)%
\put(115,75){$T^{\ast} Q_{KK}$}
%top label

\put(90,50){$\mathbf{J}_L$}
%left label

\put(160,50){$\mathbf{J}_R$}
%right arrow label

\put(52,15){$\mathfrak{aut}(M\times\mathcal{O})^*$
%$\mathfrak{X}^{\ast}(M)\times\mathcal{F}^*(M,\mathfrak{o})$
}
%left bottom label

\put(150,15){$\mathfrak{aut}(S\times\mathcal{O})^*$
%$\mathfrak{X}^{\ast}(S)\times\mathcal{F}^*(S,\mathfrak{o})$
}
%right bottom label

\put(130,70){\vector(-1, -1){40}}
% left slanted arrow

\put(135,70){\vector(1,-1){40}}
% right slanted arrow

\end{picture}\\
where
$\mathfrak{aut}(M\times\mathcal{O})^*=\mathfrak{X}(M)^*\times\mathcal{F}(M,\mathfrak{o})^*$
and analogously for $\mathfrak{aut}(S\times\mathcal{O})^*$. The geometric meaning of these dual pairs will be
presented later in a more general setting.

%%%%%%%%%%%%%%%%%%%%%%%%%%%%%%%%%%%%%%%%%%%%%%%%%%%%%%%%%%%%

\begin{remark}[The case $S=M$]{\rm
When $S=M$, then the embeddings of $S$ into $M$ are the
diffeomorphisms of $M$, that is,
$\operatorname{Emb}(S,M)=\operatorname{Diff}(M)$ and $Q_{KK}=
\operatorname{Diff}(M)\times\mathcal{F}(M,\mathcal{O})$. Moreover,
in this case the left and right actions \eqref{left_action_trivial},
\eqref{right_action_trivial} recover left and right translations on
the automorphism group. Therefore, the Kaluza-Klein configuration
manifold can be identified with the group
$\operatorname{Diff}(M)\,\circledS\,\mathcal{F}(M,\mathcal{O})\simeq
\mathcal{A}ut(M\times\mathcal{O})$, and the dual pair recovers the usual body and
spatial representations for mechanical systems on Lie groups. More
precisely, $\mathbf{J}_L$ recovers the Lagrange-to-Euler map, while
$\mathbf{J}_R$ corresponds to the conserved  momentum density.

}
\end{remark}

%%%%%%%%%%%%%%%%%%%%%%%%%%%%%%%%%%%%%%%%

\section{EP$\!\Aut$ flows on non-trivial principal bundles}\label{Sec:NonTrivialEPAut}

{
\subsection{Basic definitions}
}
In this section we consider an arbitrary principal $\mathcal{O}$-bundle $\pi:P\rightarrow M$. Recall that the Lie algebra $\mathfrak{aut}(P)$ of the automorphism group consists of equivariant vector fields $U$ on $P$, that is, we have
\[
T\Phi_g\circ U=U\circ\Phi_g,\quad\text{for all $g\in\mathcal{O}$.}
\]
In the case of non-trivial bundles, it is necessary to introduce a principal connection $\mathcal{A}$ to split the tangent space into its vertical and horizontal subspaces. Recall that a principal connection is given by an $\mathfrak{o}$-valued one form $\mathcal{A}\in \Omega^1(P,\mathfrak{o})$ such that
\[
\Phi_g^*\mathcal{A}=\operatorname{Ad}_{g^{-1}}\circ \mathcal{A}\quad\text{and}\quad\mathcal{A}(\xi_P)=\xi,
\]
where $\xi_P$ is the infinitesimal generator associated to the Lie algebra element $\xi$.
Using a principal connection, we obtain an isomorphism
\begin{equation}\label{isomorphism_A}
\mathfrak{aut}(P)\rightarrow\mathfrak{X}(M)\times\mathcal{F}_\mathcal{O}(P,\mathfrak{o}),\quad U\mapsto ([U],\mathcal{A}(U)),
\end{equation}
where $[U]\in\mathfrak{X}(M)$ is defined by the condition $T\pi\circ[U]=U\circ\pi$
and $\mathcal{F}_\mathcal{O}(P,\mathfrak{o})$ is the space of
$\O$-equivariant functions $\om:P\rightarrow\mathfrak{o}$, \ie $\om\circ\Phi_g=\operatorname{Ad}_{g^{-1}}\circ\, \om$, for all 
$g\in\mathcal{O}$.
The inverse of the isomorphism \eqref{isomorphism_A} reads $(\mathbf{u},\om)\mapsto \operatorname{Hor}^\mathcal{A}\mathbf{u}+\sigma(\om)$,
where $\operatorname{Hor}^\mathcal{A}$ is the horizontal-lift associated to the principal connection $\mathcal{A}$ and $\sigma(\om)$ is the vertical vector field on $P$ defined by
$\sigma(\om)(p):=\left(\om(p)\right)_P(p)$.
The dual isomorphism reads
\begin{equation}\label{isomorphism_A_dual}
\mathfrak{aut}(P)^*\rightarrow\mathfrak{X}(M)^*\times\mathcal{F}_\mathcal{O}(P,\mathfrak{o})^*,\quad \beta\rightarrow \left(\left(\operatorname{Hor}^\mathcal{A}\right)^*\beta,\mathbb{J}\circ \beta\right),
\end{equation}
where $\mathfrak{aut}(P)^*=\Omega_\mathcal{O}^1(P)\otimes\operatorname{Den}(M)$ is the space of $\mathcal{O}$-invariant one-form densities and analogously
$\F_{\O}(P,\mathfrak{o})^*=\F_{\O}(P,\mathfrak{o}^*)\otimes\operatorname{Den}(M)$. Here, $\mathbb{J}:T^*P\rightarrow\mathfrak{o}^*$ denotes the cotangent bundle momentum map, $\langle\mathbb{J}(\alpha_p),\xi\rangle=\langle\alpha_p,\xi_P(p)\rangle$.

\subsection{General EP$\mathcal{A}ut$ equations on principal bundles}\label{General_EPAut}

Given a Lagrangian $l:\mathfrak{aut}(P)\rightarrow\mathbb{R}$, the associated EP$\mathcal{A}ut$ equation is
\begin{equation}\label{EPAut_general}
\frac{\partial }{\partial t}\frac{\delta l}{\delta U}+\pounds_U\frac{\delta l}{\delta U}=0.
\end{equation}
We now split this equation using the isomorphism \eqref{isomorphism_A} associated to a fixed principal connection
$\mathcal{A}$. Therefore, the Lagrangian $l:\mathfrak{aut}(P)\rightarrow\mathbb{R}$ induces a connection dependent Lagrangian $l^\mathcal{A}:\mathfrak{X}(M)\times \mathcal{F}_\mathcal{O}(P,\mathfrak{o})\to\mathbb{R}$ defined by
\[
l(U)=l^\mathcal{A}(\mathbf{u},{\omega}),\quad \mathbf{u}=[U],\quad{\omega}=\mathcal{A}(U).
\]
Using the dual isomorphism \eqref{isomorphism_A_dual},
the left hand side of \eqref{EPAut_general} reads
\[
\frac{\partial }{\partial
t}\left(\frac{\delta l^\mathcal{A}}{\delta \mathbf{u}},\frac{\delta
l^\mathcal{A}}{\delta {\omega}}\right).
\]
In order to split the right hand side, we will need the lemma below.

Recall that, associated to the principal connection $\mathcal{A}$,
there is a covariant exterior derivative on $\mathfrak{o}$--valued differential forms on $P$ defined by $\mathbf{d}^\mathcal{A}=\hor^*\mathbf{d}$,
\ie
\[
\mathbf{d}^\mathcal{A}\om(X_1,\dots,X_{k+1})=\mathbf{d}\om(\hor(X_1),\dots,\hor(X_{k+1})),
\]
where $\operatorname{hor}$ is the horizontal part of a vector in $TP$. In particular, the curvature of $\mathcal{A}$ is given by
\[
\mathcal{B}:=\mathbf{d}^\mathcal{A}\mathcal{A}=\mathbf{d}\mathcal{A}+[\mathcal{A},\mathcal{A}].
\]

\begin{lemma}\label{lemma_split} Let $\beta \in\mathfrak{aut}(P)^*$ and $U\in\mathfrak{aut}(P)$, and fix a principal connection $\mathcal{A}$ on $P$. Then we have
for $\mathbf{u}=[U]$ and $\om=\mathcal{A}(U)$:
\begin{align*}
\left(\operatorname{Hor}^\mathcal{A}\right)^*\pounds_U\beta&=
\pounds_\mathbf{u}\left(\operatorname{Hor}^\mathcal{A}\right)^*\beta +\left(\operatorname{Hor}^\mathcal{A}\right)^*\left(\mathbb{J}\circ \beta \!\cdot\!\mathbf{d}\omega+\pounds_{\operatorname{Hor}^\mathcal{A}\mathbf{u}}(\mathbb{J}\circ \beta \!\cdot\!\mathcal{A})\right)\\
\mathbb{J}\circ\pounds_U \beta &=\pounds_{\operatorname{Hor}^\mathcal{A}\mathbf{u}}(\mathbb{J}\circ \beta )+\operatorname{ad}^*_\omega(\mathbb{J}\circ \beta ).
\end{align*}
\end{lemma}
\medskip
See Appendix \ref{lemma1} for a proof. Applying this Lemma, we obtain the following result.

\begin{proposition}[$\operatorname{EP}\!\mathcal{A}ut$ equations on principal bundles] The Eu\-ler-Poincar\'e equations on the automorphism group of a principal bundle, relative to a Lagrangian $l:\mathfrak{aut}(P)\rightarrow\mathbb{R}$ and a principal connection $\mathcal{A}$, are given by
\begin{equation}\label{EPDiff_nontrivial}
\left\{\begin{array}{l}
\vspace{0.2cm}\displaystyle\frac{\partial}{\partial t}\frac{\delta l^\mathcal{A}}{\delta\mathbf{u}}
+\pounds_\mathbf{u}\frac{\delta l^\mathcal{A}}{\delta\mathbf{u}}+\left(\operatorname{Hor}^\mathcal{A}\right)^*\left(\frac{\delta l^\mathcal{A}}{\delta\omega}\!\cdot\!\mathbf{d}\omega
+\pounds_{\operatorname{Hor}^\mathcal{A}\mathbf{u}}\left(\frac{\delta l^\mathcal{A}}{\delta\omega}\!\cdot\!\mathcal{A}\right)\right)=0\\
\displaystyle\frac{\partial}{\partial t}\frac{\delta l^\mathcal{A}}{\delta\omega}+\pounds_{\operatorname{Hor}^\mathcal{A}\mathbf{u}}\frac{\delta l^\mathcal{A}}{\delta\omega}+\operatorname{ad}^*_\omega\frac{\delta l^\mathcal{A}}{\delta\omega}=0.
\end{array}\right.
\end{equation}
\end{proposition}

\subsection{Formulation in terms of the adjoint bundle} We now
rewrite the EP$\mathcal{A}ut$-equations \eqref{EPDiff_nontrivial} by identifying
$\mathcal{F}_\mathcal{O}(P,\mathfrak{o})$ with the space
$\Gamma(\operatorname{Ad}P)$ of sections of the adjoint bundle
\[
\operatorname{Ad}P:=(P\times\mathfrak{o})/\mathcal{O}\rightarrow M.
\]
Recall that the quotient is taken relative to the diagonal action of $\mathcal{O}$ on $P\times\mathfrak{o}$ given by $(p,\xi)\mapsto (\Phi_g(p),\operatorname{Ad}_{g^{-1}}\xi)$. We will denote by $[p,\xi]_\mathcal{O}$ an element of the adjoint bundle.
There is a Lie algebra isomorphism
\begin{equation}\label{tilde_isomorphism}
\omega\in\mathcal{F}_\mathcal{O}(P,\mathfrak{o})\mapsto \tilde{\omega}\in\Gamma(\operatorname{Ad}P)
\end{equation}
defined by
\[
\tilde{\omega}(x):=[p,\omega(p)]_\mathcal{O},
\]
where $p\in P$ is such that $\pi(p)=x$. This isomorphism extends to $k$-forms as follows. Let
\[
\overline{\Omega^k}(P,\mathfrak{o}):=\left\{\alpha\in\Omega^k(P,\mathfrak{o})\mid\Phi_g^*\alpha=\operatorname{Ad}_{g^{-1}}\alpha\quad\text{and}\quad \mathbf{i}_{\xi_P}\alpha=0,\quad\text{for all $g\in G$ and $\xi\in\mathfrak{o}$}\right\}
\]
be the space of  $\mathcal{O}$--equivariant horizontal $k$--forms on $P$, and let $\Omega^k(M,\operatorname{Ad}P)$ be the space of $\operatorname{Ad}P$-valued $k$-forms on $M$. Then we have the isomorphism
\[
\alpha\in \overline{\Omega^k}(P,\mathfrak{o})\mapsto \tilde{\alpha}\in \Omega^k(M,\operatorname{Ad}P),\quad\tilde{\alpha}(x)(u^1_x,...,u^k_x):=\left[p,\alpha(p)\left(U^1_p,...,U^k_p\right)\right]_\mathcal{O},
\]
where $U^i_p\in T_pP$ is such that $T_p\pi(U^i_p)=u^i_x$, $i=1,...,k$.
An example of $\mathcal{O}$--equivariant horizontal $2$-form is provided by the curvature $\mathcal{B}=\mathbf{d}\mathcal{A}+[\mathcal{A},\mathcal{A}]$ to which is associated the reduced curvature form $\tilde{\mathcal{B}}\in\Omega^2(M,\operatorname{Ad}P)$.

Given a principal connection $\mathcal{A}$ on $P$, we get a linear connection $\nabla^\mathcal{A}$ on the adjoint bundle. For $\omega\in\mathcal{F}_\mathcal{O}(P,\mathfrak{o})$, the relation between the covariant exterior derivative $\mathbf{d}^\mathcal{A}$ and the connection $\nabla^\mathcal{A}$ is
\[
\widetilde{\mathbf{d}^\mathcal{A}\omega}=\nabla^\mathcal{A}\tilde\omega.
\]
This formula extends the operator $\nabla^\mathcal{A}$ to $\Omega^k(M,\operatorname{Ad}P)$. We shall use the covariant Lie derivative defined by
\[
\pounds^\mathcal{A}_\mathbf{u}=\mathbf{i}_\mathbf{u}\nabla^\mathcal{A}+\nabla^\mathcal{A}\mathbf{i}_\mathbf{u}
\]
on $\Ad(P)$-valued differential forms, and related to the ordinary Lie derivative by the formula
\[
\widetilde{\pounds_{\operatorname{Hor}^\mathcal{A}\mathbf{u}}\alpha}=\pounds_\mathbf{u}^\mathcal{A}\tilde{\alpha}.
\]

By combining the isomorphism \eqref{isomorphism_A} with \eqref{tilde_isomorphism}, we identify $\mathfrak{aut}(P)$ with $\mathfrak{X}(M)\times\Gamma(\operatorname{Ad}P)$.
In this case, we have the following analogue of Lemma \ref{lemma_split}.
\begin{lemma}\label{lemma_split_adjoint} 
Let $\beta \in\mathfrak{aut}(P)^*$ and $U\in\mathfrak{aut}(P)$, and fix a principal connection $\mathcal{A}$ on $P$. Then we have
for $\mathbf{u}=[U]$ and $\tilde{\om}=\widetilde{\mathcal{A}(U)}$:
\begin{align*}
\left(\operatorname{Hor}^\mathcal{A}\right)^*\pounds_U\beta&=
\pounds_\mathbf{u}\left(\operatorname{Hor}^\mathcal{A}\right)^*\beta +\widetilde{\mathbb{J}\circ \beta }\!\cdot\!\left(\nabla^\mathcal{A}\tilde{\omega}+\mathbf{i}_\mathbf{u}\tilde{\mathcal{B}}\right)\\
\widetilde{\mathbb{J}\circ\pounds_U \beta }&=\pounds^\mathcal{A}_\mathbf{u}\widetilde{\mathbb{J}\circ \beta }+\operatorname{ad}^*_{\tilde\omega}\widetilde{\mathbb{J}\circ \beta },
\end{align*}
where $\pounds^\mathcal{A}$ is the covariant Lie derivative acting on $\operatorname{Ad}^*P$-valued densities, that is, for $\mu=\sigma \otimes\alpha$, where $\alpha$ is a density and $\sigma$ a section of $\operatorname{Ad}^*P\rightarrow M$, we have
\[
\pounds^\mathcal{A}_\mathbf{v}\mu=\pounds^\mathcal{A}_\mathbf{v}(\sigma \otimes\alpha)=\nabla^\mathcal{A}_\mathbf{v}\sigma\otimes \alpha+\sigma\operatorname{div}(\mathbf{v})\otimes\alpha.
\]
\end{lemma}
This is used below to obtain the expression of the EP$\mathcal{A}ut$ equation in terms of the variables $\mathbf{u}$ and $\tilde{\omega}$.

\begin{theorem}[Alternative formulation of the $\operatorname{EP}\mathcal{A}ut$ equations]\label{altern_form} The Euler-Poincar\'e equations \eqref{EPDiff_nontrivial} on the automorphism group of a principal bundle can be alternatively written as
\begin{equation}\label{EPAut_Adjoint}
\left\{\begin{array}{l}
\vspace{0.2cm}\displaystyle\frac{\partial}{\partial t}\frac{\delta l^\mathcal{A}}{\delta\mathbf{u}}
+\pounds_\mathbf{u}\frac{\delta l^\mathcal{A}}{\delta\mathbf{u}}+\frac{\delta l^\mathcal{A}}{\delta\tilde\omega}\!\cdot\!\left(\nabla^\mathcal{A}\tilde{\omega}+\mathbf{i}_\mathbf{u}\tilde{\mathcal{B}}\right)=0\\
\displaystyle\frac{\partial}{\partial t}\frac{\delta l^\mathcal{A}}{\delta\tilde\omega}+\pounds_{\mathbf{u}}^\mathcal{A}\frac{\delta l^\mathcal{A}}{\delta\tilde{\omega}}+\operatorname{ad}^*_{\tilde\omega}\frac{\delta l^\mathcal{A}}{\delta\tilde{\omega}}=0.
\end{array}\right.
\end{equation}
\end{theorem}
As in the trivial bundle case, it is convenient to fix a Riemannian metric $g$ on $M$ in order to write the equations more explicitly. Defining $\mathbf{v}$ and ${\zeta}$ by the equalities
\begin{align*}
\frac{\delta l^\mathcal{A}}{\delta\mathbf{u}}&=:\mathbf{v}^\flat\otimes\mu_M\in \Om^1(M)\otimes\Den(M)\\
\frac{\delta l^\mathcal{A}}{\delta\tilde\om}&=:{\zeta}\otimes\mu_M\in
\Ga(\Ad^*P)\otimes\Den(M)
\end{align*}
and using the formula
\[
\pounds^\mathcal{A}_{\mathbf{u}}({\zeta}\otimes\mu_M)
=\nabla^\mathcal{A}_{\mathbf{u}}{\zeta}\otimes\mu_M
+{\zeta}(\div\mathbf{u})\otimes\mu_M,
\]
the EP$\mathcal{A}ut$ equations
\eqref{EPAut_Adjoint} can be rewritten as
\[
\left\{\begin{array}{l}
\vspace{0.2cm}\displaystyle\partial_t\mathbf{v}+\nabla_\mathbf{u}\mathbf{v}+\nabla\mathbf{u}^\mathsf{T}\cdot\mathbf{v}+\mathbf{v}\div\mathbf{u}
+\left({\zeta}\!\cdot\!\left(\nabla^\mathcal{A}\tilde{\omega}+\mathbf{i}_\mathbf{u}\tilde{\mathcal{B}}\right)\right)^\sharp=0\\
\displaystyle\partial_t{\zeta}+\nabla_\mathbf{u}^\mathcal{A}{\zeta}+{\zeta}\operatorname{div}\mathbf{u}+\operatorname{ad}^*_{\tilde{\omega}}{\zeta}=0,
\end{array}\right.
\]
which reduce to an Euler-Poincar\'e equation on the diffeomorphism group (\cite{HoMaRa98}) in the case when $\mathcal{A}=0$ and $\zeta=0$.

\begin{remark}[Trivial bundles and curvature representation]\normalfont
When the principal bundle is trivial, we recover the equation
\eqref{EPAut_trivial}. Indeed, we can identify
$\omega\in\mathcal{F}_\mathcal{O}(P,\mathfrak{o})$ with
${\nu}+A\!\cdot\!\mathbf{u}\in\mathcal{F}(M,\mathfrak{o})$.
Here $A$ denotes the
one-form on $M$ induced by the connection $\mathcal{A}$ on $P$:  we have
$\mathcal{A}(v_x,\xi_g)=\operatorname{Ad}_{g^{-1}}\left(A(v_x)+\xi_gg^{-1}\right)$ and also
\[
l^\mathcal{A}(\mathbf{u},{\nu}+A\!\cdot\!\mathbf{u})=l(\mathbf{u},{\nu}).
\]
Thus, $\delta l^\mathcal{A}/\delta\omega$ can be identified with
$\delta l/\delta{\nu}$ and $\delta l^\mathcal{A}/\delta
\mathbf{u}$ is identified with $\delta l/\delta\mathbf{u}-\delta
l/\delta{\nu}\!\cdot\!A$. Similarly, the Lie derivative
$\pounds_{\operatorname{Hor}^\mathcal{A}\mathbf{u}}\delta
l^\mathcal{A}/\delta\omega$ can be identified with
$\pounds_\mathbf{u}\delta
l/\delta{\omega}-\operatorname{ad}^*_{A\!\cdot\!\mathbf{u}}\delta
l/\delta{\nu}$.
We shall denote by $B=\mathbf{d}A+[A,A]$ the two-form on $M$ induced by the curvature $ \mathcal{B}$. Then,
the EP$\mathcal{A}ut$ equations \eqref{EPAut_trivial} written in terms of the Lagrangian $l^\mathcal{A}=l^\mathcal{A}(\mathbf{u},{\omega})$ read
\begin{equation}\label{EPAut_trivial_curvature}
\left\{\begin{array}{l}
\vspace{0.2cm}\displaystyle\frac{\partial}{\partial t}\frac{\delta l^\mathcal{A}}{\delta\mathbf{u}}+\pounds_{\mathbf{u}}\frac{\delta l^\mathcal{A}}{\delta\mathbf{u}}+\frac{\delta l^\mathcal{A}}{\delta{\omega}}\!\cdot\!\left(\mathbf{d}{\omega}+[A,{\omega}]+\mathbf{i}_\mathbf{u}B\right)=0\\
\displaystyle\frac{\partial}{\partial t}\frac{\delta l^\mathcal{A}}{\delta{\omega}}+\pounds_{\mathbf{u}}\frac{\delta l^\mathcal{A}}{\delta{\omega}}+\operatorname{ad}^*_{{\omega}-A\!\cdot\!\mathbf{u}}\frac{\delta l^\mathcal{A}}{\delta{\omega}}=0,
\end{array}\right.
\end{equation}
where the curvature two form $B$ on $M$ satisfies the relation
${\bf i}_{\mathbf{u}}B
=[A\cdot \mathbf{u},A]-{\bf d}(A\cdot \mathbf{u})+\pounds_{\mathbf{u}}A$.
\end{remark}

\subsection{Kaluza-Klein Lagrangians} Let $g$ be a Riemannian
metric on $M$ and $\ga$ an $\operatorname{Ad}$-invariant scalar
product on $\mathfrak{o}$. These data, together with the principal
connection $\mathcal{A}$, define a \textit{Kaluza-Klein metric} on
the principal $\mathcal{O}$-bundle $\pi:P\rightarrow M$:
\begin{equation}\label{kakl}
\kappa_p(U_p,V_p)=g_{\pi(p)}
\left(T_p\pi(U_p),T_p\pi(V_p)\right)+\ga(\mathcal{A}_p(U_p),\mathcal{A}_p(V_p)),\quad
U_p,V_p\in T_pP.
\end{equation}
The associated $L^2$ Kaluza-Klein Lagrangian is obtained by integration over $M$ and reads
\[
l(U)=\int_Mg([U],[U])\mu_M+\int_M\gamma(\mathcal{A}(U),\mathcal{A}(U))\mu_M.
\]
Note that the last term is well defined, since the integrand is an $\mathcal{O}$--invariant function on $P$. More generally, we can consider Kaluza-Klein Lagrangians of the form
\begin{equation}\label{KK_Lagrangian}
l(U)=\frac{1}{2}\|[U]\|^2_1+\frac{1}{2}\|\mathcal{A}(U)\|_2^2,
\end{equation}
relative to inner product norms $\|\,\|_1$ and $\|\,\|_2$ on
$\mathfrak{X}(M)$ and $\mathcal{F}_{\mathcal{O}}(P,\mathfrak{o})$.
As before, the norm $\|\,\|_2$ is usually given with the help of and
$\operatorname{Ad}$-invariant inner product $\gamma$ on $\mathfrak{o}$, that
makes possible to integrate the expression over the base manifold
$M$. See \cite{GBRa2008a} for an application of the $L^2$ Kaluza-Klein Lagrangian to Yang-Mills fluids.
In the trivial case, \eqref{KK_Lagrangian}
reduces to
\eqref{KK_Lagrangian_trivial} and this emphasizes the role of the vector potential $A$ in CH2 and MCH2 dynamics.

We now compute explicitly the EP$\mathcal{A}ut$ equations for the Kaluza-Klein Lagrangian \eqref{KK_Lagrangian}. We assume that the norms are associated to symmetric and positive definite differential operators $Q_1$ and $Q_2$ on $TM$ and $\operatorname{Ad}P$, respectively.
In this case, the Lagrangian $l^\mathcal{A}:\mathfrak{X}(M)\times\Gamma(\operatorname{Ad}P)\rightarrow\mathbb{R}$ reads
\[
l^\mathcal{A}(\mathbf{u},\tilde\omega)=\int_Mg(Q_1\mathbf{u},\mathbf{u})\mu_M+\int_M\bar\gamma(Q_2\tilde\omega,\tilde\omega)\mu_M,
\]
where $\bar\gamma$ is the vector bundle metric induced on
$\operatorname{Ad}P$ by $\gamma$. We thus obtain
the EP$\mathcal{A}ut$ equations
\begin{equation}\label{kkl}
\left\{\begin{array}{l}
\vspace{0.2cm}\displaystyle\partial_t\mathbf{v}+\nabla_\mathbf{u}\mathbf{v}+\nabla\mathbf{u}^\mathsf{T}\cdot\mathbf{v}+\mathbf{v}\div\mathbf{u}
+{\zeta}\!\cdot\!\left(\nabla^\mathcal{A}\tilde{\omega}+\mathbf{i}_\mathbf{u}\tilde{\mathcal{B}}\right)^\sharp=0\\
\displaystyle\partial_t{\zeta}+\nabla_\mathbf{u}^\mathcal{A}{\zeta}+{\zeta}\operatorname{div}\mathbf{u}+\operatorname{ad}^*_{\tilde{\omega}}{\zeta}=0,
\end{array}\right.
\end{equation}
where
\[
\mathbf{v}=Q_1\mathbf{u}\in\mathfrak{X}(M)\quad\text{and}\quad\zeta=\bar\gamma(Q_2\tilde\omega,\cdot)\in\Gamma(\operatorname{Ad}^*P),
\]
which evidently extend MCH2 dynamics to the case of a Yang-Mills
fluid flow on a non-trivial principal $\mathcal{O}$-bundle with
fixed connection $\mathcal{A}$.

\subsection{Kelvin-Noether circulation theorem}

Let $U$ be a
solution of the EP$\mathcal{A}ut$ equation \eqref{EPAut_general} and
let $\rho$ be a density variable on $M$ satisfying the equation
$\partial_t\rho+\pounds_{\mathbf{u}}\rho=0$, where
$\mathbf{u}:=[U]$. Then we have
\[
\frac{d}{dt}\oint_{\gamma_t}\frac{1}{\rho}\frac{\delta l}{\delta U}=0,
\]
where $\gamma_t$ is a loop in the total space $P$ of the principal bundle which moves with the velocity $U$, see \S7 in \cite{GBRa2008a}. Given a connection $\mathcal{A}$ on $P$, we can decompose the functional derivative as
\[
\frac{\delta l}{\delta U}=\pi^*\frac{\delta l^\mathcal{A}}{\delta \mathbf{u}}+\frac{\delta l^\mathcal{A}}{\delta \omega}\!\cdot\!\mathcal{A},
\]
and the circulation reads
\begin{equation}\label{KN_nontrivial}
\frac{d}{dt}\left[\oint_{c_t}\frac{1}{\rho}\frac{\delta l^\mathcal{A}}{\delta\mathbf{u}}+\oint_{\gamma_t}\frac{1}{\rho}\frac{\delta l^\mathcal{A}}{\delta \omega}\!\cdot\!\mathcal{A}\right]=0,
\end{equation}
where $c_t:=\pi\circ\gamma_t$ is the closed curve in $M$ induced by $\gamma_t$. Note that the second integral takes into account the internal structure of the bundle, since $\gamma_t$ is a curve in $P$.

In the trivial bundle case the time dependent curve reads
$\gamma_t=(c_t,g_t)=(\eta_t\circ c_0,(\chi_t\circ c_0)g_0)$, and a
direct computation shows that formula \eqref{KN_nontrivial} yields
\begin{equation}\label{KN3} 
\frac{d}{dt} \left[ \oint_{ c _t } \frac{1}{\rho } \frac{\delta l }{\delta \mathbf{u} }+\oint_{ \gamma _t } \frac{1}{\rho}\frac{\delta l}{\delta\nu  } \cdot  \kappa ^l\right]  =0,
\end{equation}
where $ \kappa ^l ( \xi _g) = \xi _g g ^{-1} $ is the left Maurer-Cartan form on $ \mathcal{O} $. Note that $\frac{1}{\rho}\frac{\delta l}{\delta\nu  } \cdot  \kappa ^l$ is interpreted as a one-form on $P= M \times \mathcal{O} $, integrated along the curve $ \gamma _t \in M \times \mathcal{O} $. In a more explicit notation, the above Kelvin-Noether theorem can be written as
\begin{equation}%\label{KN3} 
\frac{d}{dt} \left[ \oint_{ c _t } \frac{1}{\rho } \frac{\delta l }{\delta \mathbf{u} }+\oint_{ \gamma _t } \frac{1}{\rho}\frac{\delta l}{\delta\nu  } \cdot  \mathbf{d}g g^{-1}\right]  =0.
\end{equation}
where $\gamma_t=(c_t,g_t)$ is a loop in the
trivial bundle $M \times \mathcal{O}$.

\bigskip

\bigskip

\subsection{Momentum maps and singular solutions}\label{sec:momaps_and_singsol}

\subsubsection{Preliminaries on the Kaluza-Klein configuration space}

Recall that in order to explain the geometric properties of the singular solutions in the case of a trivial bundle, we introduced the Kaluza-Klein configuration manifold \eqref{QKK_trivial}. We now describe the corresponding object in the case of an arbitrary principal bundle. Let $S\subset M$ be a submanifold of $M$ and consider two
principal $\O$-bundles $P\rightarrow P/\mathcal{O}=M$ and
$P_S\rightarrow P_S/\mathcal{O}=S$. We say that the map
$\mathcal{Q}:P_S\rightarrow P$ is equivariant if
$\mathcal{Q}\circ\Phi_g=\Phi_g\circ\mathcal{Q}$, for all
$g\in\mathcal{O}$, where $\Phi_g$ denotes the $\mathcal{O}$-action
on $P_S$ or $P$. Such a $\mathcal{Q}$ defines a unique map
$\mathbf{Q}:S\rightarrow M$ verifying the condition
$\pi\circ\mathcal{Q}=\mathbf{Q}\circ\pi$. We now define the object that generalizes the configuration space $Q_{KK}$ \eqref{QKK_trivial}.

\begin{definition} The Kaluza-Klein
configuration space is defined as the following subset of
equivariant maps from $P_S$ to $P$:
\[
Q_{KK}
=\{\mathcal{Q}:P_S\rightarrow P\mid \mathcal{Q}\circ\Phi_g=\Phi_g\circ\mathcal{Q}\;\;\text{and}\;\;\mathbf{Q}\in\operatorname{Emb}(S,M)\}.
\]
\end{definition}
Thus, $Q_{KK}$ consists of equivariant mappings that projects onto embeddings. Another characterization of the Kaluza-Klein configuration space is given in the following (see Appendix \ref{lemma3}
for a proof of this result)

\begin{lemma}\label{lemma_Q_KK}
The Kaluza-Klein configuration space coincides with the set
\[
Q_{KK}=\operatorname{Emb}_\mathcal{O}(P_S,P)
\]
of all equivariant embeddings of $P_S$ into $P$.
\end{lemma}

The tangent space
$T_{\mathcal{Q}}Q_{KK}$ consists of equivariant vector fields
$\mathcal{V}_{\mathcal{Q}}$ along $\mathcal{Q}$, that is
\[
\mathcal{V}_{\mathcal{Q}}(p_s)\in T_{\mathcal{Q}(p_s)}P,\;\;\text{for all $p_s\in P_S,$}\quad\text{and}\quad \mathcal{V}_\mathcal{Q}\circ\Phi_g=T\Phi_g\circ\mathcal{V}_\mathcal{Q}.
\]
Assume that the submanifold $S$ is endowed with a volume form ${\rm
d}^ks$.
The cotangent space $T^*_\mathcal{Q}Q_{KK}$ can
be identified with the space of equivariant one-forms
$\mathcal{P}_\mathcal{Q}:P_S\rightarrow T^*P$ along $\mathcal{Q}$.
Therefore, the contraction
$\mathcal{P}_\mathcal{Q}\!\cdot\!\mathcal{V}_\mathcal{Q}$ defines a
function on $S$ which can be integrated over $S$. We thus obtain the
paring
\begin{equation}\label{pairing}
\left\langle\mathcal{P}_\mathcal{Q},\mathcal{V}_\mathcal{Q}\right\rangle=\int_S\left(\mathcal{P}_\mathcal{Q}\!\cdot\!\mathcal{V}_\mathcal{Q}\right) {\rm d}^ks.
\end{equation}

The groups $\mathcal{A}ut(P)$ and $\mathcal{A}ut(P_S)$ act
naturally on $Q_{KK}$ by left and right composition,
respectively. Note that when $P$ and $P_S$ are trivial bundles, we recover the left and right actions given in \eqref{left_action_trivial} and \eqref{right_action_trivial}. As before, we denote by
\[
\mathbf{J}_L:T^*Q_{KK}\rightarrow\mathfrak{aut}(P)^*\quad\text{and}\quad \mathbf{J}_R:T^*Q_{KK}\rightarrow\mathfrak{aut}(P_S)^*
\]
the momentum maps associated to the cotangent lift of these actions. We compute below these momentum maps by introducing principal connections.

Given an equivariant map $f\in\mathcal{F}_\mathcal{O}(P_S,\mathfrak{o})$ and
$\mathcal{Q}\in\operatorname{Emb}(P_S,P)$, we define the vertical
vector field $\sigma_{\mathcal{Q}}(f)\in T_{\mathcal{Q}}Q_{KK}$ by
\[
\sigma_{\mathcal{Q}}(f)(p_s):=\left(f(p_s)\right)_P(\mathcal{Q}(p_s)).
\]
Given a vector field $V_{\mathbf{Q}}\in
T_\mathbf{Q}\operatorname{Emb}(S,M)$, a connection $\mathcal{A}$
on $P$, and $\mathcal{Q}\in Q_{KK}$ projecting to $\mathbf{Q}$, we
define the horizontal-lift
$\operatorname{Hor}^\mathcal{A}_{\mathcal{Q}}(V_\mathbf{Q})\in
T_\mathcal{Q}Q_{KK}$ of $V_{\mathbf{Q}}$ along $\mathcal{Q}$ by
\[
\operatorname{Hor}^\mathcal{A}_{\mathcal{Q}}(V_\mathbf{Q})(p_s):=\operatorname{Hor}^\mathcal{A}_{\mathcal{Q}(p_s)}(V_{\mathbf{Q}}(s)).
\]
This defines a connection dependent isomorphism
\[
\mathcal{F}_\mathcal{O}(P_S,\mathfrak{o})\times T_{\mathbf{Q}}\operatorname{Emb}(S,M)\rightarrow T_\mathcal{Q}Q_{KK},\quad (f,V_{\mathbf{Q}})\mapsto \sigma_{\mathcal{Q}}(f)+\operatorname{Hor}^\mathcal{A}_{\mathcal{Q}}(V_\mathbf{Q}),
\]
with inverse given by
\[
T_\mathcal{Q}Q_{KK}\rightarrow\mathcal{F}_\mathcal{O}(P_S,\mathfrak{o})\times T_{\mathbf{Q}}\operatorname{Emb}(S,M),\quad\mathcal{V}_{\mathcal{Q}}\mapsto \left(\mathcal{A}\circ\mathcal{V}_{\mathcal{Q}},\left[\mathcal{V}_\mathcal{Q}\right]\right)=:(f^\mathcal{A},V_\mathbf{Q}),
\]
where $\left[\mathcal{V}_\mathcal{Q}\right]\in
T_{\mathbf{Q}}\operatorname{Emb}(S,M)$ is defined by the condition
\[
\left[\mathcal{V}_\mathcal{Q}\right]\circ\pi=T\pi\circ \mathcal{V}_\mathcal{Q}.
\]
It will be useful to identify the space
$\mathcal{F}_{\mathcal{O}}(P_S,\mathfrak{o})$ of equivariant functions with
the space $\Gamma(\operatorname{Ad}P_S)$ of all sections of the
adjoint vector bundle.
Since $S$ is endowed with a volume form ${\rm d}^ks$,
the cotangent
bundle $T^*_\mathcal{Q}Q_{KK}$ can be naturally identified with the
space
\[
\mathcal{F}_\mathcal{O}(P_S,\mathfrak{o})^*\times T^*_\mathbf{Q}\operatorname{Emb}(S,M),
\]
where $\mathcal{F}_\mathcal{O}(P_S,\mathfrak{o})^*:=\mathcal{F}_\mathcal{O}(P_S,\mathfrak{o}^*)$ and $T^*_\mathbf{Q}\operatorname{Emb}(S,M)$ is defined as in the preceding section. The isomorphism is
\[
\mathcal{F}_\mathcal{O}(P_S,\mathfrak{o})^*\times T^*_\mathbf{Q}\operatorname{Emb}(S,M)\rightarrow T_{\mathcal{Q}}^*Q_{KK},\quad (\zeta,\alpha_{\mathbf{Q}})\mapsto \zeta\!\cdot\!\mathcal{A}+T^*\pi\!\cdot\!\alpha_\mathbf{Q}
\]
with inverse
\[
T_{\mathcal{Q}}^*Q_{KK}\rightarrow \mathcal{F}_\mathcal{O}(P_S,\mathfrak{o})^*\times T^*_\mathbf{Q}\operatorname{Emb}(S,M),\quad \mathcal{P}_{\mathcal{Q}}\mapsto \left(\mathbb{J}\circ\mathcal{P}_\mathcal{Q},\left(\operatorname{Hor}^\mathcal{A}_\mathcal{Q}\right)^*\mathcal{P}_\mathcal{Q}\right)=:\left(\zeta,\mathbf{P}^\mathcal{A}_\mathbf{Q}\right ),
\]
where $\mathbb{J}:T^*P\rightarrow\mathfrak{o}^*$ is defined by $\langle\mathbb{J}(\alpha_p),\xi\rangle=\langle\alpha_p,\xi_P(p)\rangle$.
Note that using these isomorphisms, the natural pairing \eqref{pairing} reads
\begin{equation}\label{split_pairing}
\int_S\left(\zeta\!\cdot\! f^\mathcal{A}+\mathbf{P}^\mathcal{A}_\mathbf{Q} \!\cdot\!V_\mathbf{Q}\right){\rm d}^ks.
\end{equation}

\subsubsection{Left-action momentum map and singular solutions} The
momentum map $\mathbf{J}_L$ associated to the cotangent lifted left
action of $\mathcal{A}ut(P)$ on $T^*Q_{KK}$ takes values in
$\mathfrak{aut}(P)^*
\simeq\mathfrak{X}(M)^*\times\mathcal{F}_\mathcal{O}(P,\mathfrak{o})^*$,
where
$\mathcal{F}_\mathcal{O}(P,\mathfrak{o})^*=\mathcal{F}_\mathcal{O}(P,\mathfrak{o}^*)\otimes\operatorname{Den}(M)$. It will be more convenient to identify
$\mathcal{F}_\mathcal{O}(P_S,\mathfrak{o})$ with the space
$\Gamma(\operatorname{Ad}P_S)$ of sections of the adjoint bundle
$\operatorname{Ad}P_S\rightarrow M$. In the same way, we will identify
$\mathcal{F}_\mathcal{O}(P_S,\mathfrak{o})^*= \mathcal{F} _ \mathcal{O} (P_S,\mathfrak{o}^*)$ with the space
$\Gamma(\operatorname{Ad}P_S)^*=\Gamma(\operatorname{Ad}^*P_S)$,
and we denote by $\bar\zeta$ the section associated to
$\zeta\in\mathcal{F}_\mathcal{O}(P_S,\mathfrak{o})^*$.
An embedding $\mathcal{Q}$ in $Q_{KK}$ induces
naturally a map
$\widetilde{\mathcal{Q}}:\operatorname{Ad}^*P_S\rightarrow\operatorname{Ad}^*P$
covering $\mathbf{Q}$.

\begin{proposition}[Left-action momentum map]
With the previous notations, the momentum map associated to the cotangent lifted left action of the automorphism group $\mathcal{A}ut(P)$ on $T^*Q_{KK}$ reads
\begin{align*}
\mathbf{J}_L\left( \mathcal{P}_\mathcal{Q}\right)=&\left(\int_S\mathbf{P}^\mathcal{A}_\mathbf{Q}(s)\delta (x-\mathbf{Q}(s)){\rm d}^ks,\int_S\widetilde{\mathcal{Q}}(\bar\zeta(s))\delta(x-\mathbf{Q}(s)){\rm d}^ks\right)\\
&\in\mathfrak{X}(M)^*\times\Gamma(\operatorname{Ad}P)^*.
\end{align*}
\end{proposition}
\begin{proof} Using the formula for the momentum map associated to a cotangent lifted action and formula \eqref{split_pairing}, we have
\begin{align*}
\left\langle\mathbf{J}_L\left( \mathcal{P}_\mathcal{Q}\right),U\right\rangle&=\int_S\mathcal{P}_\mathcal{Q}(p_s)\!\cdot\!U(\mathcal{Q}(p_s)){\rm d}^ks=\int_S\left(\zeta(p_s)\!\cdot\!\mathcal{A}(U(\mathcal{Q}(p_s)))+\mathbf{P}^\mathcal{A}_ \mathbf{Q} (s)\!\cdot\![U\circ\mathcal{Q}](s)\right){\rm d}^ks\\
&=\int_S\left(\bar\zeta(s)\!\cdot\!\overline{\mathcal{A}\circ U\circ\mathcal{Q}}(s)+\mathbf{P}^\mathcal{A}_\mathbf{Q} (s)\!\cdot\!([U]\circ\mathbf{Q}(s))\right){\rm d}^ks\\
&=\int_S\left(\widetilde{\mathcal{Q}}(\bar\zeta(s))\!\cdot\!\overline{\mathcal{A}\circ U}(\mathbf{Q}(s))+\mathbf{P}^\mathcal{A}_ \mathbf{Q} (s)\!\cdot\![U](\mathbf{Q}(s))\right){\rm d}^ks.
\end{align*}
Note that in the first line, the expression $\mathcal{P}_ \mathcal{Q} (p_s)\!\cdot\!U(\mathcal{Q}(p_s))$ only depends on $s$ and not on $p_s$, by equivariance. Thus it makes sense to integrate it on $S$. From the last equality we obtain the desired expression.
\end{proof}

\medskip

As in the trivial case, for a certain class of Hamiltonians $h:\mathfrak{aut}(P)\rightarrow\mathbb{R}$ we can define the collective Hamiltonian $H:=h\circ\mathbf{J}_L:T^*Q_{KK}\rightarrow\mathbb{R}$. Since $\mathbf{J}_L:T^*Q_{KK}\to\mathfrak{aut}(P)^*$ is a Poisson map, a solution of Hamilton's equations for $H$ on $T^*Q_{KK}$ gives a (possibly measure-valued) solution of the EP$\mathcal{A}ut$ equation associated to the Hamiltonian $h$.
The trivial bundle case is treated in Appendix \ref{trivial_case}.

\subsubsection{Right-action momentum map and Noether's Theorem} We
now compute the momentum map $\mathbf{J}_R$ associated to the right
action of $\mathcal{A}ut(P_S)$ on $T^*Q_{KK}$. The Lie algebra is
denoted by $\mathfrak{aut}(P_S)$ and consists of equivariant vector
fields on $P_S$. The infinitesimal generator associated to
$U\in\mathfrak{aut}(P_S)$ is $\mathbf{d}\mathcal{Q}\circ U$. Thus,
the momentum map is simply given by
\[
\mathbf{J}_R(\mathcal{P}_\mathcal{Q})=\mathcal{P}_ \mathcal{Q} \!\cdot\!\mathbf{d}\mathcal{Q}\in\mathfrak{aut}(P_S)^*.
\]
In the following proposition we give a more concrete expression for
$\mathbf{J}_R$, by fixing principal connections $\mathcal{A}_S$ and $\mathcal{A}$ on $P_S$ and $P$, respectively.

\begin{proposition}[Right-action momentum map] The momentum map associated to the cotangent lifted right action of the automorphism group of $P_S$ on $T^*Q_{KK}$ reads
\begin{equation}\label{altern}
\mathbf{J}_R\left( \mathcal{Q},\mathbf{P}^\mathcal{A}_{\mathbf{Q} },\zeta \right) =\left(\mathbf{P}^\mathcal{A}_{\mathbf{Q} }\!\cdot\!\mathbf{d}\mathbf{Q}+\zeta\cdot\left(\mathcal{Q}^*\A-\A_S\right),\zeta\right)\in\X(S)^*\x\F_\O(P_S,\ou)^*,
\end{equation}
where  $\mathbf{P}^\mathcal{A}_{\mathbf{Q}}=\left(\operatorname{Hor}^\mathcal{A}_\mathcal{Q}\right)^*\mathcal{P}_\mathcal{Q}$ and $\zeta:=\mathbb{J}\circ \mathcal{P}_\mathcal{Q}$.
\end{proposition}
\begin{proof}
Using the connection dependent isomorphisms
\[
U\in \mathfrak{aut}(P_S)\simeq \mathfrak{X}(S)\times\mathcal{F}_\mathcal{O}(P_S,\mathfrak{o})\ni \left([U],\mathcal{A}_S(U)\right)
\]
and
\[
\mathcal{P}\in T^*_\mathcal{Q}Q_{KK}\simeq \mathcal{F}_\mathcal{O}(P_S,\mathfrak{o})^*\times T_\mathbf{Q}^*\operatorname{Emb}(S,M)\ni (\zeta,\mathbf{P}^\mathcal{A}),
\]
we find
\begin{align}\label{intermiediate_computation}
&\left\langle\mathbf{J}_R(\mathcal{Q},\mathbf{P}^\mathcal{A}_{\mathbf{Q} },\zeta),([U],\mathcal{A}_S(U))\right\rangle=\left\langle\mathbf{J}_R(\mathcal{Q},\mathcal{P}),U\right\rangle=\langle\mathcal{P},\mathbf{d}\mathcal{Q}\circ U\rangle\nonumber\\
&\qquad=\left\langle\mathbf{P}^\mathcal{A}_{\mathbf{Q} },[\mathbf{d}\mathcal{Q}\circ U]\rangle+\langle \zeta,\mathcal{A}(\mathbf{d}\mathcal{Q}\circ U)\right\rangle\nonumber\\
&\qquad=\left\langle\mathbf{P}^\mathcal{A}_{\mathbf{Q} },\mathbf{d}\mathbf{Q}\circ [U]\right\rangle+\left\langle \zeta,\mathcal{A}(\mathbf{d}\mathcal{Q}\circ \operatorname{Hor}^{\mathcal{A}_S}[U]+\mathbf{d}\mathcal{Q}\circ\left(\mathcal{A}_S(U)\right)_{P_S})\right\rangle\nonumber\\
&\qquad=\left\langle\mathbf{P}^\mathcal{A}_{\mathbf{Q} },\mathbf{d}\mathbf{Q}\circ [U]\right\rangle+\left\langle \zeta,\mathcal{A}(\mathbf{d}\mathcal{Q}\circ \operatorname{Hor}^{\mathcal{A}_S}[U])\right\rangle+\left\langle \zeta,\mathcal{A}\left(\left(\mathcal{A}_S(U)\right)_{P}\circ\mathcal{Q}\right)\right\rangle\nonumber\\
&\qquad=\left\langle\mathbf{P}^\mathcal{A}_{\mathbf{Q} }\!\cdot\!\mathbf{d}\mathbf{Q}+\zeta\!\cdot\!(\mathcal{A}\circ \mathbf{d}\mathcal{Q}\circ \operatorname{Hor}^{\mathcal{A}_S}),[U]\right\rangle+\left\langle \zeta,\mathcal{A}_S(U)\right\rangle,
\end{align}
where $\langle\,,\rangle$ denotes the $L^2$ pairing.
Now we observe that
\begin{align*}
\mathcal{A}\circ \mathbf{d}\mathcal{Q}\circ \operatorname{Hor}^{\mathcal{A}_S}[U]&=\mathcal{A}\circ\mathbf{d}\mathcal{Q}\circ U-\mathcal{A}\circ\mathbf{d}\mathcal{Q}\circ(\mathcal{A}_S(U))_P\\
&=\mathcal{A}\circ\mathbf{d}\mathcal{Q}\circ U-\mathcal{A}\circ(\mathcal{A}_S(U)\circ\mathcal{Q})_P=\left(\mathcal{Q}^*\mathcal{A}-\mathcal{A}_S\right)(U)
\end{align*}
as required.
\end{proof}
\rem{ %%%%%%%%%%%%%%%%%%%%%%%%%%%%%%
The difference between two principal connection forms is a horizontal $\mathfrak{o}$-valued equivariant 1-form, so if we pair it with an equivariant function $\ze\in\F(P_S,\ou^*)$, we get a horizontal invariant 1-form, hence a pull-back of a 1-form on the base manifold $S$.
}  %%%%%%%%%%%%%%%%%%%%%%%%%%%%%%
\\
The left momentum map $\mathbf{J}_L$ is invariant under the cotangent lift of composition on the right by elements in $\mathcal{A}ut(P_S)$. Thus, the collective Hamiltonian $H:=h\circ\mathbf{J}_L$ is also right invariant and by Noether's theorem, we have
\[
\frac{d}{dt}\mathcal{P}_t\!\cdot \!\mathbf{d} \mathcal{Q}_t=0,
\]
where $\mathcal{P}_t\in T^*_{\mathcal{Q}_t}Q_{KK}$ is a solution of Hamilton's equations. This can be written as
\begin{equation}\label{conservation_law_theta}
\frac{d}{dt}\mathcal{P}_t^*\Theta_P=0,
\end{equation}
where $\Theta_P$ is the canonical one-form on $T^*P$ and $\mathcal{P}_t$ is interpreted as a map $\mathcal{P}_t:S\rightarrow T^*M$ covering $\mathcal{Q}_t$.
Introducing principal connections $\mathcal{A}$ and $\mathcal{A}_S$ on $P$ and $P_S$ we have the conservation laws
\[
\frac{d}{dt}\left(\mathbf{P}^\mathcal{A}_{\mathbf{Q} }\!\cdot\!\mathbf{d} \mathbf{Q}+\zeta\cdot\left(\mathcal{Q}^*\A-\A_S\right)\right)=0\quad\text{and}\quad \frac{d}{dt}\zeta=0,
\]
where we recall that $\mathbf{P}^\mathcal{A}_{\mathbf{Q}}=\left(\operatorname{Hor}^\mathcal{A}_\mathcal{Q}\right)^*\mathcal{P}_\mathcal{Q}$ and $\zeta:=\mathbb{J}\circ \mathcal{P}_\mathcal{Q}$.
The resulting dual pair is illustrated in the following diagram

\begin{picture}(150,100)(-70,0)%
\put(93,76){$T^* \operatorname{\!Emb}_\mathcal{O}(P_S,P)$}
%top label

\put(90,50){$\mathbf{J}_L$}
%left label

\put(160,50){$\mathbf{J}_R$}
%right arrow label

\put(65,15){$\mathfrak{aut}(P)^*$
%$\mathfrak{X}^{\ast}(M)\times\mathcal{F}^*(M,\mathfrak{o})$
}
%left bottom label

\put(165,16){$\mathfrak{aut}(P_S)^*$
%$\mathfrak{X}^{\ast}(S)\times\mathcal{F}^*(S,\mathfrak{o})$
}
%right bottom label

\put(130,70){\vector(-1, -1){40}}
% left slanted arrow

\put(135,70){\vector(1,-1){40}}
% right slanted arrow

\end{picture}\\
Note that in the particular case $S=M$ and $P_S=P$, we have $\operatorname{\!Emb}_\mathcal{O}(P,P)=\mathcal{A}ut(P)$, since equivariant embeddings of $P_S=P$ into $P$ are automorphisms. In this case, $\mathbf{J}_L$ is the usual Lagrange-to-Euler map, and $\mathbf{J}_R$ is the conserved momentum density. 

The trivial bundle case is treated in Appendix \ref{trivial_case}, where it is shown that the right momentum map has the expression
\[
\mathbf{J}_R\left(\mathcal{Q},\mathbf{P}^\mathcal{A}_{ \mathbf{Q} },\kappa_\theta\right)
=
\left(\mathbf{P}_{ \mathbf{Q} }\!\cdot\!\mathbf{d}\mathbf{Q}+\kappa_\theta\!\cdot\!\left(\mathbf{d}\theta-
\theta A_S\right),\theta^{-1}\kappa_\theta\right),
\]
which recovers \eqref{momap_right_trivial} when $A_S$ is the
trivial connection.

\begin{remark}[Physical interpretation of the connection $A_S$]{\rm
When $\mathcal{O}=S^1$, the quantity $A_S$ may be interpreted as a
magnetic potential that is localized on the subbundle $P_S$, where
the whole motion takes place. For example, if we think of a trivial
subbundle $P_S=S\times\mathcal{O}$ where ${\rm dim}S=2$, then the
charged particles moving on $S$ produce a current sheet carrying a
magnetic field $B_S={\bf d}A_S$ normal to the sheet. This suggestive
picture is consistent with the concept that the singular solutions
possess their own dynamics, independently of the properties of the
ambient space. Moreover, we recall that the first component of ${\bf
J}_R$ yields the Clebsch representation of a fluid variable in
$\mathfrak{aut}(P_S)^*$. In this case, the connection one-form $A_S$
on $S$ appears naturally in the Clebsch representation, due to the
principal bundle structure of the embedded subspace $P_S=S\times\O$.
}
\end{remark}

%%%%%%%%%%%%%%%%%%%%%%%%%%%%%%%%%%%%%%%%%%%%%%%%%%%%%%%%

%%%%%%%%%%%%%%%%%%%%%%%%%%%%%%%%%%%%%%%%%%%%%%%%%%%%%%%%%%%%%%%%%%%%%

\section{Incompressible EP$\!\mathcal{A}ut$ flows}\label{Sec:EPAut_vol}

At this point of the paper, a natural question concerns the
restriction of the above geodesic flows on the $\Aut(P)$ group to
the volume-preserving case. This question is motivated by
the fact that such construction yields a two-component
Euler system in which an incompressible fluid flow transports a
Yang-Mills charge, under the influence of an external magnetic
potential. This picture generalizes Arnold's well known construction
of Euler's equation as a geodesic on Diff$_{\rm vol}(M)$ to the
formulation of a geodesic fluid flow on $\Aut_{\rm vol}(P)$. In the
case of a trivial bundle $P=M\times\O$, one has $\Aut_{\rm
vol}(P)\simeq {\rm Diff}_{\rm
vol}(M)\,\circledS\,\mathcal{F}(M,\mathcal{O} )$ and the incompressible 
EP$\Aut$ equation (EP$\Aut_{\rm vol}$) is then a geodesic equation on a
semidirect-product Lie group. In this setting, natural questions
arise about how the geometric properties of ideal fluid flows
transfer to this more general situation. For example, the system can
be written in terms of the vorticity $\boldsymbol{\omega} ={\rm
curl}{\bf u}$ and one can ask how the dual pair construction of
\cite{MaWe83} applies to this case.

\subsection{The group of volume preserving automorphisms}

Let $\pi:P\rightarrow M$ be a \textit{right} principal $\O$-bundle and suppose that $M$ is orientable, endowed with a Riemannian metric $g$. Let $\mu_M$ be the volume form induced by $g$.

The group $\Aut_{\vol}(P)$ consists, by definition, of the automorphisms of the principal bundle $P$ which descend to volume preserving diffeomorphisms of the base manifold $M$ with respect to the volume form $\mu_M$. In other words, given an automorphism $\varphi\in\mathcal{A}ut(P)$, we have by definition
\[
\varphi\in \Aut_{\vol}(P)\Leftrightarrow \bar{\varphi}\in\operatorname{Diff}_{\vol}(M),
\]
where $\operatorname{Diff}_{\vol}(M)$ is the group of volume preserving diffeomorphisms of $M$. Its Lie algebra, denoted by $\mathfrak{aut}_{\vol}(P)$ consists of equivariant vector fields $U$ such that their projection is divergence free, that is, for $U\in\mathfrak{aut}(P)$ we have
\[
U\in \mathfrak{aut}_{\vol}(P)\quad\Leftrightarrow\quad [U]\in\mathfrak{X}_{\vol}(M).
\]
Hence, given a principal connection $\mathcal{A}$ on $P$, the Lie algebra $\mathfrak{aut}_{\vol}(P)$ decomposes under the isomorphism \eqref{isomorphism_A} as
\begin{equation}\label{akau}
\mathfrak{aut}_{\vol}(P)\rightarrow\mathfrak{X}_{\vol}(M)\times\mathcal{F}_\mathcal{O}(P,\mathfrak{o}),\quad U\mapsto ([U],\mathcal{A}(U)).
\end{equation}

An inner product $\gamma$ on $\mathfrak{o}$ being also given, one defines the Kaluza-Klein Riemannian metric on $P$ as
\[
\ka(U_p,V_p)=g(T\pi(U_p),T\pi(V_p))+\ga(\A(U_p),\A(V_p)),\quad U_p, V_p\in T_pP.
\]
The induced volume form $\mu_P$ on $P$ is
\[
\mu_P:=\pi^*\mu_M\wedge\mathcal{A}^*\rm{det} _\ga,
\]
where $\mathcal{A}^*\det_\ga$ denotes the pullback by the connection $\mathcal{A}:TP\to\mathfrak{o}$ of the canonical determinant form induced by $\ga$ on $\mathfrak{o}$, that is $\det_\ga\in\wedge^k\mathfrak{o}^*$, $k={\dim\mathfrak{o}}$, defined by
\[
\operatorname{det}_{\ga}(\xi_1,...,\xi_k):=\sqrt{\operatorname{det}(\gamma(\xi_i,\xi_j))}.
\]
where $\xi_1,\dots, \xi_k\in\mathfrak{o}$ is a positively oriented basis.

We now suppose that $\gamma$ is $\operatorname{Ad}$-invariant. In this case the Kaluza-Klein metric $\ka$ and the volume form $\mu_P$ are $\mathcal{O}$-invariant.
More general $\O$-invariant metrics on $P$ are considered in \cite{Molitor}.

\begin{lemma}\label{volume_preserving_aut} Let $\varphi\in\mathcal{A}ut(P)$ be a principal bundle automorphism. Then we have the equivalence
\[
\varphi\in\mathcal{A}ut_{\vol}(P)\Leftrightarrow \varphi^*\mu_P=\mu_P.
\]
\end{lemma}
\begin{proof} The pull-back $\mathcal{A}_\ph:=\ph^*\mathcal{A}$ of a principal connection $\A$ by an automorphism $\ph$ of $P$ is a principal connection. Therefore, the $\O$-invariant form $\A_\ph^*\det_\ga-\A^*\det_\ga\on\in \Om^k(P)$ is horizontal, so it is the pull-back $\pi^*\al$ for some $\al\in\Om^k(M)$.
Denoting by $\bar\ph$ the diffeomorphism of $M$ induced by $\ph$,
we compute
\begin{align*}
\ph^*\mu_P&=\ph^*(\pi^*\mu_M\wedge\A^*{\det}_\ga) =\ph^*\pi^*\mu_M\wedge\A_\ph^*{\det}_\ga\\
&=\pi^*\bar\ph^* \mu_M\wedge(\A^*{\det}_\ga+\pi^*\al)
=\pi^*\bar\ph^* \mu_M\wedge\A^*{\det}_\ga.
\end{align*}
With this formula, it is clear that any $\ph\in\mathcal{A}ut_{\vol}(P)$ preserves the volume form $\mu_P$.

Conversely, if $\ph$ preserves $\mu_P$, then we have
\[
\pi^*\bar\ph^* \mu_M\wedge\A^*{\det}_\ga=\pi^* \mu_M\wedge\A^*{\det}_\ga.
\]
Plugging in the horizontal lifts of $n=\dim M$ arbitrary vector fields $\mathbf{u}_1,\dots, \mathbf{u}_n$, we obtain
\[
\pi^*((\bar\ph^*\mu_M-\mu_M)(\mathbf{u}_1,\dots,\mathbf{u}_n))\A^*{\det}_\ga=0,
\]
hence $\bar\ph^* \mu_M=\mu_M$.
\end{proof}

\medskip

As a consequence, the Lie algebra $\mathfrak{aut}_{\vol}(P)$ coincides with the Lie algebra of equivariant divergence free vector fields with respect to the volume form $\mu_P$ induced by the Kaluza-Klein metric on $P$.

\medskip

We will also need the following result.

\begin{lemma}\label{integration_invariant} Assume that $ \mathcal{O} $ is compact and consider an $ \mathcal{O} $-invariant function $f$ on $P$ with compact support. Then we have the formula
\begin{equation}\label{integra}
\int_Pf(p)\mu_P=\operatorname{Vol}(\O)\int_M\tilde f(m)\mu_M,
\end{equation}
where $\operatorname{Vol}(\O)$ denotes the Riemannian volume of $ \mathcal{O} $ relative to the bi-invariant metric induced by $\gamma$, and $\tilde f$ is the function induced on $M$ by the formula $\tilde f\circ \pi= f$.
\end{lemma}
\begin{proof} It is enough to check the formula for a trivial bundle $P=M\x\O$.
In this case the connection 1-form is defined with a 1-form $A\in\Om^1(M,\mathfrak{o})$ by
$\mathcal{A}(v_x,\xi_g)=\Ad_{g^{-1}}(A(v_x)+\ka^l(\xi_g))$, where $\ka^l(\xi_g)=\xi_gg^{-1}\in\Om^1(\O,\mathfrak{o})$ denotes the left Maurer-Cartan form on the Lie group $\O$. Note that $\mu_{ \mathcal{O} }:=(\ka^l)^*{\rm det}_\ga\in\Om^k(\O)$ is the volume form induced by the bi-invariant
Riemannian metric on $\O$ associated to $\ga$. Let $\pi_2:M\x\O\to\O$ denote the projection on the second factor. Taking into account the $\O$-invariance of the inner product $\ga$ we obtain
\[
\mu_P=\pi^*\mu_M\wedge\mathcal{A}^*{\rm det} _\ga
=\pi^*\mu_M\wedge\pi_2^*\left((\ka^l)^*{\rm det}_\ga\right)
=\pi^*\mu_M\wedge\pi_2^*\mu_\O.
\]
Now, formula \eqref{integra} follows for every compactly supported smooth function $\tilde f$ on $M$.
\end{proof}

\subsection{Dynamics on a trivial principal bundle}\label{colorado}

If the principal bundle is trivial, then we have the group isomorphism $\mathcal{A}ut_{\vol}(P)\simeq\operatorname{Diff}_{\vol}(M)\,\circledS\,\mathcal{F}(M,\mathcal{O})$.
We now give the expression of the Euler-Poincar\'e equations on $\mathcal{A}ut_{\vol}(P)$ in the case of a trivial bundle. Since a volume form $\mu_M$ has been fixed, it is not necessary to include densities in the dual Lie algebras. The regular dual to $\mathfrak{X}_{\vol}(M)\,\circledS\,\mathcal{F}(M,\mathfrak{o})$ is identified with $\left(\Omega^1(M)/\mathbf{d}\mathcal{F}(M)\right)\times\mathcal{F}(M,\mathfrak{o}^*)$ via the $L^2$-pairing given by the volume form $\mu_M$. Here $\Omega^1(M)/\mathbf{d}\mathcal{F}(M)$ denotes the space of one-forms modulo exact one-forms. As in Proposition \ref{EPaut_trivial}, but this time considering a Lagrangian $l$ on the semidirect product $\X_{\vol}(M)\,\circledS\,\F(M,\mathfrak{o})$, we get the following

\begin{proposition}[The $\operatorname{EP}\!\mathcal{A}ut_{\vol}$ equations on a trivial principal bundle] The Eu\-ler-Poincar\'e equations on the group of volume preserving automorphisms of a trivial principal bundle are written as
\begin{equation}\label{EPAut_vol_trivial}
\left\{\begin{array}{l}
\vspace{0.2cm}\displaystyle\frac{\partial}{\partial t}\frac{\delta l}{\delta\mathbf{u}}+\pounds_{\mathbf{u}}\frac{\delta l}{\delta\mathbf{u}}+\frac{\delta l}{\delta{\nu}}\!\cdot\!\mathbf{d}{\nu}=-\mathbf{d}p,\quad\operatorname{div}(\mathbf{u})=0\\
\displaystyle\frac{\partial}{\partial t}\frac{\delta l}{\delta{\nu}}+\mathbf{d}\frac{\delta l}{\delta{\nu}}\!\cdot\!\mathbf{u}+\operatorname{ad}^*_{{\nu}}\frac{\delta l}{\delta{\nu}}=0,
\end{array}\right.
\end{equation}
where $\pounds_\mathbf{u}$ denotes the Lie-derivative of a one-form (and not a one-form density), and the pressure function $p$ is determined by the incompressibility condition.
\end{proposition}
Here, the first equation is written on $\Om^1(M)$ and not on the quotient space $\Om^1(M)/d\F(M)$.
As in the compressible case, one can use Kaluza-Klein Lagrangians of the form \eqref{KK_Lagrangian_trivial}, relative to differential operators $Q_1:\mathfrak{X}(M)\rightarrow\mathfrak{X}(M)$ and $Q_2:\mathcal{F}(M,\mathfrak{o})\rightarrow\mathcal{F}(M,\mathfrak{o}^*) $, that is,
\begin{equation}\label{KK_Lagrangian_incomp}
l(\mathbf{u},{\nu})=\frac{1}{2}\int_Mg(\mathbf{u},Q_1\mathbf{u})\mu_M+\frac{1}{2}\int_M\left\langle A\cdot\mathbf{u}+{\nu} ,Q_2(A\cdot\mathbf{u}+{\nu}) \right\rangle\mu_M,
\end{equation}
where $\sf g$ is the given Riemannian metric on $M$ and $\mu_M$ is the Riemannian volume.

\paragraph{Other duality pairings.} Since a Riemannian metric $\sf g$
is given, we can alternatively use the $L^2$-paring induced by $\sf g$.
In this case, we can choose the dual space
$\mathfrak{X}_{\vol}(M)\times \mathcal{F}(M,\mathfrak{o}^*)$ and the
$\operatorname{EP}\!\mathcal{A}ut$ equations read
\begin{equation}
\left\{\begin{array}{l}
\vspace{0.2cm}\displaystyle\frac{\partial}{\partial t}\frac{\delta l}{\delta\mathbf{u}} +
\nabla_\mathbf{u}\frac{\delta l}{\delta\mathbf{u}} +
\nabla\mathbf{u}^\mathsf{T}\cdot\frac{\delta l}{\delta\mathbf{u}} +
\left(\frac{\delta
l}{\delta{\nu}}\cdot\mathbf{d}{\nu}\right)^\sharp
= -\operatorname{grad}p,
\quad
\operatorname{div}(\mathbf{u})=0\\
\displaystyle\frac{\partial}{\partial t}\frac{\delta l}{\delta{\nu}}+\mathbf{d}\frac{\delta l}{\delta{\nu}}\!\cdot\!\mathbf{u}+\operatorname{ad}^*_{{\nu}}\frac{\delta l}{\delta{\nu}}=0,
\end{array}\right.
\end{equation}
where $\nabla$ denotes the Levi-Civita connection, $\sharp$ is the
sharp-operator associated to the Riemannian metric and
$\operatorname{grad}p:=(\mathbf{d}p)^\sharp$.

Alternatively, if $H^1(M)=\{0\}$, one can also identify the dual space to $\mathfrak{X}_{\vol}(M)$ with exact two-forms $\mathbf{d}\Omega^1(M)$, through the duality pairing
\begin{equation}\label{domeg}
\left\langle\omega,\mathbf{u}\right\rangle=\int_M(\alpha\cdot\mathbf{u})\mu_M,\quad\text{where}\quad\mathbf{d}\alpha=\omega.
\end{equation}
In this case, we get the first equation in vorticity representation as
\begin{equation}\label{EPAut_vol_vorticity}
\left\{\begin{array}{l}
\vspace{0.2cm}\displaystyle\frac{\partial}{\partial t}\frac{\delta l}{\delta\mathbf{u}}+\pounds_{\mathbf{u}}\frac{\delta l}{\delta\mathbf{u}}+\mathbf{d}\frac{\delta l}{\delta{\nu}}\wedge\mathbf{d}{\nu}=0\\
\displaystyle\frac{\partial}{\partial t}\frac{\delta l}{\delta{\nu}}+\mathbf{d}\frac{\delta l}{\delta{\nu}}\!\cdot\!\mathbf{u}+\operatorname{ad}^*_{{\nu}}\frac{\delta l}{\delta{\nu}}=0,
\end{array}\right.
\end{equation}
where the wedge product involves a contraction between
$\mathfrak{o}$ and $\mathfrak{o}^*$-valued forms. It is important to
recall that the functional derivatives appearing above depend on the
chosen pairing. For example, in \eqref{EPAut_vol_vorticity} and with the Lagrangian \eqref{KK_Lagrangian_incomp} we have
\[
\omega:=\frac{\delta l}{\delta\mathbf{u}}=\mathbf{d}\left((Q_1\mathbf{u})^\flat+\sigma \cdot A\right) \in\mathbf{d}\Omega^1(M),\quad \sigma  :=\frac{\delta l}{\delta{\nu}}=Q_2(A\cdot\mathbf{u}+{\nu})\in \mathcal{F}(M,\mathfrak{o}^*).
\]

An interesting special case of the above equations comes out
when $M$ is the Euclidean space $\mathbb{R}^3$. In this case, the exact two-form $\omega=\mathbf{d}\alpha$ can be identified with the divergence free vector field ${\boldsymbol{\omega}}=\operatorname{curl}\boldsymbol{\alpha}$, via the relations $\boldsymbol{\omega}=(\star \omega)^\sharp$ and $\boldsymbol{\alpha}=\alpha^\sharp$, where $\star$ denotes the Hodge star operator associated to the Euclidean metric. Using this identification, the $\operatorname{EP}\!\mathcal{A}ut$ equations \eqref{EPAut_vol_vorticity} can be written as
\begin{equation}\label{EPAut_vol_vorticity_3D}
\left\{\begin{array}{l} \vspace{0.2cm}\displaystyle\frac{\partial
{\boldsymbol{\omega}}}{\partial t} +
\left(\mathbf{u}\cdot\nabla\right){\boldsymbol{\omega}} -
\left({\boldsymbol{\omega}}\cdot\nabla\right)\mathbf{u}+
\operatorname{curl}\left\langle\boldsymbol{\mu}\,,\nabla{\nu}\right\rangle=0
\\
\displaystyle\frac{\partial\sigma  }{\partial
t}+\left(\mathbf{u}\cdot\nabla\right)\sigma  +\operatorname{ad}^*_{{\nu  }}\sigma  =0.
\end{array}\right.
\end{equation}
The first equation can be obtained from \eqref{EPAut_vol_vorticity} by using the equalities
\[
\left(\star(\pounds_\mathbf{u}\boldsymbol{\omega})\right)^\sharp=\left(\star \mathbf{d}(\mathbf{i}_\mathbf{u}\boldsymbol{\omega})\right)^\sharp=\left(\star\mathbf{d}({\boldsymbol{\omega}}\times \mathbf{u})^\flat\right)^\sharp=\operatorname{curl}({\boldsymbol{\omega}}\times \mathbf{u})=\left(\mathbf{u}\cdot\nabla\right){\boldsymbol{\omega}} -
\left({\boldsymbol{\omega}}\cdot\nabla\right)\mathbf{u}.
\]
In this case, the duality pairing is given by
\[
\left\langle{\boldsymbol{\omega}},\mathbf{u}\right\rangle=\int_{\mathbb{R}^3}(\mathbf{v}\cdot\mathbf{u}){\rm d}x,\quad\text{where}\quad \operatorname{curl}\mathbf{v}={\boldsymbol{\omega}}.
\]
For example, taking the Kaluza-Klein Lagrangian \eqref{KK_Lagrangian_incomp}, the corresponding $\operatorname{EP}\!\mathcal{A}ut$ equations are obtained by inserting the functional derivatives
\[
{\boldsymbol{\omega}}=\frac{\delta l}{\delta\mathbf{u}}=\operatorname{curl}(Q_1\mathbf{u})+\operatorname{curl}\left(\mu \cdot A\right)\quad\text{and}\quad\sigma  =\frac{\delta l}{\delta{\nu}}=Q_2(A\cdot\mathbf{u}+{\nu})
\]
into \eqref{EPAut_vol_vorticity_3D}. The two dimensional case is presented in Appendix \ref{2d}.

\begin{remark}{\rm 
The above equations generalize an important regularization model in fluid dynamics, which is known under the name of Euler-$\alpha$
(\cite{HoMaRa98}). This model is derived by specializing the above equations to the particular case when $Q_1=1-\alpha^2\Delta$ and
${\nu}\equiv 0$, the latter condition being preserved by
the flow. Thus, the complete system of the EP$\Aut_{\rm vol}$
equations generalizes the Euler-$\alpha$ equation to the case of a
Yang-Mills charged fluid moving under the influence of an external
magnetic field. 
\rem{ %%%%%%%%%%%%%%%%%%%%%%%%
For the special case when $\mathcal{O}=S^1$, this
system is of relevance in plasma physics. In this case, the density
$\sigma  =\rho$ is the electric charge density. This is a typical
situation, for example, in the case of non-neutral plasmas.
}    %%%%%%%%%%%%%%%%%%%%%%%%
}
\end{remark}

\subsection{Dynamics on a non trivial principal bundle}

By \eqref{akau} the regular dual of $\mathfrak{aut}_{\vol}(P)$ can be identified with the cartesian product $\left(\Om^1(M)/\mathbf{d}\mathcal{F}(M)\right)\times \Ga(\Ad^*P)$, via the $L^2$ pairing
\[
\langle([\be],{\zeta}),(\mathbf{u},\tilde\om)\rangle
=\int_M(\be(\mathbf{u})+{\zeta}\cdot\tilde\om)\mu_M.
\]
Given a Lagrangian $l:\mathfrak{aut}_{\vol}(P)\rightarrow\mathbb{R}$ and a principal connection $\mathcal{A}$, the associated EP$\mathcal{A}ut_{\vol}$ equations are given in the following

\begin{proposition}[The general $\operatorname{EP}\mathcal{A}ut_{\vol}$ equations on principal bundles] The Eu\-ler-Poincar\'e equations on the group of volume preserving automorphisms of a principal bundle are written as
\begin{equation}\label{EPAut_vol_Adjoint}
\left\{\begin{array}{l}
\vspace{0.2cm}\displaystyle\frac{\partial}{\partial t}\frac{\delta l^\mathcal{A}}{\delta\mathbf{u}}+\pounds_\mathbf{u}\frac{\delta l^\mathcal{A}}{\delta\mathbf{u}}+\frac{\delta l^\mathcal{A}}{\delta\tilde\omega}\!\cdot\!\left(\nabla^\mathcal{A}\tilde{\omega}+\mathbf{i}_\mathbf{u}\tilde{\mathcal{B}}\right)=-\mathbf{d}p,\quad\operatorname{div}\mathbf{u}=0\\
\displaystyle\frac{\partial}{\partial t}\frac{\delta l^\mathcal{A}}{\delta\tilde\omega}+\nabla_{\mathbf{u}}^\mathcal{A}\frac{\delta l^\mathcal{A}}{\delta\tilde{\omega}}+\operatorname{ad}^*_{\tilde\omega}\frac{\delta l^\mathcal{A}}{\delta\tilde{\omega}}=0,
\end{array}\right.
\end{equation}
where $\pounds_u$ denotes the Lie derivative acting one one-forms and the pressure $p$ is determined by the incompressibility condition.
\end{proposition}
\begin{proof} The equations can be obtained as in \S\ref{General_EPAut} by splitting the general formulation
\[
\frac{\partial}{\partial t}\frac{\delta l}{\delta U}+\pounds_U\frac{\delta l}{\delta U}=0,
\]
and using the connection dependent isomorphism. Then we deduce \eqref{EPAut_vol_Adjoint} from \eqref{EPAut_Adjoint} by simply adding the incompressibility condition and using the fixed Riemannian volume $\mu_M$ to identify one-form densities with one-forms, and sections in $\Gamma(\operatorname{Ad}^*P)\otimes\operatorname{Den}(M)$ with sections in $\Gamma(\operatorname{Ad}^*P)$.
Then one obtains the second equation since $\pounds_u^\A$ and $\nabla_u^\A$ coincide when applied to functions.

Alternatively, it is instructive to derive the equations directly by using the expression of the infinitesimal coadjoint action $\operatorname{ad}^*$ on $\left(\Om^1(M)/\mathbf{d}\mathcal{F}(M)\right)\times \Ga(\Ad^*P)$. In order to do this, we first need the expression of the Lie bracket on $\mathfrak{X}_{\vol}(M)\times\Gamma(\Ad P)$.
Under the isomorphism \eqref{isomorphism_A}, the Lie bracket is
\begin{equation}\label{deco}
\left[(\mathbf{u},\tilde\om),(\mathbf{v},\tilde\th)\right]_L=\left([\mathbf{u},\mathbf{v}]_L,
[\tilde\om,\tilde\th]+\nabla^\mathcal{A}_\mathbf{v}\tilde\om-\nabla^\mathcal{A}_\mathbf{u}\tilde\th+\tilde{\mathcal{B}}(\mathbf{u},\mathbf{v})\right),
\end{equation}
where $\tilde{\mathcal{B}}\in\Omega^2(M,\operatorname{Ad}P)$ is the reduced curvature. Then it suffices to compute the associated infinitesimal coadjoint action. This is done in the following lemma.
\end{proof}

\begin{lemma}
The infinitesimal coadjoint action of $\mathfrak{aut}_{\vol}(P)$ on its regular dual, using the decomposition provided by a principal connection $\mathcal{A}$, can be written as
\begin{equation}
\ad^*_{(\mathbf{u},\tilde\om)}([\be],\al)
=\left([\pounds_\mathbf{u}\be+\al\cdot(\nabla^\mathcal{A}\tilde\om+i_\mathbf{u}\tilde{\mathcal{B}})], \ad_{\tilde\om}^*\al+\nabla^\mathcal{A}_\mathbf{u} \al\right),
\end{equation}
for $(\mathbf{u},\tilde\om)\in \mathfrak{X}_{\vol}(M) \times\Gamma(\operatorname{Ad}P)$
and $([\be],\al)\in \left(\Om^1(M)/\mathbf{d}\mathcal{F}(M)\right)\times \Gamma(\operatorname{Ad}^*P)$.
\end{lemma}

\begin{proof}
Using the expression \eqref{deco} of the bracket on $\mathfrak{aut}_{\vol}(P)$ under the decomposition \eqref{akau}, we compute
\begin{align*}
\langle \ad^*_{(\mathbf{u},\tilde\om)}([\be],\al),(\mathbf{v},\tilde\th)\rangle
=&\left\langle[\be],[\mathbf{u},\mathbf{v}]_L\right\rangle
+\left\langle\al,[\tilde\om,\tilde\th]+\nabla^\mathcal{A}_\mathbf{v}\tilde\om-\nabla^\mathcal{A}_\mathbf{u}\tilde\th+\tilde{\mathcal{B}}(\mathbf{u},\mathbf{v})\right\rangle\\
=&-\int_M\be([\mathbf{u},\mathbf{v}])\mu_M
+\int_M\al\cdot(\ad_{\tilde\om}\tilde\th)\mu_M
+\int_M\left(\al\cdot \nabla^\mathcal{A}_\mathbf{v}\tilde\om\right)\mu_M\\
&-\int_M\al\cdot \left(\nabla^\mathcal{A}_{\mathbf{u}}\tilde\th\right)\mu_M
+\int_M\left(\al\cdot i_\mathbf{u}\tilde{\mathcal{B}}\right)(\mathbf{v})\mu_M\\
=&\int_M\pounds_\mathbf{u}\be(\mathbf{v})
\mu_M
+\left\langle\al\cdot\left(\nabla^\mathcal{A}\tilde\om+\mathbf{i}_\mathbf{u}\tilde{\mathcal{B}}\right),\mathbf{v}\right\rangle\\
&+\int_M\left(\ad^*_{\tilde\om}\al+\nabla^\mathcal{A}_\mathbf{u}\al\right)\cdot\tilde\th\mu_M
\\
=&\left\langle\left([\pounds_\mathbf{u}\be+\al\cdot (\nabla^\mathcal{A}\tilde\om+\mathbf{i}_\mathbf{u}\tilde{\mathcal{B}}) ], \ad_{\tilde\om}^*\al+ \nabla^\mathcal{A}_\mathbf{u}\al\right),(\mathbf{v},\tilde\th)\right\rangle,
\end{align*}
where we use that $\div\mathbf{u}=0$.
\end{proof}

\medskip

\begin{remark}[Kaluza-Klein Lagrangians]
{\rm As in the compressible
case, an important class of Lagrangians is provided by the
Kaluza-Klein expression \eqref{KK_Lagrangian}, thus obtaining the incompressible version of \eqref{kkl}. 
In the case of the $L^2$ Kaluza-Klein Lagrangian, 
%that is $Q_1=id$, $Q_2=id$, 
the EP$\mathcal{A}ut_{\vol}$ equation takes the simple form
\[
\left\{\begin{array}{l}
\vspace{0.2cm}\displaystyle\partial_t\mathbf{u}+\nabla_\mathbf{u}\mathbf{u}+\bar\ga\left(\tilde\om,\mathbf{i}_\mathbf{u}\tilde{\mathcal{B}}\right)^\sharp=-\operatorname{grad}p,\quad\operatorname{div}\mathbf{u}=0\\
\displaystyle\partial_t\tilde\omega+\nabla_\mathbf{u}^\mathcal{A}\tilde\omega=0.
\end{array}\right.
\]
These equations describe the motion of an ideal fluid moving in an external Yang-Mills field $\mathcal{B}$, where $\tilde\omega$ denotes
the magnetic charge (\cite{Vi08}).
For the more general case with generic metrics
$Q_1$ and $Q_2$, the  EP$\mathcal{A}ut_{\vol}$ equation generalizes the
Euler-$\alpha$ model (\cite{HoMaRa98}) to the case of a Yang-Mills
charged fluid, whose configuration manifold is a non-trivial
principal bundle. Again, the fluid moves under the influence of an
external Yang-Mills magnetic field, which is given by the curvature
$\tilde{\mathcal{B}}$ as usual.
}
\end{remark}

%%%%%%%%%%%%%%%%%%%%%%%%%
%%%%%%%%%%%%%%%%%%%%%%%%%

\section{Clebsch variables and momentum maps}\label{MomapsForEPAut_vol}

This section presents the dual pair of momentum maps underlying   EP$\mathcal{A}ut_{\bf vol\,}$ flows. This extends the dual pair structure underlying ideal incompressible fluid flows, as it was shown by  \cite{MaWe83}. The next section reviews briefly the Marsden-Weinstein dual pair, whose details are given in Appendix \ref{App:MWDP}.

\subsection{The Marsden-Weinstein dual pair}
In the paper \cite{MaWe83}, the authors present a pair of momentum maps
 that apply to Euler's equation $\partial_t\boldsymbol\omega+\operatorname{curl}(\boldsymbol\omega\times\mathbf{u})=0$ for the fluid vorticity $\boldsymbol\omega=\operatorname{curl}\mathbf{u}$.  In
 particular, these momentum maps provide the Clebsch representation
 of the fluid vorticity
 $\operatorname{curl}\mathbf{u}$ on the configuration manifold $S$ in terms of
 canonical variables taking values in a symplectic
 manifold $M$. Again, one constructs the dual pair of momentum maps
 
 \begin{picture}(150,100)(-70,0)%
\put(105,75){$\Emb(S,M)$}
%top label

\put(90,50){$\mathbf{J}_L$}
%left label

\put(160,50){$\mathbf{J}_R$}
%right arrow label

\put(70,15){$\mathcal{F}(M)^*$
%$\mathfrak{X}^{\ast}(M)\times\mathcal{F}^*(M,\mathfrak{o})$
}
%left bottom label

\put(160,15){$\mathfrak{X}_{\vol}(S)^*$
%$\mathfrak{X}^{\ast}(S)\times\mathcal{F}^*(S,\mathfrak{o})$
}
%right bottom label

\put(130,70){\vector(-1, -1){40}}
% left slanted arrow

\put(135,70){\vector(1,-1){40}}
% right slanted arrow

\end{picture}\\
 whose  right leg provides the Clebsch representation of the vorticity on $S$. A rigorous construction of these momentum maps and the proof of the dual pair property were given in \cite{GBVi2011}.
 
 In particular,   $S$ is a compact manifold with volume form $\mu_S$ while $(M,\omega)$ is an exact (and hence noncompact) symplectic manifold, with $\omega=-\mathbf{d}\theta$. 
The Marsden-Weinstein dual pair (\cite{MaWe83}) is associated to the
action of the groups $\operatorname{Diff}_{\vol}(S)$ and
$\operatorname{Diff}_{\ham}(M)$ on $\operatorname{Emb}(S,M)$ by
composition on the right and on the left, respectively. Here
$\operatorname{Diff}_{\ham}(M)$ is the group of Hamiltonian diffeomorphisms
of $(M,\omega)$. The momentum map associated to the right action reads
\begin{equation}\label{santiago}
\mathbf{J}_R(f)=[f^*\theta]\in\Omega^1(S)/\mathbf{d}\mathcal{F}(S)
\,,
\end{equation}
or 
\[
\mathbf{J}_R(\boldsymbol{Q},\boldsymbol{P})=[(\boldsymbol{Q},\boldsymbol{P})^*\Theta_{\mathbb{R}^3}]=[\boldsymbol{P}\cdot
\mathbf{d}\boldsymbol{Q}]=\operatorname{curl}\!\left(\nabla
\boldsymbol{Q}^T\!\cdot\boldsymbol{P}\right).
\]
when $M=\Bbb{R}^{2k}$ and $S=\Bbb{R}^3$. (Here we have used  the duality pairing arguments in
\S\ref{colorado} to identify an exact two form $\omega$ with a
divergence-free vector field). 
Notice that the  corresponding
 variables also appear in the left leg
\begin{equation}\label{J_L_MW}
\mathbf{J}_L(f)=\int_S\delta(n-f(x))\mu_S\in\mathcal{F}(M)^*,
\end{equation} 
or, in local coordinates,
\[
\mathbf{J}_L(\boldsymbol{Q},\boldsymbol{P})=\int_S\delta(\boldsymbol{q}-\boldsymbol{Q}(s))\delta(\boldsymbol{p}-\boldsymbol{P}(s))\,\mu_S\in\mathcal{F}(\Bbb{R}^{2k})^*.
\]
This is due to the fact that Clebsch variables are conjugate variables
 taking values in $M$.
However, in their work \cite{GBVi2011} proved that the dual pair structure requires replacing the group $\operatorname{Diff}_{\ham}(M)$ by a central extension of the type $\operatorname{Diff}_{\ham}(M)\times_B\mathbb{R}$, where $B$ is a group two cocycle. This central extension of $\operatorname{Diff}_{\ham}(M)$ first appeared in \cite{IsLoMi2006}, who showed how this is isomorphic to 
the quantomorphism group $\mathcal{Q}uant(T^*M\times S^1)$, well known in quantization problems.
 
Also, if the symplectic form on $S$ is not exact, then one needs in addition the Ismagilov extension 
of the group of volume preserving diffeomorphisms of $S$. 
In what follows we consider only exact symplectic manifolds, so there will be no central extension
needed for the right leg of the dual pair.

The purpose of this section is to present the $\O$-equivariant variant of the Marsden-Weinstein dual pair,
which is the dual pair of momentum maps that apply to the EP$\mathcal{A}ut_{\vol}$ equation.

\subsection{Clebsch variables for incompressible EP$\!\mathcal{A}ut$ flows}\label{Clebsch_EPAutvol}

After the preceding review of the momentum maps underlying the
Clebsch representation of Euler's fluid vorticity, we are now ready
to generalize that construction to our principal bundle setting. In
particular, we shall characterize the momentum maps underlying the
Clebsch representation of the EP$\Aut_{\rm vol}$ system.  

We recall
from the Clebsch representation of Euler's vorticity that Clebsch
variables belong to a space of embeddings $S\hookrightarrow M$, where $S$ is the
fluid particle configuration manifold and $M$ is a symplectic
manifold. For simplicity, this section will introduce the
fundamental concepts in the simple case when $S=\mathbb{R}^3$, so that
the fluid moves in the ordinary physical space. In the case of
EP$\Aut_{\rm vol}$, the equations describe the evolution of a
Yang-Mills charged fluid under the influence of an external magnetic
field. Thus, it is natural to identify $M$ with the usual
Yang-Mills phase space $T^*\bar{P}$, where $\bar{P}$ is a principal
$\mathcal{O}$-bundle.

As a further simplification, this section will describe the case of
a trivial bundle $\bar{P}=\mathbb{R}^k\times\mathcal{O}$ with
$\mathcal{O}\subseteq GL(n,\Bbb{R})$ being a matrix Lie group. Thus, we are led to the
following definition on Clebsch variables:
\[
\left(\boldsymbol{Q},\boldsymbol{P},\sigma ,\theta\right)
\,\in \,\operatorname{Emb}(\mathbb{R}^3,\mathbb{R}^{2k}\times
\mathfrak{o}^*)\times\mathcal{F}(\Bbb{R}^3,\mathcal{O}),
\]
where $\operatorname{Emb}(\mathbb{R}^3,\mathbb{R}^{2k}\times
\mathfrak{o}^*)\times\mathcal{F}(\Bbb{R}^3,\mathcal{O})$ is
naturally endowed with the Poisson structure:
\begin{multline}\label{PBB2}
\{F,G\}(\boldsymbol{Q},\boldsymbol{P},\sigma ,\theta)=\int\frac1w\left(\frac{\delta
F}{\delta \boldsymbol{Q}}\cdot\frac{\delta G}{\delta
\boldsymbol{P}}-\frac{\delta G}{\delta
\boldsymbol{Q}}\cdot\frac{\delta F}{\delta
\boldsymbol{P}}\right)\,{\rm d}^3x
\\
+ \int\frac1w\left(
 \left\langle\sigma ,\left[\frac{\delta F}{\delta
\sigma },\frac{\delta G}{\delta \sigma }\right]\right\rangle +
\left\langle\frac{\delta F}{\delta \theta},\frac{\delta G}{\delta
\sigma }\theta\right\rangle-\left\langle\frac{\delta G}{\delta
\theta},\frac{\delta F}{\delta \sigma }\theta\right\rangle\right)\,{\rm
d}^3x,
\end{multline}
where $w(x)\,{\rm d}^3x$ is a fixed volume form on  $\mathbb{R}^3$, 
and the angle bracket $\langle\cdot,\cdot\rangle$ denotes the trace
pairing $\left\langle
A,B\right\rangle=\operatorname{Tr}\left(A^TB\right)$. At this point, one observes that

\begin{theorem}\label{rightlegtrivial}
The infinitesimal action of
volume-preserving automorphisms
$(\mathbf{u},\zeta)\in\mathfrak{aut}_{\rm
vol}(\mathbb{R}^3\times\mathcal{O})\simeq\mathfrak{X}_{\rm
vol}(\mathbb{R}^3)\,\circledS\,\mathcal{F}(\mathbb{R}^3,\mathfrak{o})$
\begin{equation}\label{LieAlgAction}
\left(\mathbf{u},\zeta\right)_{\operatorname{Emb}(\mathbb{R}^3,\mathbb{R}^{2k}\times
\mathfrak{o}^*)\times\mathcal{F}(\Bbb{R}^3,\mathcal{O})}
\left(\boldsymbol{Q},\boldsymbol{P},\sigma ,\theta\right) =\Big(
(\mathbf{u}\cdot\nabla)\boldsymbol{Q},
(\mathbf{u}\cdot\nabla)\boldsymbol{P},
(\mathbf{u}\cdot\nabla)\sigma ,(\mathbf{u}\cdot\nabla)\theta+\theta\zeta \Big)
\end{equation}
yields the right-action momentum map
\[
\mathbf{J}_R\left(\boldsymbol{Q},\boldsymbol{P},\sigma ,\theta\right)=
\Big(\!\operatorname{curl}\!\left(\nabla\boldsymbol{Q}^T\cdot\boldsymbol{P}
+\left\langle\sigma ,\nabla\theta\,\theta^{-1}\right\rangle\right),\,\operatorname{Ad}^*_\th\sigma \Big)
\,\in\,\X_{\vol}(\mathbb{R}^3)\times\mathcal{F}(\mathbb{R}^3,\mathfrak{o})^*
\,.
\]
\end{theorem}
See Appendix \ref{proposition1} for the proof.
Notice that the above result will follow very easily in the more abstract setting of the next sections.

Besides the above right-action momentum map ${\bf J}_R$, one also has the momentum map ${\bf J}_L$ arising from the left infinitesimal action of the Hamiltonian functions
$h\in\mathcal{F}(\mathbb{R}^{2k}\times T^*\mathcal{O})$, similarly to the left leg of 
the Marsden Weinstein dual pair for ideal fluid. 
However, there is an
important special situation that arises in the physics of Yang-Mills
charged particles. Indeed, the celebrated Wong's equations for the
motion of a single Yang-Mills charge are produced by Hamiltonians in
the space $\mathcal{F}(\mathbb{R}^{2k}\times
T^*\mathcal{O})^\mathcal{O}$ of $\mathcal{O}$-invariant functions on
$\mathbb{R}^{2k}\times T^*\mathcal{O}$. In this context, the
trivialization map
$T^*\mathcal{O}\simeq\mathcal{O}\times\mathfrak{o}^*$ induces (by
Lie-Poisson reduction) the Poisson structure on the space
$\mathbb{R}^{2k}\times
T^*\mathcal{O}/\mathcal{O}\simeq\mathbb{R}^{2k}\times\mathfrak{o}^*$
(cf. e.g. \cite{MoMaRa84}). Therefore, it is reasonable to consider
a left-action momentum map arising from the infinitesimal action of
the Poisson subalgebra
\[
\mathcal{F}(\mathbb{R}^{2k}\times
T^*\mathcal{O})^\mathcal{O}=\mathcal{F}(\mathbb{R}^{2k}\times\mathfrak{o}^*)\subset\mathcal{F}(\mathbb{R}^{2k}\times
T^*\mathcal{O}) \,.
\]
Indeed,
%upon introducing the Lie Poisson isomorphism
%\begin{eqnarray*}
%\sim\ :\,\mathcal{F}(\mathbb{R}^{2k}\timesT^*\mathcal{O})^\mathcal{O}&\to&\mathcal{F}(\mathbb{R}^{2k}\times\mathfrak{o}^*)
%\\
%h(\boldsymbol{q},\boldsymbol{p},g,\alpha)&=&\bar{h}(\boldsymbol{q},\boldsymbol{p},\alpha\,g^{-1}) \,,
%\end{eqnarray*}
one verifies that the infinitesimal action of
$\bar{h}(\boldsymbol{q},\boldsymbol{p},{\zeta })\in\mathcal{F}(\mathbb{R}^{2k}\times\mathfrak{o}^*)$
\[
\big(\,\bar{h}\,\big)_{\mathcal{F}(\mathbb{R}^3,\mathbb{R}^{2k}\times
T^*\mathcal{O})}\left(\boldsymbol{Q},\boldsymbol{P},\sigma ,\theta\right)
=\!
\left.\left(\frac{\partial \bar{h}}{\partial\boldsymbol{p}},-\frac{\partial \bar{h}}{\partial\boldsymbol{q}},-\operatorname{ad}^*_{\textstyle\frac{\partial \bar{h}}{\partial\zeta }}\zeta ,\frac{\partial \bar{h}}{\partial{\zeta }}\,\theta\right)\!\right|_{\left(\boldsymbol{q},\boldsymbol{p},{\zeta }\right)=\left(\boldsymbol{Q,P},{\sigma }\right)}
\]
produces the left-action momentum map
\[
\mathbf{J}_L(\boldsymbol{Q,P},{\sigma },\theta) =
\!\int\!w(x)\,\delta(\boldsymbol{q}-\boldsymbol{Q}(x))\,\delta(\boldsymbol{p}-\boldsymbol{P}(x))\,\delta\!\left(\boldsymbol{\sigma}-\sigma (x)\right)\,
{\rm d}^3x \in \mathcal{F}(\mathbb{R}^{2k}\times\mathfrak{o}^*)^*
\]
In this case, the proof is straightforward and will be omitted.

\begin{remark}[Relations to kinetic theory and ideal fluids]\label{rem:YMV}\rm
In analogy with the structures arising from Euler's equation in two dimensions,
the above expression is well known in kinetic theory, particularly
the kinetic theory of Yang-Mills plasmas of interest in stellar
astrophysics. Indeed, the above expression coincides with the well
known Klimontovich particle solution of the Yang-Mills Vlasov
equation
%\begin{equation}\label{YMVlasov}
%\frac{\partial f}{\partial t}+\left\{f,\frac{\delta H}{\delta f}\right\} + \left\langle \mu, 
%\left[ \frac{\partial f}{\partial\mu}, \frac{\partial}{\partial \mu}\frac{\delta H}{\delta f}
%\right]\right\rangle=0
%\end{equation}
for 
%elements $f\in\mathcal{F}(\mathbb{R}^{2k}\times\mathfrak{o}^*)^*$, that is
distributions on the reduced Yang-Mills phase space
$\mathbb{R}^{2k}\times\mathfrak{o}^*$ (\cite{GiHoKu1983}). However,
the fortunate coincidence holding for ordinary Euler's equation in
two dimensions does not hold for EP$\Aut_{\rm vol}$, that is the left
momentum map does not provide point vortex solutions of the 2D
EP$\Aut_{\rm vol}$ system. Indeed, the latter does not have the form
of a single Vlasov kinetic equation as it happens for 2D Euler.
\end{remark}

The next section characterizes the group actions underlying the
above momentum maps in the more general setting when the physical
space $\mathbb{R}^3$ is replaced by a volume manifold $S$ and the
trivial Yang-Mills phase space $T^*\mathbb{R}^k\times T^*\mathcal{O}$
is replaced by a generic (non trivial) $\mathcal{O}$-bundle carrying
an exact invariant symplectic form.

%%%%%%%%%%%%%%%%%%%%%%%%%%%%%

\subsection{Group actions and dual pair}

We now present a dual pair of momentum maps in the context of the
EP$\mathcal{A}ut_{\vol}$ equation, that consistently generalizes the
Marsden-Weinstein dual pair. As we shall see, in the principal bundle context the groups $\Diff_{\vol}(S)$ and $\Diff_{\ham}(M)\times _B \mathbb{R}  $ are naturally replaced by the automorphism groups $\Aut_{\vol}(P_S)$ and $\Aut_{\ham}(P)\times_B \mathbb{R} $ of two $ \mathcal{O} $-principal bundles $P_S$ and $P$, respectively. For the left momentum map, we will exhibit a subgroup $\overline{\Aut}_{\ham}(P)\times_B \mathbb{R}\subset \Aut_{\ham}(P)\times_B \mathbb{R} $ that is more appropriate and whose Lie algebra identifies with the space $\F(P/ \mathcal{O} )$ endowed with the reduced Poisson bracket.

\subsubsection{Geometric setting} Let $\pi:P_S\rightarrow S$ be the
principal $\O$-bundle of our EP$\mathcal{A}ut_{\vol}$ equation, and
consider another principal $\O$-bundle $P\rightarrow M$ such that
$P$ carries an exact symplectic form $\omega=-\mathbf{d}\theta$.
Moreover, we assume that $\theta$ is $\mathcal{O}$-invariant. As
above, we endow $P_S$ with the $\mathcal{O}$-invariant volume form
$\mu_{P_S}=\pi^*\mu_S\wedge\mathcal{A}^*\operatorname{det}_\gamma$,
where $\gamma$ is an $\operatorname{Ad}$-invariant inner product on
$\mathfrak{o}$ and $\A$ a principal connection on $P_S$.

Consider the manifold
\[
\Emb_{\O}(P_S,P)
\]
of $\O$-equivariant embeddings, see Appendix \ref{lemma3}. The tangent space
$T_f\Emb_{\O}(P_S,P)$ can be identified with the space of
equivariant maps $u_f: P_S\rightarrow TP$ such that $\pi\circ u_f=f$. We endow the manifold $\Emb_{\O}(P_S,P)$ with the symplectic form $\bar\omega$ as above, that is
\[
\bar\omega(f)(u_f,v_f):=\int_{P_S}\omega(f(p))(u_f(p),v_f(p))\mu_{P_S}.
\]
The function under the integral is $\mathcal{O}$-invariant, so the right hand side is an integral over $S$.
Now the local triviality of the bundle $P_S\to S$
ensures the non-degeneracy of $\bar\om$.

Thanks to Lemma \ref{volume_preserving_aut}, the group $\mathcal{A}ut_{\vol}(P_S)$ acts symplectically on $\Emb_{\O}(P_S,P)$ by composition on the right. Similarly, the group of Hamiltonian automorphisms, defined by
\[
\mathcal{A}ut_{\ham}(P):=\mathcal{A}ut(P)\cap \operatorname{Diff}_{\ham}(P)
\]
acts symplectically on the left by composition. If the structure group $\O$ of the principal bundles is the trivial group with one element, we recover the actions associated to the Marsden-Weinstein dual pair.

\subsubsection{The right momentum map} The
momentum map associated to the right action is
\begin{equation}\label{J_R_aut_vol}
\mathbf{J}_R(f)=[f^*\theta]\in \Omega^1_\mathcal{O}(P_S)/\mathbf{d}\F_{\mathcal{O}}(P_S)=\mathfrak{aut}_{\vol}(P_S)^*
\end{equation}
The dual to $\mathfrak{aut}_{\vol}(P_S)$ is identified here with
$\O$-invariant 1-forms quotiented by exact $\O$-invariant 1-forms which are differentials of $\O$-invariant functions.

One can write this identification in detail,
using a principal connection $\A$. Since a volume form is fixed, $\mathfrak{aut}(P_S)^*$ can be identified with $\Omega^1_\mathcal{O}(P_S)$. As in \eqref{isomorphism_A_dual}, the dual to the isomorphism
$\mathfrak{aut}(P_S)\rightarrow\mathfrak{X}(S)\times\mathcal{F}_\mathcal{O}(P_S,\mathfrak{o})$ reads
\begin{equation}\label{inv1form}
\Omega^1_\mathcal{O}(P_S)\rightarrow\Om^1(S)\x\F_\O(P_S,\mathfrak{o}^*),
\quad \al\mapsto \left(\bar\al=(\operatorname{Hor}^\mathcal{A})^*\al,\lambda=\mathbb{J}\circ \al\right),
\end{equation}
where $\operatorname{Hor}^\mathcal{A}$ denotes the horizontal-lift associated to the principal connection $\mathcal{A}$ and $\mathbb{J}:T^*P_S\rightarrow\mathfrak{o}^*$ is the cotangent bundle momentum map.
The inverse to the isomorphism \eqref{inv1form} is
\[
(\bar\al,\la)\mapsto \pi^*\bar\al+\mathcal{A}^*\lambda.
\]
Now the isomorphism $\mathfrak{aut}_{\vol}(P_S)\rightarrow\mathfrak{X}_{\vol}(S)\times\mathcal{F}_\mathcal{O}(P_S,\mathfrak{o})$, see \eqref{akau}, ensures the desired identification
\begin{align*}
\mathfrak{aut}_{\vol}(P_S)^*&=\X_{\vol}(S)^*\x\F_\O(P_S,\mathfrak{o}^*)
=(\Om^1(S)/\dd\F(S))\x\F_\O(P_S,\ou^*)\\
&=(\Om^1(S)\x\F_\O(P_S,\mathfrak{o}^*)) /(\mathbf{d}\F(S)\x\{0\})=\Omega^1_\mathcal{O}(P_S)/\mathbf{d}\F_\O(P_S).
\end{align*}
As we shall see later, for trivial bundles one recovers the conclusions in theorem \ref{rightlegtrivial}.

\subsubsection{The left momentum map} In
order to properly define the momentum map associated to the action
by composition on the left, there are two difficulties to overcome.
First, as above, we need to pass to the central extension of
$\mathcal{A}ut_{\ham}(P)$ by the cocycle defined in
\cite{IsLoMi2006}. This is possible since $\mathcal{A}ut_{\ham}(P)$
is a subgroup of $\operatorname{Diff}_{\ham}(P)$ on which the
cocycle $B$ can be defined as in \eqref{cocycle}. Second, we note
that the Lie algebra of the Hamiltonian automorphisms is given by
\[
\mathfrak{aut}_{\ham}(P)=\mathfrak{aut}(P)\cap\mathfrak{X}_{\ham}(P),
\]
that is, it consists of Hamiltonian vector fields that are $\mathcal{O}$-equivariant. Such vector fields are not necessarily associated to $\mathcal{O}$-invariant Hamiltonian functions on $P$. The latter form the Lie subalgebra
\[
\overline{\mathfrak{aut}}_{\ham}(P):=\left\{X_h\mid
h\in\mathcal{F}(P)^\mathcal{O}\right\}\subset
\mathfrak{aut}_{\ham}(P)\,,
\]
since the Lie algebra bracket is $[X_h,X_k]_L=X_{\{h,k\}}$ and the Poisson bracket of two $\mathcal{O}$-invariant functions is $\mathcal{O}$-invariant, the symplectic form being $\mathcal{O}$-invariant.

The subgroup associated to this Lie subalgebra is determined in the following theorem (see Appendix \ref{proposition2} for the proof).

\begin{theorem}\label{kernel}
Let $(P,-\dd\th)$  be a connected exact symplectic manifold
with $\th\in\Om_\O^1(P)$.
Let $\Hom(\O,\mathbb{R})$ be the vector space of group homomorphisms from $\O$ to $\mathbb{R}$ and consider the map
\[
\Psi:\mathcal{A}ut_{\ham}(P)\to\Hom(\O,\mathbb{R}),
\]
defined by
\[
\Psi(\varphi)(g):=F_\varphi-F_\varphi\circ\Phi_g\,,\qquad
 \forall\ g\in
G\,,
\]
where $F_\varphi\in\mathcal{F}(P)$ is such that
$\mathbf{d}F_\varphi=\ph^*\th-\th$. Then $\Psi$ is well-defined and
is a group homomorphism.

Consider the normal subgroup $\overline{\mathcal{A}ut}_{\ham}(P)\subset\mathcal{A}ut_{\ham}(P)$ defined by
\[
\overline{\mathcal{A}ut}_{\ham}(P):=\operatorname{ker}\Psi.
\]
Then the formal Lie algebra of $\overline{\mathcal{A}ut}_{\ham}(P)$ is the space $\overline{\mathfrak{aut}}_{\ham}(P)$ of Hamiltonian vector fields on $P$ associated to $\mathcal{O}$-invariant Hamiltonians.
\end{theorem}

\begin{remark}[Special Hamiltonian automorphisms]\rm
The elements of the Lie group $\overline{\mathcal{A}ut}_{\ham}(P)$ will be called \emph{special Hamiltonian automorphisms} of the principal bundle $P$. This group is of central importance in common gauge theories: for example, they govern Wong's equations for the dynamics of a Yang-Mills charge moving in the physical space. (Just like Hamiltonian diffeomorphisms govern canonical Hamilton's equations).
\end{remark}

Consider now the central extension $\overline{\mathcal{A}ut}_{\ham}(P)\times_B\mathbb{R}$. The Lie algebra isomorphism \eqref{Lie_algebra_isom} shows that the associated central extended Lie algebra can be identified with the Lie algebra $\mathcal{F}(P)^\mathcal{O}$ of $\mathcal{O}$-invariant functions on $P$, endowed with the symplectic Poisson bracket $\{f,g\}_{P}=\omega(X_f,X_g)$ on $P$. Equivalently, it is also identified with the Lie algebra $\mathcal{F}(M)$ of functions on $M$ with the reduced Poisson bracket $\{\,,\}_{M}$ on $M$.
Equivariant functions on $P$ are constant on the fibers, so there is an isomorphism
\begin{equation}\label{isom_Poisson}
h\in \mathcal{F}(P)^\mathcal{O}\rightarrow \bar h\in\mathcal{F}(M)
\end{equation}
with inverse $\bar h\mapsto h=\bar h\circ\pi$. We recall here that the reduced Poisson bracket $\{\,,\}_{M}$ on the quotient $M=P/\mathcal{O}$ is uniquely determined by the formula
\[
\{f,g\}_{M}\circ\pi=\{f\circ\pi,g\circ\pi\}_{P},
\]
for all $f, g\in\mathcal{F}(M)$, see \S10.5 in \cite{MaRa99}.

\begin{remark}[The Vlasov chromomorphism group]\rm
The elements of the central extension $\overline{\mathcal{A}ut}_{\ham}(P)\times_B\mathbb{R}$ will be called \emph{Vlasov chromomorphisms} and this group will be denoted by $V\mathcal{C}hrom(P)$. The name is clearly inspired by the fact that this Lie group is the configuration space underling the collisionless Vlasov kinetic theory of interacting Yang-Mills charges moving in the physical space. As explained in \cite{GiHoKu1983}, the fluid closure of the Yang-Mills-Vlasov equation yields the equations of chromohydrodynamics for Yang-Mills plasmas, thereby providing one more reason for the name `chromomorphisms'.
\end{remark}

We now consider the action of the group $V\mathcal{C}hrom(P)$ on $\Emb_{\O}(P_S,P)$ by composition on the left. One easily checks that the associated momentum map is
\begin{equation}\label{left_momap_autvol}
\mathbf{J}_L:\Emb_{\O}(P_S,P)\rightarrow\mathcal{F}(M)^*=(\mathcal{F}(P)^\mathcal{O})^*,\quad  \left\langle\mathbf{J}_L(f),\bar h\right\rangle=\left\langle\mathbf{J}_L(f),h\right\rangle=\int_{P_S}h(f(p))\mu_{P_S}.
\end{equation}
Since the function $p\mapsto h(f(p))$ on $P_S$ is $\mathcal{O}$-invariant, it defines a function on $S$, and we have in fact an integral over $S$.
By abuse of notation, we can write
\[
\mathbf{J}_L(f)=\int_S\delta (n-f(p))\mu_S\in\mathcal{F}(M)^*.
\]
We thus obtain the dual pair

\begin{picture}(150,100)(-70,0)%
\put(102,75){$\Emb_{\O}(P_S,P)$}
%top label

\put(90,50){$\mathbf{J}_L$}
%left label

\put(160,50){$\mathbf{J}_R$}
%right arrow label

\put(60,15){$\mathcal{F}(P/\O)^*$
%$\mathfrak{X}^{\ast}(M)\times\mathcal{F}^*(M,\mathfrak{o})$
}
%left bottom label

\put(160,15){$\mathfrak{aut}_{\vol}(P_S)^*$
%$\mathfrak{X}^{\ast}(S)\times\mathcal{F}^*(S,\mathfrak{o})$
}
%right bottom label

\put(130,70){\vector(-1, -1){40}}
% left slanted arrow

\put(135,70){\vector(1,-1){40}}
% right slanted arrow

\end{picture}\\
Despite the fact that the left leg seems to be the same as the left leg of the Marsden-Weinstein dual pair, there is a key difference here, since here $M=P/\O$ is not symplectic but Poisson, and the Lie bracket on $\mathcal{F}(M)$ is given by the reduced Poisson bracket $\{\,,\}_{M}$.

The results obtained so far are summarized in the following

\begin{theorem}[Dual pair for  $\operatorname{EP}\!\mathcal{A}ut_{\vol}$ flows] Let $\pi:P_S\rightarrow S$ be a principal $\O$-bundle and consider another principal $\O$-bundle $P\rightarrow M$ such that $P$ carries an exact symplectic form $\omega=-\mathbf{d}\theta$, where $\theta$ is $\mathcal{O}$-invariant.

Then the group $\mathcal{A}ut_{\vol}(P_S)$ acts symplectically on the right on the symplectic manifold $(\operatorname{Emb}_\mathcal{O}(P_S,P),\bar\omega)$ and admits the momentum map
\[
\mathbf{J}_R:\operatorname{Emb}_\mathcal{O}(P_S,P)\rightarrow\mathfrak{aut}_{\vol}(P_S)^*,\quad \mathbf{J}_R(f)=[f^*\theta].
\]

There is a subgroup $\overline{\mathcal{A}ut}_{\ham}(P)\subset\mathcal{A}ut_{\ham}(P)$ of special Hamiltonian automorphisms, whose Lie algebra consists of Hamiltonian vector fields associated to $\mathcal{O}$-invariant Hamiltonian functions on $(P,\omega)$. The Lie algebra of $V\mathcal{C}hrom(P)$  is isomorphic to the space of functions on $M$ endowed with the reduced Poisson bracket on $M=P/\mathcal{O}$.

The Vlasov chromomorphism group $V\mathcal{C}hrom(P)$ acts symplectically on the left on the symplectic manifold $(\operatorname{Emb}_\mathcal{O}(P_S,P),\bar\omega)$ and admits the momentum map
\[
\mathbf{J}_L:\operatorname{Emb}_\mathcal{O}(P_S,P)\rightarrow\mathcal{F}(M)^*,\quad\mathbf{J}_L(f)=\int_S\delta (n-f(p))\mu_S.
\]
\end{theorem}

\begin{remark}[Equivariance, Clebsch variable, and Noether theorem]\normalfont The equivariance and invariance properties of the momentum maps still hold as in the Marsden-Weinstein dual pair, see Remark \ref{Equiv_Clebsch}. Therefore, we have the same Clebsch variables and Noether theorem interpretations.
The Lie-Poisson system on $\mathcal{F}(M)^*$ is a Vlasov system
whose Poisson bracket is not symplectic but is the reduced Poisson
bracket on $M=P/\mathcal{O}$. When $P=T^*\bar P$ and the
principal bundle $\bar P$ is trivial, this system is related to
Yang-Mills Vlasov equation. This particular case
will be considered in the following section.
\end{remark}

%%%%%%%%%%%%%%%%%%%%%%%%%%%%%%%%%%%%%%

\subsection{The Yang-Mills phase space}\label{Sec:EPAut_volForYM}
We now consider the special case when the total space $P$ of the principal
bundle $P\rightarrow M$ is the cotangent bundle of another
principal bundle $\bar P\rightarrow \bar M$. We endow $P=T^*\bar P$ with the canonical symplectic form $\Omega_{\bar
P}=-\mathbf{d}\Theta_{\bar P}$ and we let $\mathcal{O}$ act on $P$ by cotangent lift. Thus we have $P=T^*\bar P$ and $M=T^*\bar
P/\mathcal{O}$. This particular choice is motivated by \S\ref{Clebsch_EPAutvol}, particularly by the dynamics of Yang-Mills-Vlasov kinetic theories, as mentioned in remark \ref{rem:YMV}.
If moreover $\bar P$ is trivial, that is $\bar P=\bar M\times\mathcal{O}$, then $P=T^*\bar P$ is also a trivial principal $\O$-bundle over  $M=T^*\bar P/\mathcal{O}=T^*\bar M\times\mathfrak{o}^*$, since we have the equivariant diffeomorphism
\begin{equation}\label{equiv_diffeo}
\rho: T^*\bar M\times T^*\mathcal{O} \rightarrow  \left(T^*\bar M\times \mathfrak{o}^*\right)\times\mathcal{O},
\quad\rho(\alpha_q,\alpha_g)=\left((\alpha_q,\alpha_gg^{-1}),g\right).
\end{equation}
In this case, the dual pair diagram reads

\begin{picture}(150,100)(-70,0)%
\put(67,75){$\Emb_{\O}(P_S,T^*\bar M\times T^*\mathcal{O})$}
%top label

\put(90,50){$\mathbf{J}_L$}
%left label

\put(160,50){$\mathbf{J}_R$}
%right arrow label

\put(50,15){$\mathcal{F}(T^*\bar M\times\mathfrak{o}^*)^*$
%$\mathfrak{X}^{\ast}(M)\times\mathcal{F}^*(M,\mathfrak{o})$
}
%left bottom label

\put(160,15){$\mathfrak{aut}_{\vol}(P_S)^*$.
%$\mathfrak{X}^{\ast}(S)\times\mathcal{F}^*(S,\mathfrak{o})$
}
%right bottom label

\put(130,70){\vector(-1, -1){40}}
% left slanted arrow

\put(135,70){\vector(1,-1){40}}
% right slanted arrow

\end{picture}\\
The Lie bracket on $\mathcal{F}(T^*\bar M\times\mathfrak{o}^*)$ is the reduced Poisson bracket given here by
\[
\{f,g\}_{M}=\{f,g\}_{T^*\bar M}+\{f,g\}_+,
\]
where the first term denotes the canonical Poisson bracket on
$T^*\bar M$ and the second term is the Lie-Poisson bracket on
$\mathfrak{o}^*$. In what follows, it will be easy (although not necessary) to assume that $P_S$ is trivial, i.e. $P_S=S\times\mathcal{O}$ and $\mathfrak{aut}_{\vol}(P_S)^*\simeq\mathfrak{X}_{\vol}(S)^*\times\mathcal{F}(S,\mathcal{O})$.

%%%%%%%%%%%%%%%%%%%%%%%%%%%%%%%%%%%%

\subsubsection{The group of special Hamiltonian automorphisms}\label{Sec:Aut_ham-bar}

Note that, since the principal bundle is trivial, we can write the automorphism group of $T^*\bar M\x T^*\O\rightarrow T^*\bar M\times\mathfrak{o}^*$ as a semidirect product
\begin{equation}\label{Aut_T*P_trivial}
\mathcal{A}ut(T^*\bar M\x T^*\O)=\operatorname{Diff}(T^*\bar M\times\mathfrak{o}^*)\,\circledS\,\mathcal{F}(T^*\bar M\times\mathfrak{o}^*,\mathcal{O}).
\end{equation}
From the general theory above, the group we need is the central extension
\[
V\mathcal{C}hrom(T^*\bar M\x T^*\O)= \overline{\mathcal{A}ut}_{\ham}(T^*\bar M\x T^*\O)\times_B\mathbb{R}.
\]

In order to describe more concretely this group, we start with the special case $\overline\Aut_{\ham}(T^*\O)$.
We will use the following expressions for the canonical one-form $ \Theta _\mathcal{O}$ and the Hamiltonian vector field $X _h $, in the right trivialization $T^* \mathcal{O} \rightarrow  \mathcal{O} \times \mathfrak{o} ^\ast $, $ \beta _g \mapsto (g, \alpha )$, $ \alpha = \beta _g g ^{-1} $:
\begin{equation}\label{triv_theta_Xh} 
(\Th_\O)_{(g,\al)}(\xi_g,\nu)=\left\langle\alpha ,\xi_g g^{-1}\right\rangle,\qquad X_h(g,\al)=\left(\frac{\delta h}{\delta \al}g,-\operatorname{ad}^*_{\frac{\delta h}{\delta \al}}\al\right),
\end{equation} 
see \cite{AbMa1978}. Note also that a $\O$-equivariant diffeomorphism $\ph$ of $T^*\O$, expressed in the right trivialization $ \mathcal{O} \times \mathfrak{o} ^\ast $, 
is of the form
\begin{equation}\label{equiv}
\ph(g,\al)=(\hat{\varphi}(\al)g,\bar\ph(\al)),
\end{equation}
where $\bar\ph$ is a diffeomorphism of $\ou^*$
and $\hat{ \varphi }:\ou^*\to\O$ a smooth map.

\begin{proposition}\label{pautbar}
In the right trivialization, the group $\overline\Aut_{ham}(T^*\O)$ of equivariant Hamiltonian automorphisms $\ph$
with the property that the $\O$-invariant 1-form $\ph^*\Th_\O-\Th_\O$ on $\O\times\ou^*$
is the differential of an $\O$-invariant function can be expressed as
\[
\overline\Aut_{ham}(T^*\O)=\left \{\varphi \in\Aut(T^* \mathcal{O} )\mid
\bar \varphi ( \alpha )= \operatorname{Ad}^*_{\hat{ \varphi }( \alpha ) ^{-1} } \alpha \;\; \text{and}\;\;\left\langle \id_{ \mathfrak{o}^*}, \hat{ \varphi } ^{-1} \mathbf{d} \hat{ \varphi }\right\rangle  \in \mathbf{d} \mathcal{F} ( \mathfrak{o}  ^\ast )\right\},
\]
where $\left\langle \id_{ \mathfrak{o}^*}, \hat{ \varphi } ^{-1} \mathbf{d} \hat{ \varphi }\right\rangle$ denotes the 1-form $\left\langle \alpha ,\hat{\varphi}( \alpha ) ^{-1} \dd _\alpha\hat{\varphi} (\_\,) \right\rangle$.
\end{proposition}

\begin{proof}
We know that the Hamiltonian diffeomorphisms $\ph$ of an exact symplectic manifold 
with symplectic form $-\dd\Th_\O$ are characterized by the condition
that $\ph^*\Th_\O-\Th_\O$ is an exact 1-form.
Using the first equality in \eqref{triv_theta_Xh} and \eqref{equiv}, we compute
\begin{align*}
(\ph^*\Th_\O)_{(g,\al)}(\xi_g,\nu)&=(\Th_\O)_{\ph(g,\al)}\left( T_{(g,\al)}\ph(\xi_g,\nu)\right) \\
&=(\Th_\O)_{(\hat{\varphi}(\al)g,\bar\ph(\al))}\left( \hat{\varphi}(\al)\xi_g+(T_\al \hat{\varphi} \cdot \nu)g,T_\al\bar\ph \cdot \nu\right) \\
&=\left\langle \bar\ph(\al),\hat{\varphi}(\al)\xi_g g^{-1}\hat{\varphi}(\al)^{-1}+\left( T_\al\hat{\varphi} \cdot \nu\right) \hat{\varphi}(\al)^{-1}\right\rangle \\
&=\left\langle \Ad^*_{\hat{\varphi}(\al)}\bar\ph(\al),\xi_g g^{-1}\right\rangle +\left\langle \bar\ph(\al),(( \mathbf{d} \hat{\varphi} )\hat{\varphi}^{-1})_\al \cdot \nu\right\rangle 
\end{align*}
so that
\begin{equation}\label{exact}
(\ph^*\Th_\O-\Th_\O)_{(g,\al)}(\xi_g,\nu)=\left\langle \Ad^*_{\hat{\varphi}(\al)}\bar\ph(\al)-\al,\xi_g g^{-1}\right\rangle 
+\left\langle \bar\ph(\al),((\dd \hat{\varphi}) \hat{\varphi}^{-1})_\alpha \cdot \nu\right\rangle .
\end{equation}
Note also that an exact 1-form on $ \mathcal{O} \times\mathfrak{o}^\ast$ reads
\[
\left\langle \frac{\partial h}{\partial g} g ^{-1} , \xi _g g ^{-1} \right\rangle + \left\langle \frac{\delta  h}{\delta  \alpha }, \nu \right\rangle,
\]
where $h$ is a function on $ \mathcal{O} \times \mathfrak{o}  ^\ast $. So, the automorphism $\ph$ is Hamiltonian if and only if there exists a function $h(g, \alpha )$ such that
\begin{equation}\label{cond_exact}
\Ad^*_{\hat{\varphi}(\al)}\bar\ph(\al)-\al= \frac{\partial h}{\partial g}(g, \alpha ) \quad \text{and} \quad \left\langle \bar\ph(\al),\dd _\alpha\hat{\varphi} (\_\,) \hat{\varphi}(\alpha) ^{-1} \right\rangle= \frac{\delta h}{\delta \alpha }(g , \alpha ).
\end{equation}
In the right trivialization, an invariant function on $T^*\O$ is of the form $h(g,\al)=\bar h(\al)$ with $\bar h\in\F(\ou^*)$. So, from the first equality in \eqref{cond_exact}, we get the conditions $\Ad^*_{\hat{\varphi}(\al)}\bar\ph(\al)=\al$ and the second equality in \eqref{cond_exact} thus reads $ \left\langle \alpha ,\hat{\varphi}( \alpha ) ^{-1} \dd _\alpha\hat{\varphi} (\_\,) \right\rangle= \frac{\delta \bar h}{\delta \alpha }$.
\end{proof}
\begin{remark}[The Lie algebra of the special Hamiltonian automorphism group]\normalfont
%The diffeomorphism $\bar\ph$ of $\ou^*$ is completely determined by the smooth map $a:\ou^*\to\O$,
%since $\bar\ph(\al)=\Ad^*_{a(\al)^{-1}}\al$ for all $\al\in\ou^*$.
Similarly with \eqref{equiv}, an equivariant vector field $U\in \mathfrak{aut}(T^* \mathcal{O} )$ reads $U( g, \alpha )=(\hat U( \alpha )g, \bar U( \alpha ))$, where $\hat U: \mathfrak{o} ^\ast \rightarrow \mathfrak{o} $ and $ \bar U: \mathfrak{o} ^\ast \rightarrow \mathfrak{o} ^\ast $.
Using Proposition \ref{pautbar}, we obtain that the formal Lie algebra of $\overline\Aut_{ham}(T^*\O)$ is
\begin{align*}
\overline{\mathfrak{aut}}_{ham}(T^*\O)
&=\left\{U \in \mathfrak{aut}(T^* \mathcal{O} ) \mid \bar U(\alpha) = - \operatorname{ad}^*_{\hat U(\alpha) }\alpha   \;\; \text{and}\;\;\langle \id_{ \mathfrak{o}^*},  \mathbf{d} \hat{ U}\rangle  \in \mathbf{d} \mathcal{F} ( \mathfrak{o}  ^\ast ) \right\}\\
&=\left\{\left .(\bar U, \hat U)=\left(-\ad_{\frac{\de \bar h}{\de \al}},\frac{\de \bar h}{\de \al}\right)\, \right |\,\bar h\in\F(\ou^*)\right\}\\
&=\left\{X_{h}\in\mathfrak{aut}_{ham}(T^*\O)\mid h\in\F(T^*\O),h(g,\al)=\bar h(\al)\right\},
\end{align*}
where at the second equality, we used that fact that the 1-form $\langle \id_{\ou^*},\dd \hat U\rangle $ is exact if and only if there exists $\bar h \in \mathcal{F} (\mathfrak{o} ^\ast )$ such that $\hat U= \frac{\delta \bar h}{\delta \alpha } $; in the last equality we used the second expression in \eqref{triv_theta_Xh}. So, consistently with Theorem \ref{kernel}, the Lie algebra $\overline{\mathfrak{aut}}_{\ham}(T^*\O)$ consists of Hamiltonian vector fields associated to $ \mathcal{O} $-invariant Hamiltonians on $T^* \mathcal{O} $ (i.e. functions on $ \mathfrak{o} ^\ast  $), hence the Lie algebra of $\overline\Aut_{ham}(T^*\O)\x_B\RR$ 
can be identified with the Poisson algebra of smooth
functions on $\ou^*$.
\end{remark}

\begin{remark}[Generating functions of special Hamiltonian automorphisms]\normalfont
Notice that the function $\mathcal{S}\in\mathcal{F}(\mathfrak{o}^*)$ such that 
\[
\mathbf{d}\mathcal{S}(\alpha)=\left\langle \id_{ \mathfrak{o}^*}, \hat{ \varphi } ^{-1} \mathbf{d} \hat{ \varphi }\right\rangle(\alpha) =\left\langle \alpha, \hat{ \varphi } ^{-1}(\alpha) \mathbf{d} \hat{ \varphi }(\alpha)\right\rangle
\]
is a generating function for special Hamiltonian automorphisms $\overline\Aut_{\rm ham}(T^*\O)$, analogously to what happens for canonical transformations $\operatorname{Diff}_{\rm ham}(T^*\bar{M})$.
\end{remark}

\begin{remark}[The chromomorphism group]\rm
In the special case considered above, i.e. when $P=T^*\mathcal{O}$, the central extension $\overline\Aut_{ham}(T^*\O)\x_B\RR$ will be called \emph{chromomorphism group} and will be denoted by $\mathcal{C}hrom(T^*\mathcal{O})$.
\end{remark}

\noindent
In the case of a non-trivial manifold $\bar M$, 
%so $P=T^*\bar M\x T^*\O$ is a trivial bundle over $M=T^*\bar M\x\ou^*$,
there is an expression similar to that of Proposition \ref{pautbar} 
for the  subgroup $\overline\Aut_{\ham}(T^*\bar M\x T^*\O)$ of the automorphism group \eqref{Aut_T*P_trivial},
namely
\begin{align}\label{bara}
\overline\Aut_{\ham}(T^*\bar M\x T^*\O)=&\left \{\ph=(\bar\ph,\hat \varphi )\in\Diff(T^*\bar M\x\ou^*)\x\F(T^*\bar M\x\ou^*,\O)\mid
\bar\ph_2=\Ad^*_{\hat \varphi^{-1}},\right .\nonumber\\
&\quad \left .\left\langle \id_{\ou^*},\hat \varphi^{-1}\dd \hat \varphi\right\rangle +\bar\ph_1^*\Th_{\bar M}-p_1^*\Th_{\bar M}\in\dd\F(T^*\bar M\x\ou^*)\right \},
\end{align}
where $p_1$ is the projection of $T^*\bar M\x\ou^*$ onto its first component, and where we split the diffeomorphism $\bar\ph$ into its two components $\bar\ph_1\in\F(T^*\bar M\x\ou^*,T^*\bar M)$ and $\bar\ph_2\in\F(T^*\bar M\x\ou^*,\ou^*)$. The verification is left to the reader.

Let $h\in\F(T^*\bar M\x T^*\O)^\O$ and let
$\bar h\in\F(T^*\bar M\x \mathfrak{o}^*)$ induced via
$h(\be_q,\al_g)=\bar h(\be_q,\al_g g^{-1})$ for all $g\in\O$, as
in \eqref{isom_Poisson}. 
From \eqref{triv_theta_Xh}, the Hamiltonian vector field $X_h$ on $T^*\bar M\x T^*\O$, pushed
forward to $(T^*\bar M\x\mathfrak{o}^*)\times\O$ by the right
trivialization \eqref{equiv_diffeo}, reads
\[
\left(\rho_*X_h\right)(\beta_q,\al,g)=\left(\left(X_{\bar h_\al}(\beta_q),-\operatorname{ad}^*_{\frac{\delta \bar h_{\beta_q}}{\delta \al}}\al\right),\frac{\delta \bar h_{\beta_q}}{\delta \al}g\right).
\]
Here $\bar h_\al$ denotes the function on $T^*\bar M$ obtained from $\bar h$ by fixing an
element $\al\in\mathfrak{o}^*$. Similarly, $\bar h_{\beta_q}$ is
the function on $\mathfrak{o}^*$ obtained by fixing $\beta_q\in
T^*\bar M$. 

One can see that the Lie algebra $\overline{\mathfrak{aut}}_{\ham}(T^*\bar M\x T^*\O)$
of all these Hamiltonian vector fields with $\O$-invariant Hamiltonian function $h$ 
is indeed the formal Lie algebra of $\overline \Aut_{\ham}(T^*\bar M\x T^*\O)$.
The second condition in \eqref{bara}, differentiated at $(\bar\ph,\hat \varphi )$ in direction $(\bar U, \hat U)$, tells us that
$\left\langle \id_{\ou^*},\dd \hat U\right\rangle +\pounds_{\bar U}\Th_{\bar M}$ is an exact 1-form on $T^*\bar M\x T^*\O$, which is equivalent to
$\hat U+\mathbf{i} _{\bar U}\Om_{\bar M}=\dd\bar h$ for some $\bar h\in\F(T^*\bar M\x T^*\O)$.
This means that $\bar U=X_{\bar h_\al}$ and $\hat U=\frac{\delta \bar h_{\beta_q}}{\delta \al}$.

\begin{remark}[The cocycle for Vlasov chromomorphisms]
{\rm 
In the special case $P=T^*\bar M \times T^* \mathcal{O} $, there is a more concrete expression of the  cocycle \eqref{cocycle} when restricted to $\overline{\Aut}_{\ham}(P)$, namely
\[
B((\bar\ph,a),(\bar\ps,b))=\int_{(\be_0,\si_0)}^{\bar\ps(\be_0,\si_0)}\mathbf{d}\mathcal{S},
\] 
where $(\be_0,\si_0)\in T^*\bar M\x\ou^*$ and we have introduced the generating function $\mathcal{S}$, such that 
\[
\mathbf{d}\mathcal{S}= \left\langle \id_{\ou^*},\hat{\varphi }^{-1}\dd \hat{ \varphi }\right\rangle +\bar\ph_1^*\Th_{\bar M}-p_1^*\Th_{\bar M}.
\]
}
\end{remark}

%%%%%%%%%%%%%%%

\subsubsection{Momentum maps}

We now give the
expression of these momentum maps in the case where both principal
bundles are trivial, that is,
\[
P_S=S\times\mathcal{O}\quad\text{and}\quad\bar P=\bar M\times\mathcal{O}.
\]
In this case, an equivariant embedding in $\Emb_{\O}(P_S,T^*\bar P)=\Emb_{\O}(S\times\mathcal{O}, T^*\bar M\times T^*\mathcal{O})$ is necessarily of the form
\begin{equation}\label{efff}
f:S\times\mathcal{O}\rightarrow T^*\bar M\times T^*\mathcal{O},\quad f(x,g)=\left(\boldsymbol{P_Q}(x),\kappa_\theta(x)g\right),
\end{equation}
where $\boldsymbol{P_Q}:S\rightarrow T^*\bar M$, $\kappa_\theta:S\rightarrow T^*\mathcal{O}$ with $\boldsymbol{P_Q}(x)\in T^*_{Q(x)}\bar M$ and $\kappa_\theta(x)\in T^*_{\theta(x)}\mathcal{O}$.
We endow the manifold $S\times\mathcal{O}$ with the volume form $\mu_{P_S}:=\mu_S\wedge \mu_\O$, where $\mu_\mathcal{O}$ is the Haar measure on the compact group $\mathcal{O}$, that is, $\int_\O\mu_\O=1$ and $\mu_\O$ is $\O$-invariant,
so that $\mu_{P_S}$ is $\O$-invariant too. Remark that choosing the trivial connection $\mathcal{A}(v_x,\xi_g)=g^{-1}\xi_g$
on $P_S=S\times \mathcal{O}$ and the fact that $\gamma$ is $\operatorname{Ad}$-invariant, we see that $\mu_{P_S}:=\mu_S\wedge \mu_\O$ and $\mu_{P_S}=\pi^*\mu_S\wedge \mathcal{A}^*\operatorname{det}_\gamma$ coincide, since the constant factor $\operatorname{Vol}(\O)$
from lemma \ref{integration_invariant} is 1.

\paragraph{Right momentum map.} In order to compute explicitly
the momentum map associated to the action of
$\mathcal{A}ut_{\vol}(P_S)$ we shall make use of the isomorphism
\eqref{inv1form}. Since the principal bundle $P_S=S\x\O$ is trivial, a
principal connection 1-form $\A$ is determined by a 1-form
$A\in\Om^1(S,\ou)$ through
\[
\mathcal{A}(v_x,\xi_g)=\operatorname{Ad}_{g^{-1}}\left(A(v_x)+\xi_gg^{-1}\right).
\]
The space $\mathcal{F}_\mathcal{O}(P_S,\mathfrak{o}^*)$ can be identified with $\mathcal{F}(S,\mathfrak{o}^*)$ via the relation
\[
\lambda(x,g)=\operatorname{Ad}^*_g(\bar\lambda(x)),\quad \lambda\in \mathcal{F}_\mathcal{O}(P_S,\mathfrak{o}^*),\quad \bar\lambda\in\mathcal{F}(S,\mathfrak{o}^*).
\]
Finally, the horizontal lift of the vector $v_x\in T_xS$ is
$(v_x,-A(v_x)g)\in T_{(x,g)}P_S$. Using these observations, the isomorphism \eqref{inv1form} reads
\begin{equation}\label{isom_trivial_bundle}
\Omega^1_\mathcal{O}(P_S)\rightarrow \Omega^1(S)\times\mathcal{F}(S,\mathfrak{o}^*),\quad\alpha\mapsto (\bar\alpha,\bar\lambda),
\end{equation}
where $\bar\alpha(x)(v_x)=\alpha(x,e)(v_x,-A(v_x))$ and $\bar\lambda(x)\cdot\xi=\alpha(x,e)(0_x,\xi)$. This clearly induces an isomorphism
\[
\mathfrak{aut}_{\vol}(P_S)^*=\Omega^1_\mathcal{O}(P_S)/\mathbf{d}\mathcal{F}_\mathcal{O}(P_S)\rightarrow \left(\Omega^1(S)/\mathbf{d}\mathcal{F}(S)\right)\times\mathcal{F}(S,\mathfrak{o}^*).
\]
Note that, since the bundle is trivial, we can always choose the trivial connection $A=0$.

\begin{proposition}\label{Momap_right_YM} Suppose that $P_S$ and $\bar P$ are trivial bundles. Then the momentum map associated to the right action of $\mathcal{A}ut_{\vol}(P_S)$ on $\Emb_{\O}(P_S,T^*\bar P)$ has the expression
\[
\mathbf{J}_R(f)=\left( \left[\boldsymbol{P_Q}^*\Theta_{\bar M}+\kappa_\theta^*\Theta_\mathcal{O}-\th^{-1}\ka_\th\cdot A\right] ,\theta^{-1}\kappa_\theta\right)\in \left(\Omega^1(S)/\mathbf{d}\mathcal{F}(S)\right)\times\mathcal{F}(S,\mathfrak{o}^*),
\]
where the embedding $f$ is written as
\[
f(x,g)=(\boldsymbol{P_Q}(x),\kappa_\theta(x)g),
\]
and $\Theta_{\bar M},\Theta_\mathcal{O}$ are the canonical one-forms on $T^*{\bar M}$, $T^*\mathcal{O}$.
\end{proposition}
\begin{proof} We fix a connection $\A$ and use the isomorphism \eqref{isom_trivial_bundle}. Given $\be\in T^*\O$, we denote by $\ell_\be:\O\to T^*\O$ the orbit map defined by $\ell_\be(g)=\be g$.
Knowing that $f(x,g)=(\boldsymbol{P_Q}(x),\ka_\th(x)g)$, the first component $\bar\al$ of $\al=f^*\Th_{\bar P}$ computes
\begin{align*}
\bar\al(v_x)&=(f^*\Th_{\bar P})(v_x,-A(v_x))\\
&=\boldsymbol{P_Q}^*\Theta_{\bar M}(v_x)+\kappa_\theta^*\Theta_\mathcal{O}(v_x)
-\Th_\O(T\ell_{\ka_\th(x)}(A(v_x)).
\end{align*}
Using the definition of the canonical 1-form $\Th_\O$
and the identity $(T\pi\o T\ell_\be)(\xi)=
\pi(\be)\xi$ for all $\xi\in\ou$,
the last term becomes
\begin{align*}
\Th_\O(T\ell_{\ka_\th(x)}(A(v_x))
&=\ka_\th(x)\cdot(T\pi\o T\ell_{\ka_\th(x)})(A(v_x))
=\ka_\th(x)\cdot\th(x)(A(v_x))
\\&=\th(x)^{-1}\ka_\th(x)\cdot A(v_x),
\end{align*}
so $\bar\al=\boldsymbol{P_Q}^*\Theta_{\bar M}+\kappa_\theta^*\Theta_\mathcal{O}
-\th^{-1}\ka_\th\cdot A\in\Om^1(S)$.

The second component of $\al=f^*\Th_{\bar P}$ is $\bar\la=\th^{-1}\ka_\th\in\F(S,\ou^*)$ because for all $x\in S$ and $\xi\in\mathfrak{o}$
\begin{align*}
\bar\la(x)\cdot\xi&
%=\la(x,e)\cdot\xi=(f^*\Th_{\bar P})(\xi_P(x,e))
=(f^*\Th_{\bar P})(0_x,\xi)
=\Th_\O(T\ell_{\ka_\th(x)}(\xi))
=\ka_\th(x)\cdot (T\pi\o T\ell_{\ka_\th(x)})(\xi)\\
&=\ka_\th(x)\cdot\th(x)\xi=\th(x)^{-1}\ka_\th(x)\cdot\xi.
\end{align*}
This shows that
\[
\mathbf{J}_R(f)=\left( \left[\boldsymbol{P_Q}^*\Theta_{\bar M}+\kappa_\theta^*\Theta_\mathcal{O}-\th^{-1}\ka_\th\cdot A\right] ,\theta^{-1}\kappa_\theta\right)\in \left(\Omega^1(S)/\mathbf{d}\mathcal{F}(S)\right)\times\mathcal{F}(S,\mathfrak{o}^*)
\]
is the momentum map associated to the right action.
\end{proof}

\begin{remark}[Trivialized expression]\normalfont
For comparison with the formula obtained in Proposition \ref{rightlegtrivial}, it is useful to identify the manifold
$\Emb_{\O}(P_S,T^*\bar P)$ with the product $\Emb(S,T^*\bar M\times\mathfrak{o}^*)\x\F(S,\mathcal{O})$. This is possible thanks to Lemma \ref{lemma_Q_KK}. More precisely, to the equivariant embedding $f:S\times\mathcal{O}\rightarrow T^*\bar M\times T^*\mathcal{O}$,
\[
f(x,g)=\left(\boldsymbol{P_Q}(x),\kappa_\theta(x)g\right),\quad \boldsymbol{P_Q}:S\rightarrow T^*\bar M, \quad\kappa_\theta:S\rightarrow T^*\mathcal{O},
\]
we associate the pair $((\boldsymbol{P_Q},\mu),\th)\in \Emb(S,T^*\bar M\times\mathfrak{o}^*)\x\F(S,\mathcal{O})$,
where
\[
\mu(x)=\ka_\th(x)\th(x)^{-1}\in\mathfrak{o}^*.
\]
In order to obtain the momentum map of Proposition \ref{Momap_right_YM} in terms of the variables $(\mu,\theta)$, we introduce the right trivialization $R:\O\x\mathfrak{o}^*\to T^*\O$, $R(\theta ,\mu )=\mu \theta $ and we compute the expression $R^*\Theta_\mathcal{O} $. For $(g,\alpha )\in \mathcal{O} \times\mathfrak{o}^*$ and $(\xi_g,\nu)\in T^*_g \mathcal{O} \times\mathfrak{o}^*$,  we have
\begin{align*}
(R^*\Th_\O)_{(g,\al)}(\xi_g,\nu)&=( \Th_\O)_{\al g} \left( T_g\ell_\al(\xi_g)+V_\nu(\al g)\right) \\
&=\al g\cdot T_{\al g}\pi(T_g\ell_\al(\xi_g))  =\al g \cdot \xi_g =\al\cdot \xi_g g^{-1},
\end{align*}
where $V_\nu$ is the vertical vector field on $T^*\O$ defined by $V_\nu(\al g)=\left.\frac{d}{dt}\right|_{t=0} ((\al+t\nu)g)$. Then for $(\th,\sigma ):S\to\O\x\o^*$,
%Then for $\mu,\theta: S \rightarrow \mathfrak{o}^*\times \mathcal{O} $,
we have
\[
(\theta, \sigma   )^*(R^*\Th_\O)(v_x)=\left( R^*\Th_\O\right) _{(\theta (x), \mu (x))}(\mathbf{d} \theta , \mathbf{d} \mu )=\mu(x) \cdot (\dd\th)\th^{-1}(x).
\]
Therefore,  the right momentum map from proposition \ref{Momap_right_YM}
becomes
\begin{align*}
\J_R(\boldsymbol{P_Q},\sigma ,\th)& =\left( \left[\boldsymbol{P_Q}^*\Theta_{\bar M}+(\th,\mu)^*R^*\Theta_\mathcal{O}-\th^{-1}\ka_\th\cdot A\right] ,\theta^{-1}\kappa_\theta\right)\\
&=\left( \left[\boldsymbol{P_Q}^*\Theta_{\bar M}+\mu \cdot (\dd\th)\th^{-1}
-\Ad^*_\th\mu\cdot A\right] ,\Ad^*_\th\mu\right).
\end{align*}
For $A=0$ and $S=\mathbb{R}^3$, we recover the expression of $\J_R$ from proposition \ref{rightlegtrivial}.
\end{remark}

\paragraph{Left momentum map.} From the general formula
\eqref{left_momap_autvol} above, we have
\[
\mathbf{J}_L:\Emb_{\O}(S\times\mathcal{O}, T^*\bar M\times T^*\mathcal{O})\rightarrow\mathcal{F}(T^*\bar M\times T^*\mathcal{O})^\mathcal{O}=\mathcal{F}(T^*\bar M\times\mathfrak{o}^*)
\]
and
\begin{align*}
\left\langle\mathbf{J}_L(f),h\right\rangle&=\int_{S\times\mathcal{O}}h(f(x,g))\mu_S\wedge\mu_\mathcal{O}=\int_{S\times\mathcal{O}}h(\boldsymbol{P_Q}(x),\kappa_\theta(x)g)\mu_S\wedge\mu_\mathcal{O}\\
&=\int_S\bar h\left(\boldsymbol{P_Q}(x),\kappa_\theta(x)\theta(x)^{-1}\right)\mu_S.
\end{align*}
Hence, we can write the formula
\[
\mathbf{J}_L(f)=\int_S\delta (\alpha_q-\boldsymbol{P_Q}(x))\delta \left(\mu-\kappa_\theta(x)\theta(x)^{-1}\right)\mu_S\in\mathcal{F}(T^*\bar M\times\mathfrak{o}^*)^*.
\]
As we have discussed, the above momentum map is not a solution of
the EP$\Aut_{\rm vol}$ equation. However this expression is of
primary importance in the kinetic theory of Yang-Mills Vlasov
plasmas, since this is nothing else than the single particle
solution of the Vlasov equation, which is the base
for the theory of Yang-Mills charged fluids. Indeed, in
\cite{GiHoKu1983} it is shown that the Vlasov equation itself is directly constructed starting from the
above momentum map.

\section{Conclusions}

After presenting the
EP$\Aut$ equations in the case of a trivial principal $\O-$bundle
$P=M\times\O$, the paper introduced the associated left and right
momentum maps which extend the well known dual pair structure for
geodesic flows on the diffeomorphism group \cite{HoMa2004}. Thus, as
a first result, the paper incorporated the momentum maps found in
\cite{HoTr2008} in a unifying dual pair diagram, whose right leg
provides the Clebsch representation of the EP$\Aut$ system, while
the left leg provides the singular $\delta-$like solutions.

In the third section, the paper extended the previous construction to
consider the case of a non-trivial principal $\mathcal{O}-$bundle
$P$. For this more general case, the EP$\Aut$ equations were written
also in terms of the adjoint bundle ${\rm
Ad}P=(P\times\mathfrak{o})/\mathcal{O}$, by using the identification
between $\mathcal{O}-$equivariant maps
$\mathcal{F_O}(P,\mathfrak{o})$ and the space $\Gamma({\rm Ad}P)$ of
sections of the adjoint bundle. In this setting, the momentum maps
were defined on the space of $\mathcal{O}-$equivariant embeddings
${\rm Emb}_\mathcal{O}(P_S,P)$ of the sub-bundle $P_S$ in the
ambient bundle $P$. This is a crucial point because it showed for the
first time how the right action momentum map (Clebsch
representation) necessarily involves the connection on $P_S$. When
specialized to the case of a trivial bundle, the EP$\Aut$ dual pair
showed how the singular solutions may support their own magnetic
field, for example in the case of current sheets. This is a very
suggestive picture that arises from the principal bundle structure
of the embedded subspace.

The fourth section considered the case of an incompressible fluid
flow. In this framework, one considers the group of
volume-preserving automorphisms $\Aut_{\rm vol}(P)$, that is bundle
automorphisms that project down to diffeomorphisms
$\operatorname{Diff}_{\rm vol}(M)$ on the base $M=P/\mathcal{O}$,
which preserve a fixed volume form. Again, the equations were
presented in both cases of a trivial and non-trivial principal
bundle. This geometric construction was compared with a
similar formulation that also applies to Euler's vorticity equation,
in the context of \cite{MaWe83}. Left- and right-action momentum maps were presented also for
this case, along with explicit formulas. These formulas
require a deep geometric construction that involves Lie group
extensions. For example, the left leg momentum map required the
definition of the chromomorphism group, a Yang-Mills version of the quantomorphism group from
quantization theory. Indeed, new definitions of infinite-dimensional Lie groups arose naturally in this context: the special Hamiltonian automorphisms and two variants of the chromomorphism group.

%\paragraph{Acknowledgments} We are indebted to Darryl Holm for helpful and stimulating discussions.

\appendix

\section{Appendices on EP$\mathcal{A} ut$ flows}

\subsection{Proof of Lemma \ref{lemma_split}\label{lemma1}}

We will use the formula
\begin{align}\label{brac}
[U,V]_L&=\sigma\left([\mathcal{A}(U),\mathcal{A}(V)]+\mathbf{d}^\mathcal{A}(\mathcal{A}(U))V-\mathbf{d}^\mathcal{A}(\mathcal{A}(V))U+\mathcal{B}(U,V)\right)\nonumber\\
&\qquad\qquad +\operatorname{hor}([U,V]_L),
\end{align}
where $[\,,]_L$ denotes the (left) Lie algebra bracket and
$\mathcal{B}$ is the curvature of the connection. We refer to Lemma 5.2 in
\cite{GBRa2008a} for a proof of this formula.

Denoting by $\langle\,,\rangle$ the $L^2$-pairing, for an arbitrary $\mathbf{v}\in\mathfrak{X}(M)$, we have
\begin{align*}
\left\langle \left(\operatorname{Hor}^\mathcal{A}\right)^*\pounds_U\beta,\mathbf{v}\right\rangle&=\left\langle\beta ,[U,\operatorname{Hor}^\mathcal{A}\mathbf{v}]_L\right\rangle\\
&=\left\langle \beta ,\operatorname{hor}([U,\operatorname{Hor}^\mathcal{A}\mathbf{v}]_L)+\sigma(\mathbf{d}\omega\!\cdot\!\operatorname{Hor}^\mathcal{A}\mathbf{v}+\mathcal{B}(\operatorname{Hor}^\mathcal{A}\mathbf{u},\operatorname{Hor}^\mathcal{A}\mathbf{v}))\right\rangle\\
&=\left\langle \beta ,\operatorname{Hor}^\mathcal{A}([\mathbf{u},\mathbf{v}]_L)\right\rangle
+\left\langle\mathbb{J}\circ \beta ,\mathbf{d}\omega\!\cdot\!\operatorname{Hor}^\mathcal{A}\mathbf{v}+\mathcal{A}([\operatorname{Hor}^\mathcal{A}\mathbf{u},\operatorname{Hor}^\mathcal{A}\mathbf{v}]_L)\right\rangle\\
&=\left\langle \pounds_\mathbf{u}\left(\operatorname{Hor}^\mathcal{A}\right)^*\beta ,\mathbf{v})\rangle+\langle\left(\operatorname{Hor}^\mathcal{A}\right)^*\left(\mathbb{J}\circ \beta \!\cdot\!\mathbf{d}\omega\right),\mathbf{v}\right\rangle\\
&\qquad+\left\langle\left(\operatorname{Hor}^\mathcal{A}\right)^*\pounds_{\operatorname{Hor}^\mathcal{A}\mathbf{u}}(\mathbb{J}\circ \beta \!\cdot\!\mathcal{A}),\mathbf{v}\right\rangle,
\end{align*}
thus proving the first formula.

For an arbitrary $\theta\in\mathcal{F}_\mathcal{O}(P,\mathfrak{o})$, we have
\begin{align*}
\langle \mathbb{J}\circ\pounds_U\beta,\theta\rangle&=\langle \pounds_U\beta,\sigma(\theta)\rangle=\langle \beta ,[U,\sigma(\theta)]_L\rangle=\langle \beta ,\sigma([\omega,\theta]-\mathbf{d}^\mathcal{A}\theta\cdot U)\rangle\\
&=-\langle \mathbb{J}\circ \beta ,\mathbf{d}\theta \cdot U\rangle =-\langle \mathbb{J}\circ \beta ,\mathbf{d}\theta\!\cdot\!\operatorname{Hor}^\mathcal{A}\mathbf{u}-[\omega,\theta]\rangle\\
&=\langle \pounds_{\operatorname{Hor}^\mathcal{A}\mathbf{u}}(\mathbb{J}\circ \beta )+\operatorname{ad}^*_\omega(\mathbb{J}\circ \beta ),\theta\rangle
\end{align*}
This proves the second formula.$\quad\blacksquare$

\subsection{Proof of Lemma \ref{lemma_Q_KK}\label{lemma3}}

Recall from \S\ref{sec:momaps_and_singsol} that the momentum maps of the EP$\mathcal{A}ut$-equation are defined on the cotangent bundle of the Kaluza-Klein configuration space
\[
Q_{KK}=\{\mathcal{Q}:P_S\rightarrow P\mid\mathcal{Q}\circ\Phi_g=\Phi_g\circ\mathcal{Q}\;\text{and}\;\mathbf{Q}\in\operatorname{Emb}(S,M)\},
\]
where $\mathbf{Q}:S\rightarrow M$ is defined by the condition $\pi\circ\mathcal{Q}=\mathbf{Q}\circ\pi$. Note that $Q_{KK}$ is a natural generalization of the space $\operatorname{Emb}(S,M)$ associated to the EPDiff equations. Another natural generalization would be the space $\operatorname{Emb}_\mathcal{O}(P_S,P)$ of equivariant embeddings $\mathcal{Q}:P_S\rightarrow P$. We now show that these spaces coincide, that is
\[
Q_{KK}=\operatorname{Emb}_\mathcal{O}(P_S,P).
\]

$(1)$ We first prove the inclusion $\operatorname{Emb}_\mathcal{O}(P_S,P)\subset Q_{KK}$.
Given $\mathcal{Q}\in \operatorname{Emb}_\mathcal{O}(P_S,P)$ we consider the induced map $\mathbf{Q}:S\rightarrow M$. Since $\mathcal{Q}$ is smooth, $\mathbf{Q}$ is also smooth and it remains to show that $\mathbf{Q}$ is an embedding. Equivalently, we will show that $\mathbf{Q}$ is an injective immersion and a closed map. To prove that it is injective, choose $s_1,s_2\in S$ such that  $\mathbf{Q}(s_1)=\mathbf{Q}(s_2)$. Let $p_1$ and $p_2$ be such that $\pi(p_i)=s_i, i=1,2$. We have
\[
\pi(\mathcal{Q}(p_1))=\mathbf{Q}(s_1)=\mathbf{Q}(s_2)=\pi(\mathcal{Q}(p_2)),
\]
therefore, there exists $g\in G$ such that $\mathcal{Q}(p_1)=\Phi_g(\mathcal{Q}(p_2))=\mathcal{Q}(\Phi_g(p_2))$ by equivariance. Thus, since $\mathcal{Q}$ is injective, we have $p_1=\Phi_g(p_2)$. This proves that $s_1=s_2$. We now show that $\mathbf{Q}$ is an immersion. Suppose that $v_s\in TS$ is such that $T_s\mathbf{Q}(v_s)=0$. For $v_p\in T_pP_S$ such that $T\pi(v_p)=v_s$, we have
\[
T\pi(T\mathcal{Q}(v_p))=T\mathbf{Q}(T\pi(v_s))=0.
\]
This proves that $T\mathcal{Q}(v_p)$ is a vertical vector. Since $\mathcal{Q}$ is an equivariant embedding, $v_p$ must be vertical. Thus, we obtain $v_s=T\pi(v_p)=0$. This proves that $\mathbf{Q}$ is an immersion. Let $F\subset S$ be a closed subset. The preimage $\mathcal{F}=\pi^{-1}(F)\subset P_S$ is also closed since $\pi$ is continuous. We can write
\[
\mathbf{Q}(F)=\mathbf{Q}(\pi(\mathcal{F}))=\pi(\mathcal{Q}(\mathcal{F})),
\]
and this space is closed since $\mathcal{F}$ is closed, $\mathcal{Q}$ is an embedding, and $\pi$ is a quotient map.

\medskip

$(2)$ We now show the inclusion $Q_{KK}\subset \operatorname{Emb}_\mathcal{O}(P_S,P)$. We first verify that $\mathcal{Q}\in Q_{KK}$ is injective. Given $p_1,p_2\in P_S$, we have
\[
\mathcal{Q}(p_1)=\mathcal{Q}(p_2)\Rightarrow \mathbf{Q}(\pi(p_1))=\mathbf{Q}(\pi(p_2))\Rightarrow\pi(p_1)=\pi(p_2)
\]
by injectivity of $\mathbf{Q}$. Therefore, there exists $g\in G$ such that $p_2=\Phi_g(p_1)$ and we have $\mathcal{Q}(p_1)=\mathcal{Q}(p_2)=\mathcal{Q}(\Phi_g(p_1))=\Phi_g(\mathcal{Q}(p_1))$. By the freeness of the action, we have $g=e$. This proves that $p_1=p_2$ and that $\mathcal{Q}$ is injective. We now prove that $\mathcal{Q}$ is an immersion. Let $v_p\in TP_S$ such that $T\mathcal{Q}(v_p)=0$. The equality
\[
T\mathbf{Q}(T\pi(v_p))=T\pi(T\mathcal{Q}(v_p))=0,
\]
shows that the vector $v_p$ is vertical, since $\mathbf{Q}$ is an immersion. Thus, we can write $v_p=\xi_{P_S}(p)$ for a Lie algebra element $\xi\in\mathfrak{o}$. By equivariance, we get
\[
0=T\mathcal{Q}(\xi_{P_S}(p))=\xi_{P}(\mathcal{Q}(p)).
\]
This proves that $\xi=0$ and that $v_p=0$. We now show that $\mathcal{Q}$ is a closed map. Let $\mathcal{F}\subset P_S$ be a closed subset. In order to show that $\mathcal{Q}(\mathcal{F})$ is closed, we consider a sequence $(q_n)\subset \mathcal{Q}(\mathcal{F})$ converging to $q\in P$ and we show that $q\in\mathcal{Q}(\mathcal{F})$. Define $y_n:=\pi(q_n)\in \mathbf{Q}(F)$, where $F:=\pi(\mathcal{F})$ and $y:=\pi(q)$. The sequence $(y_n)$ converges to $y$ and we have $y\in \mathbf{Q}(F)$ since $\mathbf{Q}(F)$ is closed. Thus, we can consider the sequence $x_n:=\mathbf{Q}^{-1}(y_n)$ converging to $x:=\mathbf{Q}^{-1}(y)$. Let $(p_n)\in \mathcal{F}\subset P_S$ be the sequence determined by the condition $\mathcal{Q}(p_n)=q_n$. Let $s$ be a local section defined on an open subset $U$ containing $x$. Since $\pi(p_n)=x_n$, there exists $g_n\in G$ such that $p_n=\Phi_{g_n}(s(x_n))$, and we have $q_n=\mathcal{Q}(p_n)=\mathcal{Q}(\Phi_{g_n}(s(x_n)))=\Phi_{g_n}(\mathcal{Q}(s(x_n))$. Since the sequences $(q_n)$ and $(\mathcal{Q}(s(x_n)))$ converge, there exists a subsequence $g_{n_k}$ converging to $g\in \mathcal{O}$, by properness of the action. Thus, $(q_{n_k})$ converges to $\Phi_g(\mathcal{Q}(s(x)))=\mathcal{Q}(\Phi_g(s(x)))$, and $\Phi_g(s(x))\in\mathcal{F}$ since $p_{n_k}$ converges to $\Phi_g(s(x))$ and $\mathcal{F}$ is closed. This proves that $q\in\mathcal{Q}(\mathcal{F})$.$\quad\blacksquare$

\subsection{Momentum maps on trivial bundles \label{trivial_case}}

If the principal
bundles are trivial, then the Kaluza-Klein configuration manifold is
\[
Q_{KK}=\operatorname{Emb}_\mathcal{O}(S\times\mathcal{O},M\times\mathcal{O})=\operatorname{Emb}(S,M)\times\mathcal{F}(S,\mathcal{O})
\]
by Lemma \ref{lemma_Q_KK}. Moreover, $\mathcal{Q}\in Q_{KK}$, $\varphi\in\mathcal{A}ut(P)$, and
$\psi\in\mathcal{A}ut(P_S)$ read
\[
\mathcal{Q}(s,g)=(\mathbf{Q}(s),\theta(s)g),\quad\varphi(x,g)=(\eta(x),\chi(x)g),\quad\text{and}\quad \psi(s,g)=(\gamma(s),\beta(s)g).
\]
Thus the left and right compositions $\varphi\circ\mathcal{Q}$ and $\mathcal{Q}\circ\psi$ recover the actions \eqref{left_action_trivial} and \eqref{right_action_trivial}. In the same way $\mathcal{P}_\mathcal{Q}\in T^*_\mathcal{Q}Q_{KK}$ can be written $\mathcal{P}_\mathcal{Q}(s,g)=(\mathbf{P}_{\mathbf{Q} }(s),\kappa_\theta(s)g)$. Thus, $\zeta:=\mathbb{J}\circ \mathcal{P}_\mathcal{Q}\in \mathcal{F}_\mathcal{O}(P_S,\mathfrak{o})^*$ and $\mathbf{P}^\mathcal{A}_{ \mathbf{Q} }:=\left(\operatorname{Hor}^\mathcal{A}_\mathcal{Q}\right)^*\mathcal{P}_\mathcal{Q}$ read
\[
\zeta(s,g)=\operatorname{Ad}_g^*\left((\theta^{-1}\kappa_\theta)(s)\right)\quad\text{and}\quad \mathbf{P}_{ \mathbf{Q} }^\mathcal{A}(s)=\mathbf{P}_{ \mathbf{Q} }(s)-(\kappa_\theta\theta^{-1})(s)\!\cdot\!A(\mathbf{Q}(s)).
\]
Here $A$ denotes the one-form on $M$ induced by the connection $\mathcal{A}$, that is, we have
\[
\mathcal{A}(v_x,\xi_g)=\operatorname{Ad}_{g^{-1}}\left(A(v_x)+\xi_gg^{-1}\right).
\]
Since $\widetilde{\mathcal{Q}}\circ\bar\zeta$ can be identified with
$\kappa_\theta\theta^{-1}$, the left momentum map reads
\begin{align}\label{left_momap_trivial_connectiondependent}
&\mathbf{J}_L\left(\mathcal{P}_\mathcal{Q}\right)\nonumber\\
&=\left(\int_S\left(\mathbf{P}_{ \mathbf{Q} }(s)-\kappa_\theta(s)\theta(s)^{-1}\!\cdot\!A(\mathbf{Q}(s))\right)\delta (x-\mathbf{Q}(s)){\rm d}^ks,\int_S\kappa_\theta(s)\theta(s)^{-1}\delta(x-\mathbf{Q}(s)){\rm d}^ks\right)\nonumber\\
&\qquad\in\mathfrak{X}(M)^*\times\mathcal{F}(M,\mathfrak{o})^*.
\end{align}
We now show that this expression agrees with formula \eqref{left_momap_trivial}. Since the bundle is trivial, a vector field
$U\in\mathfrak{aut}(P)$ is naturally identified with the pair
$(\mathbf{u},{\nu})\in\mathfrak{X}(M)\times\mathcal{F}(M,\mathfrak{o})$.
In this case, the connection dependent isomorphism, $U\mapsto
(\mathcal{A}(U),[U])$ reads simply
$(\mathbf{u},{\nu})\mapsto
(A\!\cdot\!\mathbf{u}+{\nu},\mathbf{u})$. Similarly, the
one-form density $\beta\in\mathfrak{aut}(P)^*$ is identified with the
pair $(\mathbf{m},C)$. In this case, the isomorphism $\beta\mapsto
\left(\left(\operatorname{Hor}^\mathcal{A}\right)^*\beta,\mathbb{J}\circ
\beta\right)$ reads simply $(\mathbf{m},C)\mapsto
(\mathbf{m}-C\!\cdot\!A,C)$. This is exactly the transformation yielding \eqref{left_momap_trivial_connectiondependent}
from \eqref{left_momap_trivial}.

Concerning the right momentum map, using the formulas
\begin{align*}
\mathbf{d}\mathcal{Q}\left(\operatorname{Hor}^{\mathcal{A}_S}_{(s,g)}(v_s)\right)&=\mathbf{d}\mathcal{Q}\left(v_s,-A_S(v_s)g\right)\\
&=\left(\mathbf{d}\mathbf{Q}(v_s),\mathbf{d}\theta(v_s)g-\theta(s)A_S(v_s)g\right)\\
\mathcal{A}\left(\mathbf{d}\mathcal{Q}\left(\operatorname{Hor}^{\mathcal{A}_S}_{(s,g)}(v_s)\right)\right)&=\operatorname{Ad}_{(\theta(s)g)^{-1}}\left(A(\mathbf{d}\mathbf{Q}(v_s))+\left(\mathbf{d}\theta(v_s)g-\theta(s)A_S(v_s)g\right)(\theta(s)g)^{-1}\right)\\
&=\operatorname{Ad}_{g^{-1}}\left(\theta^{-1}A(\mathbf{d}\mathbf{Q}(v_s))\theta+\theta^{-1}\mathbf{d}\theta(v_s)-A_S(v_s)\right),
\end{align*}
and the expression \eqref{intermiediate_computation}, we get
\begin{align*}
\mathbf{J}_R\left(\mathcal{Q},\mathbf{P}^\mathcal{A}_{ \mathbf{Q} },\kappa_\theta\right)
&=\left(\left(\mathbf{P}^\mathcal{A}_{ \mathbf{Q} }+\kappa_\theta\theta^{-1}\!\cdot\!A\circ\mathbf{Q}\right)\!\cdot\!\mathbf{d}\mathbf{Q}+\kappa_\theta\!\cdot\!\left(\mathbf{d}\theta-\theta
A_S\right),\theta^{-1}\kappa_\theta\right)
\\
&
=\left(\mathbf{P}_{ \mathbf{Q} }\!\cdot\!\mathbf{d}\mathbf{Q}+\theta^{-1}\kappa_\theta\!\cdot\!\left(\theta^{-1}\mathbf{d}\theta-
A_S\right),\theta^{-1}\kappa_\theta\right)
\\
&=
\left(\mathbf{P}_{ \mathbf{Q} }\!\cdot\!\mathbf{d}\mathbf{Q}+\kappa_\theta\!\cdot\!\left(\mathbf{d}\theta-
\theta A_S\right),\theta^{-1}\kappa_\theta\right)
.
\end{align*}
This expression recovers \eqref{momap_right_trivial} when $A_S$ is the
trivial connection.

\section{Appendices on incompressible EP$\mathcal{A} ut$ flows}

\subsection{Two dimensional EP$\mathcal{A} ut_{\rm vol}$ equations\label{2d}}

If the manifold $M$ has
dimension two, the volume form $\mu_M$ can be thought of as a
symplectic form, and a vector field $\mathbf{u}$ is divergence free
if and only if it is locally Hamiltonian. If the first cohomology of
the manifold vanishes, $H^1(M)=\{0\}$, then a divergence free vector
field $\mathbf{u}$ is globally Hamiltonian and we can write
$\mathbf{u}=X_\psi$, relative to a \textit{stream function} $\psi$
defined up to an additive constant. The space
$\mathfrak{X}_{\vol}(M)$ can thus be identified with the quotient
space $\mathcal{F}(M)/\mathbb{R}$ and the dual is given by functions
$\varpi$ on $M$ such that $\int_{M}\varpi\mu_M=0$. We denote by
$\mathcal{F}(M)_0$ this space. The duality pairing is given by
\[
\langle [\psi],\varpi\rangle=\int_M\psi\varpi\mu_M,\quad [\psi]\in \mathcal{F}(M)/\mathbb{R}, \; \varpi\in\mathcal{F}(M)_0.
\]
Note that $\varpi\mu_M$ is an exact two-form since its integral over the boundaryless manifold $M$ is zero.
This description of the Lie algebra and its dual is compatible with \eqref{domeg} in the following sense. First, we have the isomorphisms $\mathbf{u}=X_\psi\in\mathfrak{X}_{\vol}(M)\mapsto[\psi]\in\mathcal{F}(M)/\mathbb{R}$, and $\omega=\varpi\mu_M\in \mathbf{d}\Omega^1(M)\mapsto\varpi\in \mathcal{F}(M)_0$. Second, the corresponding duality pairings verify
\[
\langle\varpi\mu_M, X_\psi\rangle=\langle \varpi,[\psi]\rangle.
\]
Consistently, the correspondent functional derivatives of a Lagrangian $l$ are related by the formula
\begin{equation}\label{relation_functional_derivatives}
\frac{\delta l}{\delta \mathbf{u}}=\frac{\delta l}{\delta[\psi]}\mu_M.
\end{equation}
Inserted into \eqref{EPAut_vol_vorticity}, this relation yields the system
\begin{equation}\label{EPAut_vol_2D}
\left\{\begin{array}{l}
\vspace{0.2cm}\displaystyle\frac{\partial}{\partial t}\frac{\delta l}{\delta[\psi]}+\left\{\frac{\delta l}{\delta[\psi]},\psi\right\}+\left\{\frac{\delta l}{\delta{\nu}^i},{\nu}^i\right\}=0\\
\displaystyle\frac{\partial}{\partial t}\frac{\delta l}{\delta{\nu}}+\left\{\frac{\delta l}{\delta{\nu}},\psi\right\}+\operatorname{ad}^*_{{\nu}}\frac{\delta l}{\delta{\nu}}=0,
\end{array}\right.
\end{equation}
where in the last term of the first equation, summation over $i$ is
assumed. At this point, an interesting analogy appears with the
equations of reduced MHD in the so called low $\beta$-limit
\cite{MoHa84,MaMo84,Ho1985}. These equations are Lie-Poisson on the
dual of the semidirect-product Lie algebra $\mathfrak{X}_{\rm
vol}(\mathbb{R}^2)\,\circledS\,\mathcal{F}(\mathbb{R}^2,\RR)$, that
is the Lie algebra of $\Aut_{\rm vol}(\mathbb{R}^2\times
S^1)\simeq\operatorname{Diff}_{\rm
vol}(\mathbb{R}^2)\,\circledS\,\mathcal{F}(\mathbb{R}^2)$. More
particularly, the low $\beta$-limit of the reduced MHD equations
arises as the EP$\Aut_{\rm vol}$  system \eqref{EPAut_vol_2D} on
$\mathbb{R}^2\times S^1$ with the Lagrangian
\[
l(\psi,{\nu})=\frac12\int\!\left(\left|\nabla\psi\right|^2+\left|\nabla{\nu}\right|^2\right)\,{\rm
d}x\,{\rm d}y
\]
Analogous versions involving a non-Abelian group $\mathcal{O}$ are
easily obtained.

\subsection{Group actions for the Marsden-Weinstein dual pair}\label{App:MWDP}
 Let $S$ be a compact manifold with volume form $\mu_S$ and let $(M,\omega)$ be an exact (and hence noncompact) symplectic manifold, with $\omega=-\mathbf{d}\theta$. We endow the manifold $\Emb(S,M)$ of embeddings with the symplectic form
\begin{equation}\label{symplecticform}
\bar\omega(f)(u_f,v_f)=\int_S\omega(f(x))(u_f(x),v_f(x))\mu_S.
\end{equation}
%One observes that these actions preserve the symplectic form $\bar\omega$.

\paragraph{Right momentum map.} The infinitesimal generator
associated to the right action of the group of volume preserving
diffeomorphisms is given by
\[
\mathbf{u}_{\operatorname{Emb}}(f)=Tf\circ \mathbf{u},\quad\mathbf{u}\in\mathfrak{X}_{\vol}(S),\quad f\in\operatorname{Emb}(S,M)
\]
and the associated momentum map \eqref{santiago} reads
\[
\mathbf{J}_R(f)=[f^*\theta]\in\Omega^1(S)/\mathbf{d}\mathcal{F}(S)
\]
We now check that $\mathbf{J}_R$ verifies the momentum map condition. Since the action of the diffeomorphisms of $M$ on $\Emb(S,M)$ is infinitesimally transitive, any tangent vector at $f\in \Emb(S,M)$ is of the form $ \mathbf{v} \circ f$, and it suffices to show the equality
\[
\mathbf{d}\left\langle\mathbf{J}_R,\mathbf{u}\right\rangle(f)(\mathbf{v}\circ f)=\bar\omega\left(\mathbf{u}_{\operatorname{Emb}},\mathbf{v}\circ f\right),\quad\text{for all}\quad f\in\operatorname{Emb}(S,M)\quad\text{and}\quad \mathbf{v}\in \mathfrak{X}(M).
\]
Given a curve $\ph_t\in\Diff(M)$ with $\left.\frac{d}{dt}\right|_{t=0}\ph_t=\mathbf{v}$, we have
\begin{align*}
\mathbf{d}\left\langle\mathbf{J}_R,\mathbf{u}\right\rangle(f)(\mathbf{v}\circ f)&=\left.\frac{d}{dt}\right|_{t=0}\left\langle\mathbf{J}_R(\varphi_t\circ f),\mathbf{u}\right\rangle=\left.\frac{d}{dt}\right|_{t=0}\int_S((\varphi_t\circ f)^*\theta)(\mathbf{u})\mu_S\\
&=\int_S (f^*\pounds_\mathbf{v}\theta)(\mathbf{u})\mu_S=-\int_S\mathbf{i}_\mathbf{u}f^*(\mathbf{i}_\mathbf{v}\omega)\mu_S\\
&=\int_S\omega(Tf\circ\mathbf{u},\mathbf{v}\circ f)\mu_S=(\mathbf{i}_{\mathbf{u}_{\operatorname{Emb}}}\bar\omega)(\mathbf{v}\circ f).
\end{align*}
This proves the desired formula.

\paragraph{Left momentum map and the central extension of $\operatorname{Diff}_{\ham}(M)$.} 
The momentum map one should
obtain from the left action of Hamiltonian diffeomorphisms has the expression \eqref{J_L_MW}
\[
\mathbf{J}_L(f)=\int_S\delta(n-f(x))\mu_S\in\mathcal{F}(M)^*.
\]
Note however that there is an ambiguity in this formula since it confuses the Lie algebra of 
Hamiltonian vector fields with the Lie algebra of Hamiltonian functions. In order to overcome this difficulty, it is necessary to  consider the action naturally induced by the prequantization central extension of the subgroup $\operatorname{Diff}_{\ham}(M)$ of Hamiltonian diffeomorphism. The crucial property for us being that the Lie algebra of this central extension is isomorphic to the space of functions on $M$ endowed with the symplectic Poisson bracket.

\medskip

In the particular case of an exact symplectic form $\omega=-\mathbf{d}\theta$, the prequantization central extension of $\operatorname{Diff}_{\ham}(M)$ is diffeomorphic to the cartesian product $\operatorname{Diff}_{\ham}(M)\x\mathbb{R}$ and is described by a group two cocycle $B$. We denote by $\operatorname{Diff}_{\ham}(M)\x_B\mathbb{R}$ this central extension. The ILM-cocycle $B$ is described in
\cite{IsLoMi2006} and depends on the choice of a point $n_0\in M$. It is given by
\begin{equation}\label{cocycle}
B(\ph_1,\ph_2):=\int_{n_0}^{\ph_2(n_0)}\left(\theta-\ph_1^*\theta\right),\quad \ph_1,\ph_2\in \operatorname{Diff}_{\ham}(M),
\end{equation}
where the integral is taken along a smooth curve connecting the point $n_0$ with the point $\ph_2(n_0)$. 
It is known that the $1$-form $\theta-\ph_1^*\theta$ is exact for any Hamiltonian diffeomorphism $\ph_1$ (\cite{McSa98}),
so the value of this integral does not depend on the choice of such a curve. Moreover  the cohomology class of $B$ is independent of the choice of the point $n_0$ and of the 1-form $\theta$ such that $\omega=-\mathbf{d}\theta$, see Theorem 3.1 in \cite{IsLoMi2006}. As it is well-known (\cite{Ko70}),
the Lie algebra of the prequantization central extension is isomorphic to the space of smooth functions on $M$ endowed with the symplectic Poisson bracket. In the particular case of $\operatorname{Diff}_{\ham}(M)\times_B\mathbb{R}$, the Lie algebra isomorphism being given by
\begin{equation}\label{Lie_algebra_isom}
\mathfrak{X}_{\ham}(M)\times\mathbb{R}\rightarrow\mathcal{F}(M),\quad (X_h,a)\mapsto h+a-h_\theta(n_0),
\end{equation}
where $h_\theta=h-\theta(X_h)$, see \cite{GBTr2011} for a detailed computation and further use of the ILM cocycle.

We consider the left action of $(\ph,\alpha)\in\operatorname{Diff}_{\ham}(M)\times_B\mathbb{R}$ on $\operatorname{Emb}(S,M)$ given by
$(\ph,\alpha)\cdot f:=\ph\circ f$.
The infinitesimal generator associated to the Lie algebra element $h\in\mathcal{F}(M)$ reads
\[
h_{\operatorname{Emb}}(f)=X_h\circ f.
\]
With this geometric setting, we obtain that the momentum mapping associated to the left action of the central extension $\operatorname{Diff}_{\ham}(M)\times_B\mathbb{R}$ on $\operatorname{Emb}(S,M)$ is given by
\[
\mathbf{J}_L:\operatorname{Emb}(S,M)\rightarrow\mathcal{F}(M)^*,\quad \mathbf{J}_L(f)=\int_S\delta(n-f(x))\mu_S,
\]
that is,
\[
\left\langle\mathbf{J}_L(f),h\right\rangle=\int_Sh(f(x))\mu_S.
\]
A direct computation shows that the momentum map condition
\[
\mathbf{d}\left\langle\mathbf{J}_L,h\right\rangle(f)(v_f)
=\int_S(\dd h(v_f(x)))\mu_S=\int_S\om(X_h\o f,v_f)\mu_S
=\mathbf{i}_{h_{\operatorname{Emb}}}\bar\omega(v_f),
\]
is verified for all $v_f\in T_f\operatorname{Emb}(S,M)$.

We recall below from \cite{MaWe83} some
remarkable properties of these momentum maps.

\begin{remark}[Equivariance, Clebsch variables, and Noether theorem]\label{Equiv_Clebsch}\normalfont It is readily seen that the momentum maps are equivariant, since we have
\[
\mathbf{J}_R(f\circ\eta)=\eta^*\mathbf{J}_R(f)\quad\text{and}\quad\mathbf{J}_L(\varphi\circ f)=\varphi_*\mathbf{J}_L(f),
\]
for all $\eta\in\operatorname{Diff}_{\vol}(S)$ and
$\varphi\in\operatorname{Diff}_{\ham}(M)$. Therefore, these maps are
Poisson and provide \textit{Clebsch variables} for the Lie-Poisson equations
on $\mathcal{F}(M)^*$ and $\mathfrak{X}_{\vol}(S)^*$. For
appropriate choices of Hamiltonian, one obtains the Vlasov equation
on $M$  and  Euler's equation on $S$. In particular, $\mathbf{J}_L$
can be interpreted as a singular Klimontovich solution of the Vlasov
equation on $\mathcal{F}(M)$, \cite{HoTr2009}. See \cite{MaWeRaScSp1983} for a geometric treatment of the Vlasov equation in plasma physics.

One also notices that $\mathbf{J}_L$ is invariant under the right action of $\operatorname{Diff}_{\vol}(S)$ and
$\mathbf{J}_R$ is left-invariant under the action of $\operatorname{Diff}_{\ham}(M)\x_B\mathbb{R}$. Therefore, the collective Hamiltonians are invariant and, by Noether's theorem, these momentum maps provide a conservation law for the Clebsch variables.

%In the particular case when $M$ is two dimensional, we recover the Euler equations also from the left leg.
\end{remark}

\begin{remark}[Two dimensional coincidences]{\rm
A very interesting coincidence holds for the case of Euler's
equation in two dimensions. In this case, Euler's equation for the
vorticity $\omega\in\mathfrak{X}_{\rm vol}(\mathbb{R}^2)^*$ has exactly
the same form of a Vlasov equation on $\mathcal{F}(\mathbb{R}^2)^*$
and, upon denoting by $\{\cdot,\cdot\}$ the canonical Poisson
bracket in $M=\mathbb{R}^2$, the right  momentum map becomes
$
\mathbf{J}_R(\boldsymbol{Q},\boldsymbol{P})=\sum_{i=1}^{k}\left\{Q^i,P_ i\right\}$
which coincides with the Clebsch representation of the Vlasov
distribution in $\mathcal{F}(\mathbb{R}^2)^*$. However, in two
dimensions, the left leg momentum map also returns a solution of
Euler's equation. In particular, if $S$ is a point with a certain
weight $w\in\mathbb{R}$ and the fluid moves in $M=\mathbb{R}^2$, then the
left momentum map $\mathbf{J}_L$ yields another Clebsch
representation of the vorticity, which is the well known point
vortex solution on the plane. As we shall see, this coincidence does
not hold for EP$\Aut_{\rm vol}$.
}
\end{remark}

\subsection{Proof of Theorem \ref{rightlegtrivial}\label{proposition1}}

We shall proof this result in the case when $w=1$ in \eqref{PBB2}, since extending to arbitrary values of $w$ is straightforward. We know from the definition of momentum map
that inserting $J_\beta=\langle{\bf J},\beta\rangle$ in the Poisson
bracket  \eqref{PBB2} must return the
infinitesimal generator \eqref{LieAlgAction}, that is
\[
\{F,\left\langle{\bf J}_R,(\mathbf{u},\zeta)\right\rangle\}=\mathbf{d} F\left( (\mathbf{u},\zeta)_{\,\operatorname{Emb}(\mathbb{R}^3,\mathbb{R}^{2k}\times
\mathfrak{o}^*)\times\mathcal{F}(\Bbb{R}^3,\mathcal{O})}\right),\;\text{for all}\;
( \mathbf{u} , \zeta )\in \mathfrak{X}_{\rm
vol}(\mathbb{R}^3)\,\circledS\,\mathcal{F}(\mathbb{R}^3,\mathfrak{o}).
\]
This proof shows that the above relation holds.

One evaluates
\begin{align}
J_{( \mathbf{u} , \zeta )}:=\langle{\bf
J}_R,{( \mathbf{u} , \zeta )}\rangle=\int\Big(\nabla\boldsymbol{Q}\cdot\boldsymbol{P}+\operatorname{Tr}\left(\mu^T\,\nabla\theta\,\theta^{-1}\right)\Big)\cdot\mathbf{u} +\int\operatorname{Tr}\left(\mu^T\operatorname{Ad}_\theta\zeta \right)
\end{align}
Since the $(\boldsymbol{Q},\boldsymbol{P})$-components of the
infinitesimal generator are evidently recovered by a simple verification, we focus on the
$(\mu,\theta)$-components. As a first step, we compute
the $\theta$-component, which arises from the second term in the
bottom line of \eqref{PBB2}. Thus, first we calculate
\begin{align*}
\frac{\delta J_{( \mathbf{u} , \zeta )}}{\delta
\mu }=\mathbf{u} \cdot\nabla\theta\,\theta^{-1}+\operatorname{Ad}_\theta\zeta
\end{align*}
and therefore we obtain
\begin{align*}
\int\left\langle \frac{\delta F}{\delta
\theta},\frac{\delta J_{( \mathbf{u} , \zeta )}}{\delta
\mu}\theta\right\rangle =&\int\left\langle \frac{\delta
F}{\delta
\theta},\left( \mathbf{u} \cdot\nabla\theta\,\theta^{-1}+\operatorname{Ad}_\theta\zeta \right) \theta\right\rangle
= \int\left\langle \frac{\delta F}{\delta
\theta},\left( \mathbf{u} \cdot\nabla\theta+\theta\zeta\right) \right\rangle
\end{align*}
which evidently coincides with the $\theta$-component of the Lie algebra
action in \eqref{LieAlgAction}.

We now focus on the $\mu$-component. As a first step, we compute
\begin{align*}
\int
 \left\langle\mu,\left[\frac{\delta F}{\delta
\mu},\frac{\delta J_{( \mathbf{u} , \zeta )}}{\delta \mu}\right]\right\rangle=&\int
\left\langle\mu,\left[\frac{\delta F}{\delta
\mu},\left(\mathbf{u}\cdot \nabla  \theta\right)\theta^{-1}+\operatorname{Ad}_\theta\zeta
\right]\right\rangle
\\
=& -\int
\left\langle\operatorname{ad}^*_{\left(\mathbf{u}\cdot \nabla  \theta\right)\,\theta^{-1}\,}\mu,\frac{\delta
F}{\delta \mu}\right\rangle-\int
\left\langle\operatorname{ad}^*_{\,\left(\operatorname{Ad}_\theta\zeta \right)}\mu,\frac{\delta
F}{\delta \mu}\right\rangle.
\end{align*}
The $\theta$-variation of $J_{( \mathbf{u} , \zeta )}$ is expressed as
\begin{align*}
\delta
J_{( \mathbf{u} , \zeta )}=&\int\operatorname{Tr}\Big( \mu^T \nabla\delta\theta\theta^{-1}-\mu ^T\nabla\theta\theta^{-1}\delta\theta\theta^{-1}\Big) \cdot \mathbf{u}
+
\int\Big\langle\!\operatorname{Ad}^*_{\theta\!}\big(\operatorname{ad}^*_{\delta\theta\theta^{-1}}\mu\big),\,\zeta \Big\rangle
\\
=&
-\int\operatorname{Tr}\Big(\mathbf{u} \cdot \nabla ( \theta ^{-1}\mu ^T) \delta \theta  \Big)
- \int\operatorname{Tr}\Big( \theta ^{-1} \mu ^T (\mathbf{u} \cdot \nabla  \theta) \theta^{-1}\delta\theta)\\
&
-\int\operatorname{Tr}\Big(\theta ^{-1} \left( \operatorname{ad}^*_{\left(\operatorname{Ad}_\theta\zeta \right)}\mu\right)^T\delta \theta \Big),
\end{align*}
where we have used the general formula \cite{MaRa99}
\[
\delta\big(\!\operatorname{Ad}^*_{\,\theta^{-1}}\mu\big)=\operatorname{Ad}^*_{\,\theta^{-1}}\!\Big(\delta\mu-\operatorname{ad}^*_{\,\theta^{-1}\delta\theta}\mu\Big)
\,.
\]
Therefore
\begin{align*}
\frac{\delta J_{( \mathbf{u} , \zeta )}}{\delta
\theta}=-\mathbf{u} \cdot \nabla \left( \mu (\theta ^{-1} )^T\right) -
(\theta^{-1})^T\left(\pounds_\mathbf{u} \theta^T\right)\mu(\theta^{-1})^T
-
\left(\operatorname{ad}^*_{\left(\operatorname{Ad}_\theta\zeta \right)}\mu\right)(\theta^{-1})^T
\end{align*}
and the last term in \eqref{PBB2} yields
\begin{align*}
&-\int\left\langle\frac{\delta J_{( \mathbf{u} , \zeta )}}{\delta \theta},\frac{\delta
F}{\delta
\mu}\theta\right\rangle=-\int\operatorname{Tr}\!\left(\theta\,\frac{\delta
J_{( \mathbf{u} , \zeta )}}{\delta \theta}^T\frac{\delta F}{\delta \mu}\theta\right)
\\
&\;\;=
\int\operatorname{Tr}\!\left(\left(\theta\pounds_\mathbf{u} \!\!\left(\theta^{-1}\mu^T\right)
+
\mu^T\left(\pounds_\mathbf{u} \theta\right)\theta^{-1}
+
\left(\operatorname{ad}^*_{\left(\operatorname{Ad}_\theta\zeta \right)}\mu\right)^{\!T\!}
\right)\frac{\delta F}{\delta \mu}\right)
\\
&\;\;= \int \operatorname{Tr} \!\left(\left(
\left[\mu^T,\left(\pounds_\mathbf{u} \theta\right)\theta^{-1}\right]
+
\pounds_\mathbf{u} \mu^T
 +
\left(\operatorname{ad}^*_{\left(\operatorname{Ad}_\theta\zeta \right)}\mu\right)^{\!T\!} \right) \frac{\delta F}{\delta \mu} \right)
\\
&\;\;= \int \left\langle\left(
\operatorname{ad}^*_{\left(\left(\pounds_\mathbf{u} \theta\right)\theta^{-1}\right)}
\mu +
\pounds_\mathbf{u} \mu
+
\operatorname{ad}^*_{\left(\operatorname{Ad}_\theta\zeta \right)}\mu\right),\frac{\delta
F}{\delta \mu} \right\rangle.
\end{align*}
In conclusion, we are left with
\begin{align*}
\int
 \left\langle\mu,\left[\frac{\delta F}{\delta
\mu},\frac{\delta J_{( \mathbf{u} , \zeta )}}{\delta
\mu}\right]\right\rangle-\int\left\langle\frac{\delta
J_{( \mathbf{u} , \zeta )}}{\delta \theta},\frac{\delta F}{\delta
\mu}\theta\right\rangle=\int\left\langle\pounds_\mathbf{u} \mu,\frac{\delta
F}{\delta \mu}\right\rangle
\end{align*}
which coincides with the $\mu$-component of the infinitesimal
generator in \eqref{LieAlgAction}.$\quad\blacksquare$

\subsection{Proof of Theorem \ref{kernel}\label{proposition2}}

For each Hamiltonian automorphism $\ph\in \mathcal{A}ut_{\ham}(P)$,
the 1-form $\ph^*\th-\th$ is exact and $\O$-invariant, hence there
exists a function $F_\ph\in\F(P)$ such that $\dd
F_\ph=\ph^*\th-\th$. Let $\Psi_\ph:\mathcal{O}\rightarrow\mathbb{R}$
be the map defined by
\[
\Psi_\ph(g):=F_\ph-F_\ph\o\Ph_g\in\mathbb{R},\quad\text{for
all}\quad g\in\O.
\]
Then $\Psi_\ph$ is independent of the choice of the function $F_\ph$
and is a group homomorphism, that is,
$\Psi_\ph(gh)=\Psi_\ph(g)+\Psi_\ph(h)$. We thus get a map
\[
\Psi:\mathcal{A}ut_{\ham}(P)\to\Hom(\O,\mathbb{R}),\quad
\varphi\mapsto \Psi_\ph.
\]
%where $\Hom(\O,\mathbb{R})$ denotes the group homomorphisms from $\O$ to $\mathbb{R}$.
To show that $\Ps$ is a group homomorphism, we use the identity $\dd F_{\ph\o \ps}=\dd(\ps^*F_\ph+F_\ps)$ and the $\mathcal{O}$-equivariance of $\ps$ to obtain, for all $g\in\O$,
\[
\Psi_{\ph\o \ps}(g)=\ps^*F_\ph-\Ph_g^*\ps^*F_\ph+F_\ps-\Ph_g^*F_\ps
=\ps^*\Psi_\ph(g)+\Psi_\ps(g)=(\Psi_\ph+\Psi_\ps)(g).
\]
This proves that $\Psi$ is a group homomorphism, so its kernel
$\overline{\mathcal{A}ut}_{\ham}(P):=\operatorname{ker}\Psi$ is a
normal subgroup of ${\mathcal{A}ut}_{\ham}(P)$,
by the first homomorphism theorem.

In order to find its Lie algebra, we compute the associated Lie algebra homomorphism
\[
\be:\mathfrak{aut}_{\ham}(P)\to\Hom(\O,\mathbb{R}).
\]
Let $X_h$, $h\in\F(P)$, be an $\O$-equivariant Hamiltonian vector field on $P$ and let $\ph_t$ be a curve of Hamiltonian
automorphisms of $P$ such that
$\left.\frac{d}{dt}\right|_{t=0}\ph_t=X_h$. We know from
\eqref{exact} that there exists a function $F_t:=F_{\varphi_t}$ on
$P$ such that $\ph_t^*\th-\th=\dd F_t$. Then we obtain
\[
\dd\left(\left.\frac{d}{dt}\right|_{t=0}F_t\right)=\pounds_{X_h}\th=\dd\left(\mathbf{i}_{X_h}\th-h\right),
\]
which we use, together with the $\O$-invariance of the function
$\mathbf{i}_{X_h}\th$, to compute
\begin{align*}
\be(X_h)(g)&=\left.\frac{d}{dt}\right|_{t=0}\Psi_{\ph_t}(g)
=\left.\frac{d}{dt}\right|_{t=0}\left(F_t-F_t\o\Ph_g\right)
\\&=\left(\mathbf{i}_{X_h}\th-h\right)-\left(\mathbf{i}_{X_h}\th-h\right)\o\Ph_g
=h\o\Ph_g-h.
\end{align*}
Note that, as it should, the expression
$h\o\Ph_g-h$ depends only on $X_h$ and not on the chosen
Hamiltonian.

In conclusion the homomorphism $\be(X_h)$ measures the lack of
$\O$-invariance of the Hamiltonian function $h$. This means that
$\ker \be=\overline{\mathfrak{aut}}_{\ham}(P)$, the Lie algebra of
Hamiltonian vector fields with invariant Hamiltonian functions. The
integrated version of this ideal is the normal subgroup
$\overline{\mathcal{A}ut}_{\ham}(P)=\ker \Psi$ of $\Aut_{\ham}(P)$.
$\quad\blacksquare$

{\footnotesize

\bibliographystyle{new}

\end{document}